\documentclass{amsart}

\usepackage{amssymb}
\usepackage{amsfonts}
\usepackage{latexsym}
\usepackage{wasysym}
\usepackage{mathrsfs}
\usepackage{accents}
\usepackage{xcolor}
\usepackage{fancyvrb}
\usepackage{relsize}

\usepackage{ifthen}
\usepackage{verbatim}
\usepackage{epsfig}
\usepackage[all]{xy}

\numberwithin{equation}{section}

\newtheorem{thm}{Theorem}
\newtheorem{conj}{Conjecture}
\newtheorem{lem}{Lemma}[section]
\newtheorem{prop}[lem]{Proposition}
\newtheorem{claim}[lem]{Claim}
\newtheorem{cor}[lem]{Corollary}

\theoremstyle{remark}
\newtheorem{remark}[lem]{Remark}

\theoremstyle{definition}
\newtheorem{defn}[lem]{Definition}

\newcommand{\Marking}{marking}
\newcommand{\Markings}{markings}
\newcommand{\Marked}{marked}
\newcommand{\Marker}{marker}
\newcommand{\Markers}{markers}
\newcommand{\Mark}{mark}

\newcommand{\gen}[1]{\langle #1 \rangle}
\newcommand{\Z}{\mathbb{Z}}
\newcommand{\R}{\mathbb{R}}
\newcommand\supt{\operatorname{supt}}
\renewcommand\exp{\operatorname{exp}}
\newcommand{\infl}[2]{#1^{\underaccent{\,\rightharpoondown}{#2}}}

\newcommand\strip{\angle }
\newcommand\Homeo{\operatorname{Homeo}}
\newcommand{\HomeoI}{\Homeo_+(I)}

\newcommand\Bscr{\mathscr{B}}

\newcommand\Jscr{\mathscr{J}}

\newcommand\Pscr{\mathscr{P}}

\newcommand\Rscr{\mathscr{R}}
\newcommand\Sscr{\mathscr{S}}

\newcommand\Scal{\mathcal{S}}

\newcommand\Wcal{\mathcal{W}}
\newcommand\Xcal{\mathcal{X}}
\newcommand\Ycal{\mathcal{Y}}

\newcommand{\cut}{{}^\circ}
\newcommand{\Ffrak}{\mathfrak{F}}
\newcommand{\Sfrak}{\mathfrak{S}}
\newcommand\emb{\hookrightarrow}
\newcommand\PLoI{\mathrm{PL}_+(I)}
\renewcommand\Z{\mathbf{Z}}
\renewcommand\R{\mathbf{R}}
\newcommand\formula[1]{{}^\ulcorner #1 {}^\urcorner}
\newcommand\EA{\operatorname{EA}}

\newcommand\As{\mathsf{A}}
\newcommand\Bs{\mathsf{B}}
\newcommand\Cs{\mathsf{C}}
\newcommand\Ds{\mathsf{D}}
\newcommand\Es{\mathsf{E}}
\newcommand\Fs{\mathsf{F}}
\newcommand\Ks{\mathsf{K}}
\newcommand\Ps{\mathsf{P}}

\newcommand\Rs{\mathsf{R}}
\newcommand\Xs{\mathsf{X}}
\newcommand\Ys{\mathsf{Y}}
\newcommand\Zs{\mathsf{Z}}
\newcommand\zero{\mathsf{0}}
\newcommand\one{\mathsf{Z}}

\newcommand\Bsig{\Bscr}
\newcommand\Rsig{\Rscr}
\newcommand\Ssig{\Sscr}
\newcommand\Psig{\Pscr}
\newcommand\Sgen{\Scal}
\renewcommand\epsilon{\varepsilon}

\newcommand\Ssigp{\Ssig'}
\newcommand{\varbump}[3]
{
\xy
(0,0); (#3,0)**\crv{(#1,#1)&(#2,#1)};
\endxy
}
\newcommand{\varnbump}[3]
{
\xy
(0,0); (#3,0)**\crv{(#1,-#1)&(#2,-#1)};
\endxy
}

\newcommand{\vardotbump}[3]
{
\xy
(0,0); (#3,0)**\crv{~*=<3pt>{.}(#1,#1)&(#2,#1)};
\endxy
}
\newcommand{\varndotbump}[3]
{
\xy
(0,0); (#3,0)**\crv{~*=<3pt>{.}(#1,-#1)&(#2,-#1)};
\endxy
}

\newcommand\mand{\textrm{ and }}

\title[Subgroups of Thompson's group]
{Complexity among the finitely generated subgroups of Thompson's group}



\thanks{
The authors would also like to thank the referee for their very careful and thurough reading of the paper.
This publication is in part a product of a visit of the first and third author to
the \emph{Mathematisches Forschungsinstitut Oberwolfach}, Germany in December 2016
as part of their \emph{Research In Pairs} program.
The third author was partially supported by
NSF grants DMS--1600635 and DMS-1854367.
}

\author[Bleak]{Collin~Bleak}
\author[Brin]{Matthew~G.~Brin}
\author[Moore]{Justin~Tatch~Moore}

\address{
Collin~Bleak \\
School of Mathematics and Statistics \\
University of St. Andrews \\
St. Andrews, Fife KY16 9SS \\
}
\email{{\tt cb211@st-andrews.ac.uk}}

\address{
Matthew~G.~Brin \\
Department of Mathematical Sciences \\
Binghamton University \\
Binghamton, NY 13902-6000 \\
}
\email{{\tt matt@math.binghamton.edu}}

\address{
Justin~Tatch~Moore \\
Department of Mathematics \\
Cornell University \\
Ithaca, NY 14853-4201 \\
}
\email{{\tt justin@math.cornell.edu}}

\begin{document}

\begin{abstract}
We demonstrate the existence of a family of finitely generated
subgroups of Richard Thompson's group \(F\) which is strictly well-ordered
by the embeddability relation of type \(\epsilon_0 +1\).
All except the maximum element of this family (which is \(F\)
itself) are elementary amenable groups.
In fact we also obtain, for each \(\alpha < \epsilon_0\), a finitely generated
elementary amenable subgroup of \(F\) whose EA-class is \(\alpha + 2\).
These groups all have simple, explicit descriptions and can be
viewed as a natural continuation of the progression which starts
with \(\Z + \Z\), \(\Z \wr \Z\), and the Brin-Navas group \(B\). 
We also give an example of a pair of finitely generated elementary
amenable subgroups of \(F\) with the property that neither is
embeddable into the other.
\end{abstract}

\maketitle

\section{Introduction}

\label{intro:sec}

Subgroups of \(PL_+(I)\), the group of order preserving, piecewise
linear self homeomorphisms of the unit interval, have been a source
of groups with interesting properties in which calculations are practical.
There is increasing evidence that all countable, or at
least finitely generated, such subgroups will eventually be
understood.
Among these groups is Richard Thompson's group \(F\).
It is extremely easy for a subgroup of \(PL_+(I)\) to contain an
isomorphic copy of \(F\) as a subgroup \cite{BrinU}.
Thus not containing a subgroup isomorphic to \(F\) (being \emph{\(F\)-less}) is a
severe restriction on subgroups of \(PL_+(I)\).
This has led to the following conjectured dichotomy of Brin and Sapir.

\begin{conj} \label{BrinSapir} \cite{BrinEG} \cite{SapirProblems}
If \(G\) is a subgroup of \(PL_+(I)\), then either \(G\) is
elementarily amenable or else \(G\) contains a copy of \(F\).
\end{conj}

The elementary amenable groups form a class \(EG\) and are those groups
that can be built recursively from finite and abelian groups by a
(possibly transfinite) process using extension and directed union.
The elementary amenability class (\(EA\)-class) of a group \(G\) in
\(EG\) is an ordinal valued measure of the complexity of the recursive
construction of \(G\).
(Details are given in Section \ref{prelim:sec}.)
Thompson's group \(F\) is not elementary amenable; 
it is finitely generated and every nontrivial normal subgroup of
\(F\) contains isomorphic copies of
\(F\) (see \cite{CFP}).
Thus an elementary amenable group must be \(F\)-less. 

Our basic thesis is that Conjecture \ref{BrinSapir} will eventually be a
corollary of a more complete understanding of the partial order
\((\Ffrak, \emb)\) where \(\Ffrak\) is the set of biembeddability classes of
finitely generated subgroups of \(F\) and \(A \emb B\) asserts that members of the class
\(A\) embed into members of the class \(B\). 
While we do not settle Conjecture \ref{BrinSapir}, this paper
explores the universe of \(F\)-less subgroups of \(PL_+(I)\) and
finds a complex collection \(\Sfrak\) of elementary amenable
subgroups of Thompson's group \(F\) itself.
The collection \(\Sfrak\) is likely to play an important role in settling
Conjecture \ref{BrinSapir} and more generally in understanding the class of finitely
generated $F$-less subgroups of $F$.

There are two main features of this paper.
The first is the shift of attention away from the usual ``isomorphism type and containment
relation'' (the Hasse diagram) of subgroups, and toward 
the coarser ``biembeddability class and embeddability relation'' where
two groups are biembeddable if each embeds in the other.
A finer analysis of the
isomorphism types of subgroups of \(F\) does not seem feasible at
this time.

The second feature is the discovery of a rich arithmetic that lives
on \(\Sfrak\) that greatly facilitates transfinite induction and
recursion.  The usual ingredients of transfinite recursion are base,
successor, and limit stage: a base object \(A_0\) must be built, an
object \(A_{\alpha+1}\) must be built from the object \(A_\alpha\),
and for a limit \(\alpha\), an object \(A_\alpha\) must be built
from the objects \(A_\beta\) with \(\beta<\alpha\).
We show that \(\Sfrak\) can be equipped with
arithmetic operations that allow us to
easily build from \(B_\alpha\in \Sfrak\) not only \(B_{\alpha+1}\),
but also
\(B_{\alpha\cdot\omega}\) and even
\(B_{\omega^\alpha}\) with equal ease.  This has two consequences.
First, our groups are remarkably easy to ``write down.'' 
This gives a set of groups that are remarkably simple to describe in spite
of having extremely complex constructions (high \(EA\)-class) as elementary
amenable groups.
Second,
the bulk of the work in the paper is shifted from construction to
analysis.  In fact, it is still a wonder to the authors that these
groups can be analyzed at all.

\subsection{The results}

We now state and discuss our results in somewhat more detail.  We
give indication of the meaning of terminology in what follows;
full definitions are given in Sections \ref{objects:sec} and
\ref{prelim:sec} and as noted.

The complex nature of \((\Ffrak,\emb)\) is demonstrated by our main result:
\begin{figure}
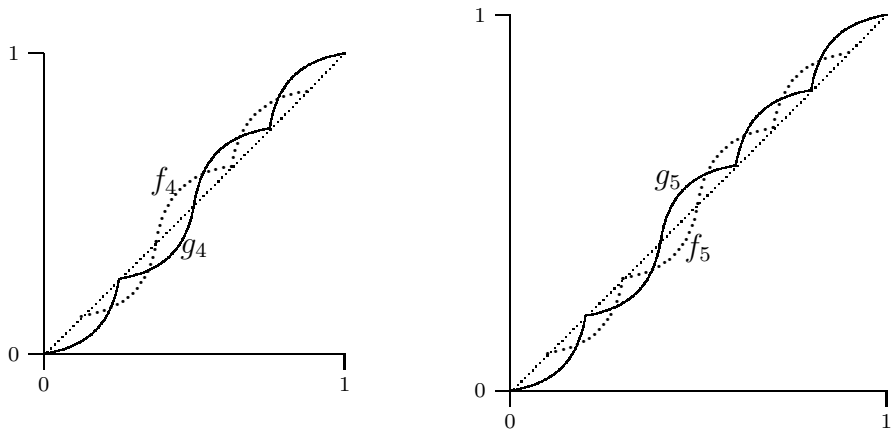

\[
\xy
(45,5); (5,5)**@{-}; (5,45)**@{-};
(5,5); (45,45)**@{.};   
(5,5); (15,15)**\crv{(11,6)&(14,9)};
(15,15); (25,25)**\crv{(21,16)&(24,19)};
(25,25); (35,35)**\crv{(26,31)&(29,34)};
(35,35); (45,45)**\crv{(36,41)&(39,44)};
(25,19)*{g_4};
(10,10); (20,20)**\crv{~*=<3pt>{.}(16,11)&(19,14)};
(20,20); (30,30)**\crv{~*=<3pt>{.}(21,26)&(24,29)};
(30,30); (40,40)**\crv{~*=<3pt>{.}(31,36)&(34,39)};
(21,28)*{f_4};
(45,5); (45,3)**@{-};
(45,1)*{\scriptstyle 1};
(5,45); (3,45)**@{-};
(1,45)*{\scriptstyle 1};
(5,5); (5,3)**@{-};
(5,1)*{\scriptstyle 0};
(5,5); (3,5)**@{-};
(1,5)*{\scriptstyle 0};
\endxy
\qquad\qquad
\xy
(50,0); (0,0)**@{-}; (0,50)**@{-};
(0,0); (50,50)**@{.};   
(0,0); (10,10)**\crv{(6,1)&(9,4)};
(10,10); (20,20)**\crv{(16,11)&(19,14)};
(20,20); (30,30)**\crv{(21,26)&(24,29)};
(30,30); (40,40)**\crv{(31,36)&(34,39)};
(40,40); (50,50)**\crv{(41,46)&(44,49)};
(21,28)*{g_5};
(5,5); (15,15)**\crv{~*=<3pt>{.}(11,6)&(14,9)};
(15,15); (25,25)**\crv{~*=<3pt>{.}(21,16)&(24,19)};
(25,25); (35,35)**\crv{~*=<3pt>{.}(26,31)&(29,34)};
(35,35); (45,45)**\crv{~*=<3pt>{.}(36,41)&(39,44)};
(25,19)*{f_5};
(50,0); (50,-2)**@{-};
(50,-4)*{\scriptstyle 1};
(0,50); (-2,50)**@{-};
(-4,50)*{\scriptstyle 1};
(0,0); (0,-2)**@{-};
(0,-4)*{\scriptstyle 0};
(0,0); (-2,0)**@{-};
(-4,0)*{\scriptstyle 0};
\endxy
\]
\caption{\(G_{\tau_4} := \gen{f_4,g_4}\) and \(G_{\tau_5} :=
\gen{f_5,g_5}\).  The EA-classes of these groups are \(\omega^\omega
+ 2\) and \(\omega^{\omega^\omega}+2\), respectively.}
\label{Gt4Gt5_fig}   \end{figure}

\begin{thm} \label{main_thm}
There is a transfinite sequence \((G_\xi \mid \xi < \epsilon_0)\) of finitely
generated elementary amenable subgroups of \(F\) such that:
\begin{itemize}

\item 
\(G_0\) is the trivial group and \(G_{\xi+1} \cong G_\xi + \Z\);

\item 
\(G_\xi\) embeds into \(G_\eta\) if and only if \(\xi \le \eta\);

\item 
Given \(0 \le \alpha < \epsilon_0\) and \(n<\omega\), let
\(\xi = \omega^{(\omega^\alpha) \cdot (2^n)}\).  If \(\alpha>0\), then
the EA-class of \(G_{\xi}\) is \(\omega \cdot \alpha + n + 2\).
If \(\alpha=0\), then the EA-class of \(G_\xi\) is \(n+1\).

\end{itemize}
\end{thm}
\noindent
In particular, for each \(\alpha < \epsilon_0\), there is a \(\xi\) such that
the EA-class of \(G_\xi\) is \(\alpha + 2\).
(If the EA-class of a finitely generated group is infinite, it is always of the form \(\alpha+2\).)
Thus Theorem \ref{main_thm} improves
previous work of the second author \cite{BrinEG},
who demonstrated that there are finitely generated subgroups of
\(F\) in \(EG\) of class \(\xi + 2\) for each \(\xi < \omega^2\). 
With \(\omega\) the smallest infinite ordinal, the ordinal
\(\epsilon_0\) is the smallest ordinal solution to the equation
\(\omega^{x}= x\). 
If we define a sequence \((\tau_k)_{k\in \omega}\) of
ordinals recursively by \(\tau_0:=2\), \(\tau_1:=\omega\) and
\(\tau_{k+1}:=\omega^{\tau_k}\) for \(k>1\), then \(\epsilon_0\) can
be described as 
\[
\epsilon_0 = \sup\{\tau_k\mid k\in\omega\} =
\omega^{\omega^{\omega^{\omega^{\cdot^{\cdot^\cdot}}}}}.
\]
This countable ordinal is well known to play a central role in proof theory
and in particular in understanding the limitations and consistency
of Peano Arithmetic (see e.g.  \cite[\S D8]{MR0491063} \cite[\#\! 4]{Gentzen_collected} \cite{MR0491157}  \cite{MR663480}). 

The groups in \(\Sfrak:=\{G_\xi \mid \xi < \epsilon_0\}\) are built
from \(\Z\) using certain familiar group-theoretic operations ---
direct sums and wreath products --- as well as a new operation which
is analogous to ordinal exponentiation base \(\omega\).
Whether this new operation
is meaningful in a broader setting
is unclear but even in our rather restrictive setting,
it already yields a wealth of examples. 
The operations also make the construction of the groups in \(\Sfrak\)
straightforward and highly analogous to the construction of ordinals below
\(\epsilon_0\) from \(0\) using exponentiation base \(\omega\) and addition.
Specifically, given the Cantor normal form for an ordinal \(\xi<\epsilon_0\) 
there is an efficient algorithm that lets one write down a finite number of generators
(explicitly as words in the generators of \(F\) if desired)
for a group with EA-class \(\omega \cdot \xi +2\).

While the results of this paper concern groups, the focus of the
analysis is on generating sets.  
The groups in \(\Sfrak\) are specified by a family of generating
sets \(\Sgen\). 
This collection has the property that \(A\) is in \(\Sgen\) if and only if
each of its two element subsets is in \(\Sgen\).
The 2-element sets in \(\Sgen\) generate precisely the groups
\(G_{\tau_k}\) in the family \(\Sfrak = \{G_\xi \mid \xi <
\epsilon_0\}\);
this is the reason for setting \(\tau_0 := 2\).
Theorem \ref{main_thm} implies, in particular, that the
\(G_{\tau_k}\) are an infinite
family of elementary amenable \(2\)-generated subgroups of \(F\)
which are pairwise not biembeddable.
Two of these generating pairs
are illustrated in Figure \ref{Gt4Gt5_fig}.

The isomorphism types of the \(G_{\tau_k}\) are 
parametrized by the nonnegative integer \(k\) which we refer to as 
the \emph{oscillation} of the generating pair from \(\Sgen\).
Figure \ref{Gt4Gt5_fig} illustrates pairs with
oscillation 4 and 5.
The function giving  the oscillations of the pairs from an 
\(A\in\Sgen\) is the \emph{signature} of \(A\). 
Each generating set in \(\Sgen\) is equipped with a total order, and
the signature serves as a complete invariant for all of
\(\Sgen\).

\begin{thm} \label{sig_thm}
If \(A,B \in \Sgen\) have the same
signature, then the order preserving bijection from \(A\) to \(B\)
extends to an isomorphism from \(\gen{A}\) to \(\gen{B}\).
\end{thm}

\noindent
Thus one may analyze \(\Sfrak\) by analyzing the set \(\Ssig\) of all signatures of \(\Sgen\).

The family \(\Sgen\) is 
robust at a group-theoretic level:
if \(A \in \Sgen\), then \(\gen{A}\) is an HNN extension of a group which is itself an increasing union of
subgroups of the form \(\gen{B}\) for \(B \in \Sgen\).
On the other hand, while the closure properties of \(\Sgen\) --- and thus of \(\Ssig\) ---
are important in the group-theoretic analysis of \(\Sfrak\), they introduce redundancies which 
obscure the structure that the embeddability order induces on these classes.
This is resolved by introducing a transitive relation \(\leq\) on \(\Ssig\) as well as algebraic operations \(+\), \(*\), and \(\exp\) and using them to define
a subclass \(\Rsig\) of \(\Ssig\).

\(\Rsig\) provides a
notion of ``normal form'' for \(\Ssig\) and consequently for \(\Sgen\).
If \(A,B \in \Sgen\) have signatures \(\As\) and \(\Bs\), the relation \(\As \leq \Bs\) implies \(\gen{A} \emb \gen{B}\);
the operations \(+\) and \(*\) correspond to the group-theoretic operations of forming direct sums
and wreath products.
Moreover, on \(\Rsig\) the relation \(\leq\) is antisymmetric and
exactly coincides with group-theoretic embeddability;
it is generated from two elements \(\zero\) and \(\one\) using \(+\), \(*\), and \(\exp\).

The next theorem is at the core of the proof of Theorem \ref{main_thm}.
We write \(\As \equiv \Bs\) to denote \(\As \leq \Bs \leq \As\).
While direct sum on groups is commutative, the operation + on \(\Ssig\) is not.
This will require care in certain statements and for that
reason we say that \(\widetilde{\As}\) is a \emph{reordering} of \(\As = \sum_{i < n} \As_i\) if
\(\widetilde{\As} = \sum_{i < n} \As_{\sigma(i)}\) for some permutation \(\sigma\).

\begin{thm} \label{R_thm}
For each \(\As\) in \(\Ssig\) there is a unique \(\Bs\) in \(\Rsig\)
such that \(\Bs \equiv \widetilde{\As}\) for some reordering \(\widetilde{\As}\) of \(\As\).
Moreover there is an order isomorphism \(\rho\) between \((\Rsig,\leq)\) and the ordinals below 
\(\epsilon_0\) which satisfies
\[
\rho(\As + \Bs) = \rho(\As) + \rho(\Bs) \qquad \textrm{ and } \qquad \rho(\exp(\As)) = \omega^{-1 + \rho(\As)}
\]
whenever \(\As,\Bs \in \Rsig\) and \(\As+\Bs \in \Rsig\).
\end{thm}

Thus each biembeddability class in \(\Sfrak\) has a distinguished representative
--- unique up to marked isomorphism --- identified by the form of its signature.
We will show this representative can be built up from \(\Z\) using simple
algebraic operations which are analogs of the fundamental operations of ordinal arithmetic.

A consequence of the proof of Theorem \ref{R_thm} is that
\((\Ssig,\leq)\) is \emph{well-founded} --- 
every nonempty set of signatures has a \(\leq\)-minimal element (see Section \ref{S_WF:subsec}).
This is a subtle matter and likely to be of independent interest.
In fact while it can be phrased in the language of arithmetic, the well-foundedness of \((\Ssig,\leq)\) is not
provable in Peano Arithmetic.
This is a consequence of Theorem \ref{R_thm},
Gentzen's analysis of the consistency of Peano Arithmetic \cite[\#\! 4]{Gentzen_collected}, and
G\"odel's second incompleteness theorem \cite{incompleteness} \cite{Godel_collected1}.
At a more pragmatic level, future methods of proof
may lend themselves more naturally to induction on
\(\Ssig\) than to induction on \(\epsilon_0\).

Extending the chain \(\Sfrak\) by setting \(G_{\epsilon_0} = F\), we
make the following conjectures.    Recall that \(\PLoI\) is the group
of all piecewise linear elements of \(\HomeoI\).

\begin{conj} \label{univ_conj}
If \(H\) is a finitely generated subgroup of \(\PLoI\), then either
\(F\) embeds into \(H\) or else there is an \(\eta < \epsilon_0\)
such that \(H\) embeds into \(G_\eta\).
\end{conj}

\begin{conj} \label{WQO} The partial order \((\Ffrak,\emb)\) is a
\emph{well-quasi-order} --- it contains no infinite decreasing
sequences and no infinite antichains.
\end{conj}

Observe that Conjecture \ref{univ_conj} immediately implies
Conjecture \ref{BrinSapir} since each \(G_\eta\) for \(\eta <
\epsilon_0\) is elementary amenable.  It also implies another
conjecture of the second author which complements Conjecture
\ref{BrinSapir}: every elementary amenable subgroup of \(\PLoI\)
embeds into \(F\).  Moreover, this would imply that \(\epsilon_0\)
is a strict upper bound for the EA-class of every finitely generated
elementary amenable subgroup of \(\PLoI\).

Conjecture \ref{WQO} is really about understanding those finitely
generated subgroups of \(F\) which do not contain \(F\).  The second
author has shown that not containing an isomorphic copy of \(F\) is
a strong restriction on subgroups of \(F\) (and more generally
subgroups of \(\PLoI\)) \cite{BrinU}.  Motivation for Conjecture
\ref{WQO} stems in part from the heuristic that forbidding a
ubiquitous substructure often portends a well developed structure
theory (see, for instance \cite{Fraisse_conj}).

There are limitations, however, as to what one can expect in the direction
of these conjectures.
We show that there is a finitely generated subgroup of \(F\)
which is not biembeddable with any \(G_\xi\) for \(\xi \leq \epsilon_0\).
Moreover, we show that \((\Ffrak,\emb)\) is not a linear ordering.

\begin{thm}\label{AntiChain}
There are finitely generated subgroups \(H_0\) and \(H_1\) of \(F\) of EA-class
\(\omega+2\) such that \(H_0\) does not embed into \(H_1\) and vice versa.
\end{thm}

Lastly we remark, without proof, that the sequence of groups
\(G_{\tau_k}\) with the generating pairs \((f_k, g_k)\) converges in
the space of 2-marked groups to the free group on 2 generators.  We
refer the reader to \cite{MR2653970} for definitions and relevant
arguments.

\subsection{Related results}

All solvable groups are necessarily elementary amenable.  The
solvable subgroups of \(F\) have been thoroughly analyzed by the
first author in \cite{Bpasc}.  In particular, he proves that
Conjecture \ref{univ_conj} holds for the finitely generated solvable
subgroups \(H\) of \(\PLoI\) \cite{Bpasc}: every finitely generated
solvable subgroup of \(\PLoI\) embeds into \(G_{\omega^\omega}\),
which is the Brin-Navas group \(B\) (\cite{BrinEG} Fig. 5, \cite{MR2135961}
Example 6.3).

In \cite{ATaylor}, A.~Taylor finds uncountably many pairwise
nonisomorphic elementary amenable subgroups of \(F\).  These
examples are not finitely generated, are not solvable but are
locally solvable, and all have EA-class \(\omega+1\).

In \cite{MR3079257}, Ol'shanskii and Osin show that for every countable ordinal \(\alpha\),
there is a finitely generated group \(G\in EG\) with \(\EA(G)=\alpha+2\).
The construction in \cite{MR3079257} is recursive and based on HNN extensions.
The operations on \(\Sfrak\) in the current paper  accelerate the recursive process and lead to
easier descriptions of the elements of \(\Sfrak\), even though our construction only carries though \(\epsilon_0\),
not the much larger \(\omega_1\).

\subsection{Organization}

Section \ref{objects:sec} is essentially a continuation of this introduction and it gives
definitions and examples sufficient to introduce the reader to the groups
that we build and how we build them.
It does not give any hints as to their analysis.
Section \ref{prelim:sec} fixes more notation and terminology which
is used in the paper.  It also contains a review of a number of
prerequisites for the paper:
details from \cite{fast_gen};
ordinals and
their arithmetic; elementary amenable groups and EA-class;
wreath products of permutation groups.  The reader may wish to skim or skip
Section \ref{prelim:sec} and then refer back to the various subsections as needed.
In Section \ref{osc:sec}, the oscillation function is developed.
This function is further developed in the context of standard generating sets
in Section \ref{signature:sec}, where we study the signature of an element of \(\Sgen\)
and prove Theorem \ref{sig_thm}.
In Section \ref{inflation:sec} we introduce the notion of an \emph{inflation} of a standard generating
set by one of its elements and show that the result is again a standard generating set.
The interaction of the family \(\Sgen\) with wreath products is detailed
in Section \ref{wreath:sec}.
Section \ref{arithmetic:sec} develops analogs of the operations of ordinal arithmetic
for signatures of elements of \(\Sgen\) and proves Theorem \ref{R_thm}.
This analysis is then used to show that
\((\Sfrak,\emb)\) is a well-order with ordertype \(\epsilon_0\) in
Section \ref{main:sec}, where the proof of Theorem \ref{main_thm}
is completed.
Finally, Section \ref{nonlinear:sec}
contains a proof of Theorem \ref{AntiChain}, 
establishing that \((\Ffrak,\emb)\) is not a linear order.

\section{The objects of study}
\label{objects:sec}

In this section we describe the generating sets in \(\Sgen\)
and the functions that are elements of such sets.  We also
describe the signature associated to a generating set.  The class of
signatures \(\Rsig\) of Theorem \ref{R_thm} is described as
well as the isomorphism from \(\Rsig\) to \(\epsilon_0\).
An aim of this section is to indicate how, given a suitable EA class
\(\alpha<\epsilon_0\), one can write down a set of generators of a
group with EA class \(\alpha\).

\subsection{The anatomy of a homeomorphism}\label{anat:sec}

To describe the elements of \(\HomeoI\) that we work with, we set
some terminology.  We write homeomorphisms to the right of their
arguments and compose from left-to-right.  The \emph{support}
\(\supt(f)\) of \(f\in \HomeoI\) is the set \(\{t\in [0,1]\mid tf\ne
t\}\), and the \emph{e-support} of \(f\) is the interior
of the closure of \(\supt(f)\).

For \(f\in \HomeoI\), a component \(J\) of \(\supt(f)\) is an
\emph{orbital} of \(f\), and a function with exactly one orbital
is called a \emph{bump}.  The \emph{transition points} of
\(f\) are the endpoints of the orbitals of \(f\).
A bump \(f\) with
orbital \(J\) is \emph{positive} if \(tf>t\) for one
(equivalently all) \(t\in J\), and negative otherwise.  If \(g\in
\HomeoI\) has multiple orbitals one of which is \(J\), then the bump
\(f\in \HomeoI\) with \(f|_J=g|_J\) is called a \emph{bump of} \(g\).
If \(f \in \HomeoI\) and \(X \subseteq I\) is a union of orbitals
and fixed points of \(f\), then we write \(f |_X\) to denote
the homeomorphism which coincides with \(f\) on \(X\) and which is
the identity outside of \(X\).

Suppose \(f, g\in \HomeoI\) have disjoint sets of transition points.
We write \(f\ll g\) if 
the e-support of \(f\) is to the left in
\([0,1]\) of the e-support of \(g\).
We write \(f\sqsubset g\) if the closure of the e-support of \(f\) is contained in
the e-support of \(g\). 
Define $f < g$ if the greatest transition point of $f$ is less than the greatest transition point of $g$.
Observe that if either $f \ll g$ or $f \sqsubset g$, then $f < g$.

To conserve space in drawing functions, we do not use horizontal and
vertical \(x\)- and \(y\)-axes as in Figure \ref{Gt4Gt5_fig}, but
draw as if the \(x\)- and \(y\)-axes are at 45 degrees and the part
of the line \(y=x\) in the first quadrant stretches horizontally to
the right from the origin.  The axes themselves and the line \(y=x\)
are suppressed, as are intervals of fixed points.  As in the following,
\begin{equation}\label{BumpsPic}
\xy
(48,0); (64,0)**\crv{(52,4)&(60,4)}; (46,0.5)*{f};
\endxy
\qquad
\qquad
\xy
(8,0);  (24,0)*\crv{(12,-4)&(20,-4)};
(40,0)**\crv{(28,4)&(36,4)};
(56,0)**\crv{(44,4)&(52,4)};
(6,0.5)*{g};
\endxy
\end{equation}
positive bumps become arcs above the horizontal, and negative bumps
are arcs below the horizontal.  The function \(f\) is a positive
bump, and the function \(g\) has one negative bump and two positive
bumps.

\subsection{The attribute fast}\label{FastAttrSec}

To control the isomorphism type of the
groups that we generate, we need some mild controls on the
dynamics of our generating sets.  We use some concepts from
\cite{fast_gen}.

Suppose that \(A \subseteq\HomeoI\) is a set of bumps.
\begin{defn}
A \emph{\Marking{}} of \(A\) is an assignment \(a \mapsto s_a\) of an
element of the support of \(a\) to each \(a \in A\).  The
\emph{feet} of \(a\) with respect to this \Marking{} are the
intervals \((u,s_a)\) and \([t_a,v)\) where \((u,v)\) is the support
of \(a\) and \(t_a:= s_a a\) if \(a\) is positive and \(t_a:= s_a
a^{-1}\) if \(a\) is negative.
\end{defn}
This is illustrated below for
positive \(a\).
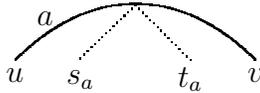
\[
\xy
(0,0); (32,0)**\crv{(10,10)&(22,10)}; (4,5.5)*{a}; 
(8.7, 0); (16,7.3)**@{.}; (23.3,0)**@{.};
(0,-2)*{u}; (8.7, -2.5)*{s_a}; (23.3, -2.5)*{t_a}; (32,-2)*{v};
\endxy
\]

An \(f\in \HomeoI\) equipped with a \Marking{} of its bumps is a
\emph{\Marked{} function}.  The expression ``a \Marked{} function
\(f \in \HomeoI\)'' always means that \(f\) comes with a fixed
\Marking{} even if the \Marking{} is not specified.

The finiteness assumptions in the next definitions are more
restrictive than in \cite{fast_gen}, but are sufficient for us and
add some conveniences.  A collection \(B\) of \Marked{} bumps is
\emph{geometrically fast} (or just \emph{fast}) if for every \(f\ne
g\) in \(B\) the feet of \(f\) are disjoint from the feet of \(g\).
If \(S \subseteq\HomeoI\) is a finite set of \Marked{} functions and
each element of \(S\) has only finitely many bumps, then we define
\(S\) to be \emph{fast} if the set of bumps of \(S\) is fast and no
bump occurs in more than one element of \(S\).  Given a fast \(S\),
we always view \(S\) as being ordered according to the order on
the maximum transition points of its elements.  Note that no two
different elements of such an \(S\) are equal under this order.

The point of these concepts is that the isomorphism type of a
subgroup of \(\HomeoI\) generated by a fast
set \(S\) of functions is determined by a very small amount of
information from \(S\).  Specifically, if \(S'\) is another
fast set and \(h:S\rightarrow S'\) is a bijection that
``preserves enough of the combinatorics'' (including the order given
above, the order of the feet, and the signs of the bumps), then \(h\)
extends to an isomorphism from \(\gen S\) to \(\gen{S'}\). 
A more detailed statement is given in Section \ref{FastGenSec}.

We exploit this in two ways.  We can specify groups up to
isomorphism by giving somewhat sloppy descriptions of the generating
sets.  Further, a result of \cite{fast_gen} says that if \(S\) is
fast, then \(\gen S\) embeds in Thompson's group \(F\).  Thus under
the assumption that the pairs \(\{f_4, g_4\}\) and \(\{f_5, g_5\}\)
in Figure \ref{Gt4Gt5_fig} are fast, we can regard the groups
\(G_{\tau_4}\) and \(G_{\tau_5}\) as subgroups of \(F\) with very
specific isomorphism types.

\subsection{Standard functions and standard pairs}

Thompson's group \(F\) is not elementary amenable, and to build a
subgroup of \(F\) that is elementary amenable one must avoid
including an isomorphic copy of \(F\) in the subgroup.  By the
Ubiquity Theorem \cite{BrinU}, this places severe restrictions on
the generating sets and the functions that can be used in the
generating sets.  The restrictions in the next definition
reflect this.

\begin{defn}
A \emph{standard function} is a \Marked{} function \(f\) in
\(\HomeoI\) with finitely many bumps satisfying all following
properties:
\begin{enumerate}

\item The e-support of \(f\) is an interval;

\item every positive bump of \(f\) is to the right of every negative bump;

\item the number of positive and negative bumps of \(f\) differ by
at most one and there are at least as many positive bumps as
negative bumps.

\end{enumerate}
\end{defn}
The functions \(f\) and \(g\) in (\ref{BumpsPic}) are both standard.
The diagram below illustrates two pairs \((f,g)\) of standard
functions, where the left  pair satisfies
\(f \ll g\) and the right pair satisfies \(f \sqsubset g\).
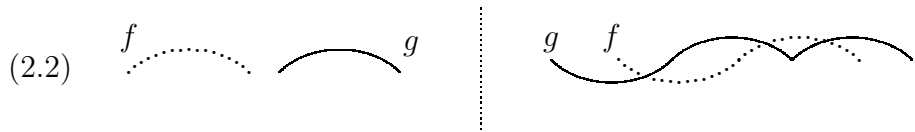
\begin{equation}\label{PairsInOrder}
\xy
(-14,4.0)*{f};
(-6,1.0)*{\vardotbump{4}{12}{16}};
(14, 1.0)*{\varbump{4}{12}{16}};
(23.5,3)*{g};
\endxy
\qquad
\xy
(0.0,8.0); (0.0,-8.0)**@{.};
\endxy
\qquad
\xy
(-8,3.5)*{g};
(0,-.85)*{\varnbump{4}{12}{16}};
(16,2.6)*{\varbump{4}{12}{16}};
(32,2.6)*{\varbump{4}{12}{16}};
(9,-0.8)*{\varndotbump{4}{12}{16}};
(25,2.6)*{\vardotbump{4}{12}{16}};
(0.0,4.0)*{f};
\endxy
\end{equation}

A set \(A\) of standard functions forms an element of \(\Sgen\)
if each pair of elements of \(A\) forms a \emph{standard pair}.
Standard pairs of functions are defined recursively using a
reduction operation that will be heavily used during our analysis in
later sections.  Suppose that \(f\) is a standard function.  If
\(f\) has more than two orbitals, then let \(X\) be the union of the
orbitals of \(f\) except the maximum (rightmost) and minimum
(leftmost) orbitals, and
define \(f \cut\) to be \(f|_X\).
Observe that \(f
\cut\) is again a standard function --- we \Mark{} \(f \cut\) with
the \Markers{} of \(f\) on the orbitals which remain in \(f \cut\).
If \(f\) has one or two orbitals and the left foot of the positive orbital is
\((r,s)\), define \(f\cut\) to be any positive bump with support
\((r,s)\) and an arbitrary \Marker{}.  It will not be necessary
to be  more specific.
It will be important later that in all cases each foot of \(f \cut\) is
a  subset of a foot of \(f\).

\begin{defn}
A pair \((f,g)\) of standard functions is a \emph{standard pair}
if it satisfies the following recursive definition:
\begin{enumerate}

\item The set \(\{f,g\}\) is fast, and

\item either \(f \ll g\), or else \(f\sqsubset g\) and \((g\cut,f)\) is 
a standard pair.

\end{enumerate}
\end{defn}
The following diagram  illustrates the recursive clause:
\[
\xy
(0,3)*{g};
(0,0); (8,0)**\crv{(2,-3)&(6,-3)};
(16,0)**\crv{(10,3)&(14,3)}; (24,0)**\crv{(18,3)&(22,3)};
(4,0); (12,0)**\crv{~*=<3pt>{.}(6,-3)&(10,-3)};
(20,0)**\crv{~*=<3pt>{.}(14,3)&(18,3)};
(4,3.5)*{f};
\endxy
\quad
\leadsto
\quad
\xy
(9,4)*{g\cut};
(8,0); (16,0)**\crv{(10,3)&(14,3)};
(4,0); (12,0)**\crv{~*=<3pt>{.}(6,-3)&(10,-3)};
(20,0)**\crv{~*=<3pt>{.}(14,3)&(18,3)};
(4,3.5)*{f};
\endxy
\quad
\leadsto
\quad
\xy
(13,4)*{g\cut};
(12,0); (24,0)**\crv{(15,3)&(21,3)};
(15.5,0); (20.5,0)**\crv{~*=<3pt>{.}(17.5,1.5)&(18.5,1.5)};
(20.5,-3.5)*{f\cut};
\endxy
\quad
\leadsto
\quad
\xy
(10,3.5)*{{g\cut}\cut};
(7.5,0); (12.5,0)**\crv{(9.5,1.5)&(10.5,1.5)};
(15.5,0); (20.5,0)**\crv{~*=<3pt>{.}(17.5,1.5)&(18.5,1.5)};
(18,4)*{f\cut};
\endxy
\]
It is easily checked that both pairs \((f,g)\) in
(\ref{PairsInOrder}), and the pairs \((f_4,g_4)\) and \((f_5,g_5)\)
in Figure \ref{Gt4Gt5_fig} are standard and that, up to topological
conjugacy, \((g_5 \cut,f_5)\) coincides with \((f_4,g_4)\).

\begin{defn}
A \emph{standard generating set} is a finite set of \Marked{}
functions in \(\HomeoI\) which is pairwise standard.
The collection of all
standard generating sets is denoted \(\Sgen\).
\end{defn}
It is important to note here that any standard generating set is fast.
Also observe that if \(A\) is a standard generating set then for each \(f,g \in A\),
\(f < g\) if and only if either \(f \ll g\) or \(f \sqsubset g\).

\subsection{Oscillation and signature}\label{OscSig:sec}

In our analysis, we reduce the information in a standard generating
set to a finite matrix of integers. 
\begin{defn} \label{OscDef}
If \(f, g \in \HomeoI\) have disjoint sets of transition points and either
$f \ll g$ or $f \sqsubset g$, then we define their \emph{oscillation}
\(o(f,g)\) to be the number of orbitals of \(g\) that contain at
least one transition point of \(f\). 
In order to make this a symmetric function we declare \(o(g,f) = o(f,g)\).
\end{defn}

\begin{defn}
If \(A\) is in \(\Sgen\), then the \emph{signature} of \(A\) is
the function denoted \(\As\) and defined by
\(\As(f,g) = o(f,g)\) whenever \(f < g\) are in \(A\).
We refer to \(A\) as the \emph{base} of the signature \(\As\).
\end{defn}

Figure \ref{osc_fig} shows the signature of a
standard generating set with four elements.
Since elements of \(\Sgen\) may be singletons or empty,
we include the base as part of the data of a signature.

If \(\As\) and \(\Bs\) are signatures with bases \(A\) and \(B\), then we say that \(\As\) and \(\Bs\) are
\emph{equivalent} if \(|A| = |B|\) and the order
preserving bijection \(\theta : A \to B\) 
satisfies \(\As(f,g) = \Bs(\theta(f), \theta(g))\) whenever \(f < g\) are in \(A\).
We also extend this notion of equivalence to when \(\As\) and \(\Bs\) are
just integer functions defined on pairs from ordered sets \(A\) and \(B\);
we use \emph{signature} to refer to a function which is equivalent in this way to a signature 
of an element of \(\Sgen\).
In particular, every signature is uniquely equivalent to a signature with base
\(\{0,\ldots,n-1\}\) for some \(n\).
These canonical representatives allow us to say without guilt
that a signature \(\As\) \emph{is} the signature of some \(S\in \Sgen\) when in reality \(\As\)
is only equivalent to the signature of \(S\).  

Theorem \ref{sig_thm} (proven in Section \ref{signature:sec}) says that if \(A,B \in \Sgen\) 
have the same signature, then the order preserving bijection between \(A\) and \(B\) 
extends to an isomorphism between \(\gen A\) and \(\gen B\).
If \(\As\) is a signature, we define \(\gen{ \As } := \gen{A}\) where \(A \in \Sgen\) has
signature (equivalent to) \(\As\).

\begin{figure}
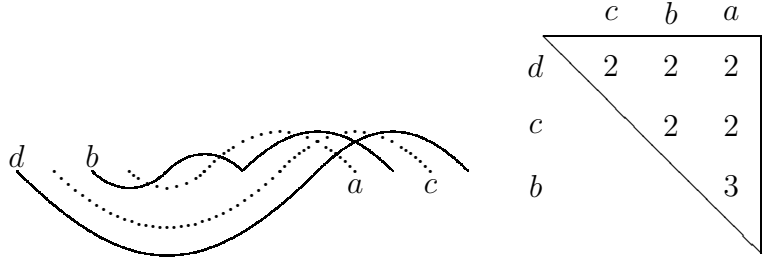

\[
\xy
(0,0); (40,0)**\crv{(15,-15)&(25,-15)};
(40,0); (60,0)**\crv{(47,7)&(53,7)};
(5,0); (35,0)**\crv{~*=<3pt>{.}(15,-10)&(25,-10)};
(35,0); (55,0)**\crv{~*=<3pt>{.}(42,7)&(48,7)};
(10,0); (20,0)**\crv{(13,-3)&(17,-3)};
(20,0); (30,0)**\crv{(23,3)&(27,3)};
(30,0); (50,0)**\crv{(37,7)&(43,7)};
(15,0); (25,0)**\crv{~*=<3pt>{.}(18,-3)&(22,-3)};
(25,0); (45,0)**\crv{~*=<3pt>{.}(32,7)&(38,7)};
(0,2)*{d};
(55,-2)*{c};
(10,2)*{b};
(45,-2)*{a};
(70,18); (99, 18)**@{-}; (99,-11)**@{-}; (70,18)**@{-};
             (79,21)*{c}; (87,21)*{b}; (95,21)*{a};
(69,14)*{d}; (79,14)*{2}; (87,14)*{2}; (95,14)*{2};
(69,6)*{c};              (87,6)*{2}; (95,6)*{2};
(69,-2)*{b}; 	                       (95,-2)*{3};
\endxy
\]
\caption{Four functions that are pairwise standard with various
oscillation numbers.
Their signature is displayed on the right.
They generate a group of 
EA-class \(\omega^2+\omega\cdot 2+2\).}
\label{osc_fig}   \end{figure}

\subsection{Arithmetic on signatures and normal forms}\label{BasicOp:sec}

We now introduce arithmetic operations on the set of
signatures which allow us to readily specify
complex standard generating sets.
The following theorem of Cantor is important motivation.
(Ordinal arithmetic is reviewed in Section \ref{ordinal:sec}.)

\begin{thm}\label{CNF} \cite[\S 19 Theorem B]{MR0045635}
If \(\alpha\) is a nonzero ordinal, there is a unique sequence
\(\beta_0\ge \beta_1 \ge \cdots \ge \beta_n\) such that
\[
\alpha = \sum_{i \leq n} \omega^{\beta_i} =
\omega^{\beta_0} + \omega^{\beta_1} + \ldots + \omega^{\beta_n}.
\]
\end{thm}

If \(0 < \alpha<\epsilon_0\), each positive \(\beta_i\) is less than \(\alpha\)
and may be further expanded in the form
of Theorem \ref{CNF}.
Iterating the process, one obtains an  expression of
\(\alpha\) in terms of \(0\), \(+\) and exponentiation base
\(\omega\) which is known as the \emph{Cantor normal form} of \(\alpha\).
The summands in the Cantor normal form are
\emph{additively indecomposable} in that none is a sum of finitely many smaller ordinals.
It is a standard fact that a positive ordinal is additively indecomposable
precisely if it is of the form $\omega^\beta$ for some ordinal $\beta$, including $\beta=0$.

We will define analogs of these arithmetic operations on the set of signatures.
We start by defining a partial binary relation on \(\Sgen\).
We use \(\zero\) to denote the signature with empty base
and \(\one\) to denote the signature with singleton base;
notice that if \(Z \subseteq\HomeoI\) with \(|Z| = 1\), then \(\gen{Z} \cong \Z\).

\begin{defn}
If \(A,B,C\) are in \(\Sgen\), then \(A = B + C\) means:
\begin{itemize}

\item  \(A = B \cup C\) and \(b < c\) for all \(b \in B\) and \(c \in C\), and

\item  \(o(b,c) = 0\) for all \(b \in B\) and \(c \in C\).

\end{itemize}
\end{defn}

Notice that if \(b < c\) are standard functions, \(o(b,c) = 0\) is equivalent to \(b \ll c\).
Thus if \(A = B + C\), then \(\gen{A}\) is the direct sum of \(\gen{B}\) and \(\gen{C}\).
Observe that while \(A=B+C\) expresses a partial binary operation on \(\Sgen\),
it induces a well defined binary operation \(+\) on the set \(\Ssig\) of signatures (see Lemma \ref{sig_sum_welldef}).
Additionally, we will show in Section \ref{signature:sec}
that if \(\As \in \Ssig\), then the function \(\exp(\As)(a,b) = \As(a,b) + 1\) defined for \(a < b\) in \(A\)
is equivalent to the signature of some standard generating set.

If $\alpha > 0$, define $-1 + \alpha$ to be the unique $\beta$ such that $1+\beta = \alpha$
(note that $-1 + \alpha = \alpha$ if $\alpha$ is infinite).
It will  be useful to adopt the convention that \(\omega^{-1 + 0} = 0\) so that 
\(\zeta \mapsto \omega^{-1 + \zeta}\) describes an increasing function which fixes \(0\) and \(1\).
Theorem \ref{CNF} implies that for each $1 < \alpha < \epsilon_0$ there is a unique sequence
\(\beta_0 \geq \beta_1 \geq \ldots \geq \beta_n \geq 1\) such that
\[
\alpha = \sum_{i \leq n} \omega^{-1 + \beta_i}
\]
Notice in particular that $1 \leq \beta_i < \alpha$ for all $i$.
Iterating this process, one obtains a \emph{revised Cantor normal form} expressing $\alpha$ in terms of $1$ and the operations $+$ and $\zeta \mapsto \omega^{-1 + \zeta}$; $0$ and $1$ are fixed points of
$\zeta \mapsto \omega^{-1 + \zeta}$ and are defined to be
their own revised Cantor normal forms.
Let \(\Rs_\xi \in \Ssig\) denote the result of evaluating these revised Cantor normal forms
in the structure \((\Ssig,\zero,\one,+,\exp)\) in place of \((\epsilon_0,0,1,+,\zeta \mapsto \omega^{-1 + \zeta})\).
The family \(\Rsig\) of \emph{reduced block form signatures} is the collection
\(\{\Rs_\xi \mid \xi < \epsilon_0\}\).

The groups \(G_\xi\) from Theorem \ref{main_thm} are given by \(G_\xi := \gen{\Rs_\xi}\).
To summarize, we will eventually prove Theorems \ref{main_thm} and \ref{R_thm}, showing that:
\begin{itemize}

\item every \(\Sgen\)-generated group is biembeddable with some \(G_\xi\);

\item \(G_\xi\) embeds into \(G_\eta\) if and only if \(\xi < \eta < \epsilon_0\);

\item if \(\xi = \omega^{\omega^\alpha \cdot 2^n}\) for \(0 < \alpha < \epsilon_0\) and \(0 < n < \omega\),
then the EA-class of \(G_\xi\) is 
\(\omega \cdot \alpha + n + 2\).

\end{itemize}

In fact we can take the analogy with ordinal arithmetic further.
\begin{defn}
If \(A,B,C\) are in \(\Sgen\), then \(A = B * C\) means:
\begin{itemize}

\item  \(A = B \cup C\) and \(b < c\) for all \(b \in B\) and \(c \in C\), and

\item  \(o(b,c) = 1\) for all \(b \in B\) and \(c \in C\).

\end{itemize}
\end{defn}
The partial binary operation \(*\) on \(\Sgen\) allows us to define a partial binary operation \(*\)
at the level of signatures as in the case of \(+\).
Unlike \(+\), this operation is only partially defined since if \(A = B * C\), then \(o(c,c') \geq 1\) for all \(c < c'\) in \(C\);
see Section \ref{wreath:sec} for further explanation.
We also show in Section \ref{wreath:sec} that if \(A = B* C\) for nonempty \(B,C\),
then \(\gen{A} \cong \gen{B} \wr \gen{C}\).
We caution that \(\wr\) is a \emph{permutation wreath product}
and typically not the \emph{standard wreath product}
--- see Section \ref{wreath:subsec} for further details, including a precise
definition of which permutation wreath product is being used here.
Also, it can be shown that if \(1 < \beta \leq \alpha < \epsilon_0\) are additively indecomposable ordinals,
then \(\Rs_{\alpha \cdot \beta} \equiv \Rs_\alpha * \Rs_\beta\) with equality holding if \(\Rs_\alpha * \Rs_\beta \in \Rsig\).

Finally, if \(\As \in \Ssig\), define \(E(\As) := \exp(\exp((\As))\).
Unlike \(\exp\), \(E\) also comes from a natural binary relation on standard generating sets:
if \(A,B \in \Sgen\) then \(A = E(B)\) 
asserts that \(o(f,g) \geq 2\) for all \(f < g \in A\) and
\(B = \{f \cut \mid f \in A\}\).
The operations \(+\), \(*\), and \(E\) generate \(\Rsig\) in the following strong sense.

\begin{prop} \label{+*_form}
If \(\As \in \Rsig\), then exactly one of the following hold:
\begin{itemize}

\item \(\As = E(\Bs)\) for some \(\Bs \in \Rsig\);

\item there are \(\Bs,\Cs \ne \zero\) in \(\Rsig\) such that
\(\As = \Bs + \Cs\), in which case \(\gen{\As} = \gen{\Bs} + \gen{\Cs}\);

\item there are \(\Bs,\Cs \ne \zero\) in \(\Rsig\) such that
\(\As = \Bs * \Cs\), in which case \(\gen{\As}\) is the permutation wreath product
\(\gen{\Bs} \wr \gen{\Cs}\)

\end{itemize}
In particular, each \(\As \ne \zero\) in \(\Rsig\) can be uniquely
expressed as a term in the operations \(+\), \(*\), and \(E\) and
the constant \(\one\) so that \(E(\one)\) does not occur as a subterm.
\end{prop}

\noindent
For instance, the example in Figure \ref{osc_fig} is \(E(\one * \one+\one+\one)\).

In this paper, four systems are brought up and related: the ordinals below \(\epsilon_0\) with
their usual order, addition (and its restriction, successor), multiplication,
and exponentiation; standard generating sets with the operation
\(f \mapsto f^\circ\) and operations induced from operations on their signatures;
signatures of standard generating sets and operations induced on these from
\(f \mapsto f^\circ\) as well as operations \(+\), \(*\), and \(\exp\) to parallel those on ordinals;
and finally biembeddability classes of the groups generated by standard generating sets.

The technical details that arise in this paper derive from the multitude of
correspondences that must be established between the four systems.
Some correspondences require care or restriction since a general correspondence cannot work.
Ordinal sums ``correspond'' to direct sums of groups, but the
first is not commutative, while the second is.
Problems are avoided by restricting which direct sums are allowed.

The combinatorial relationships between the ordinals below \(\epsilon_0\), standard generating sets, and signatures
are quite strong; there is a weaker relationship between these three and the standard-generated groups.
For instance while the operations of \(+\) (and successor) and \(*\) carry over with some work
to the setting of groups, there is no group operation that we can envision that
corresponds to exponentiation.
In spite of this weaker relation to groups,
one important relation that can be expressed between the ordinals below \(\epsilon_0\) and
signatures in purely combinatorial terms, with no mention of any group
structures, is only established in this paper through group-theoretic methods. 
We know of no other argument.

The structures and basic properties are developed mostly in Sections \ref{osc:sec}
through \ref{wreath:sec}.
The technicalities of the relationships are mostly developed in
the lengthy Section \ref{arithmetic:sec} and wrapped up in Section \ref{main:sec}.
The ordering on signatures and its relationship to group embeddings is fully developed in Section \ref{main:sec}.
Section \ref{nonlinear:sec} is almost pure group theory.

\section{Intermission}

\label{prelim:sec}

\subsection{Background and conventions}

We adopt the convention that the natural numbers include \(0\)
and in particular all counting starts at \(0\).  We use
\(\Z\) to denote the set of integers.  Unless otherwise stated,
\(i,j,k,l,m,n\) range over the natural numbers and \(p,q,r\)
range over the integers. 
For instance we write
\(i < n\) to mean that \(i\) is a natural number less than \(n\).

If \(G\) and \(H\) are two groups, we use \(G + H\) to denote
their direct sum.
Even though we work primarily with nonabelian groups,
this additive notation and terminology fits better with the
correspondence we develop between 
\(\Sfrak\) and the ordinals below \(\epsilon_0\).

Given \(f,g\in \HomeoI\), we denote by \(f^g\) the
conjugate \(g^{-1}fg\) and remark that under this notation
\(\supt(f^g)=\supt(f)g\).
Given
\(X\subseteq I\) we write \(Xf\) for 
\(
\left\{xf\mid x\in X\right\}.
\)

In this paper we think of \(F\) as being a subgroup of
\(\PLoI\), although there is typically no need to fix a
specific model.  Representing \(F\) in \(\PLoI\) gives us access to
results about subgroups of \(\PLoI\) (specifically results about
centralizers).  This representation also guarantees that any element
\(f\) of \(F\) has only finitely many components of its support and
hence is a product of the bumps of \(f\).

\subsection{Fast generating sets}\label{FastGenSec}

We now give more details to some of the claims in Section
\ref{FastAttrSec} and more information from
\cite{fast_gen}.
If \(S \subseteq \HomeoI\) is a fast set of
\Marked{} functions, 
then
the \emph{dynamical diagram} of \(S\) is the edge labeled, ordered,
directed graph \(D_S\) defined as follows: \begin{itemize}

\item the vertices of \(D_S\) are the feet of \(S\) with the order
induced from \(I\);

\item the directed edges of \(D_S\) are the bumps \(a\) which occur
in \(S\) with the endpoints of \(a\) being the feet of \(a\);

\item positive bumps are directed to the right and negative bumps
are directed to the left;

\item the label of a directed edge \(a\) in \(D_S\) is the unique 
\(f\in S\) for which \(a\) is a bump of \(f\).

\end{itemize}
Note that such ordered directed graphs do not have multiple edges.
We sometimes refer to the  constituents of a dynamical
diagram --- the vertices, the directed edges, the labels --- in
terms of the objects they are intended to represent --- the feet,
the bumps, the functions in \(S\).

For two fast sets of \Marked{} functions \(S_0\) and \(S_1\), an
isomorphism from \(D_{S_0}\) to \(D_{S_1}\) is an isomorphism
\(\theta:D_{S_0}\rightarrow D_{S_1}\) of ordered directed graphs
such that for bumps \(a\) and \(b\) which occur in \(S_0\), the
labels of \(a\) and \(b\) are equal if and only if the labels of
\(\theta(a)\) and \(\theta(b)\) are equal.  Note that the choice of
\Marking{} affects the dynamical diagram but not its isomorphism type.
Also observe that the isomorphism \(\theta\) induces a well defined
order preserving bijection from \(S_0\) to \(S_1\).  Lastly, there
is at most one isomorphism between any two dynamical diagrams.

If \(A\) is a totally ordered finite subset of a group, then the pair
\((\gen{A},A)\) is a \emph{marked group} that is \emph{marked by}
\(A\).  A \emph{marked isomorphism} \(\theta:(\gen{A},A)\rightarrow
(\gen{B},B)\) between marked groups is an isomorphism from
\(\gen{A}\) to \(\gen{B}\) that restricts to an order preserving
bijection from \(A\) to \(B\).  The next theorem asserts that the
dynamical diagram determines the marked isomorphism type of
\(\gen{S}\) whenever \(S \subseteq \HomeoI\) is a fast set of
\Marked{} functions.  (The two uses of ``marked'' in the previous
sentence are not the same.)

\begin{thm} \cite{fast_gen} \label{comb_to_iso} If \(S_0\) and
\(S_1\) are fast sets of \Marked{} functions in \(\HomeoI\), and the
dynamical diagrams of \(S_0\) and \(S_1\) are isomorphic, then the
order preserving bijection from \(S_0\) to \(S_1\) extends to an
isomorphism \(\gen{S_0} \cong \gen{S_1}\).  \end{thm}

Notice that if \(A\) is a fast set of marked functions and
\(B\) is obtained by replacing some element \(a\) of \(A\) with 
one of its positive powers, then \(B\) is also fast and the dynamical diagram
of \(B\) is isomorphic to that of \(A\).
In particular \(\gen{B} \cong \gen{A}\).
This immediate consequence of Theorem \ref{comb_to_iso} will
implicitly play an important role in our arguments starting in Section \ref{inflation:sec}.

We also need the following proposition which is closely related
to the results of \cite{fast_gen}.

\begin{prop} \label{faithful_prop}
Suppose that \(A\) is a finite geometrically fast set of bumps and \(X
\subseteq I\) intersects each orbital of \(A\).  If \(g \in
\gen{A}\) is not the identity, then there is an \(x \in X \gen{A}\)
such that \(xg \ne x\).
\end{prop}

\begin{proof} Observe that the closure of \(X \gen{A}\) contains
the transition points of \(A\).  The proof of Proposition 4.3 of
\cite{fast_gen} yields a \Marking{} of \(A\) which witnesses that it
is geometrically fast and has the property that every \Marker{} is in
the closure of \(X\gen{A}\); let \(M\) denote the set of these
\Markers{}.  From \cite{fast_gen}, if \(g \in \gen{A}\) is not the
identity, then there is a 
\(t \in M \gen{A} \subseteq \overline{X \gen{A}}\) such that \(tg \ne t\).
Continuity of \(g\) implies that there is an \(x \in X \gen{A}\) such that \(xg \ne x\).
\end{proof}

It is convenient to develop some conventions for drawing dynamical
diagrams.  First, we arrange the vertices horizontally from left to
right in increasing order.  We draw right directed edges as
over-arcs and left directed edges as underarcs, suppressing the
arrows.  If \(f\) is a generator and a right foot \(J\) of \(f\) is
immediately followed by a left foot \(J'\) of \(f\), then the pair
of vertices \(\{J,J'\}\) is contracted to a single vertex when
drawing the diagram.  This has the effect of simplifying the
dynamical diagram visually.  It also has the feature that if \(f \in
S\) has connected e-support, then the edges with label \(f\)
form a connected component of the contracted diagram.  Thus it is
sufficient to label only one bump of each such component.

This graphical representation of \(D_S\) can be derived from the
graphs of the elements of \(S\) drawn as in Section \ref{anat:sec}. 
The drawing 
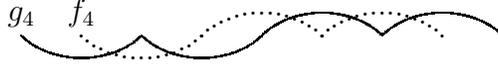
\[
\xy
(0,0); (16,0)**\crv{(4,-4)&(12,-4)}; (32,0)*\crv{(20,-4)&(28,-4)};
(48,0)**\crv{(36,4)&(44,4)}; (64,0)**\crv{(52,4)&(60,4)}; (0,2.5)*{g_4};
(8,0);  (24,0)*\crv{~*=<3pt>{.}(12,-4)&(20,-4)};
(40,0)**\crv{~*=<3pt>{.}(28,4)&(36,4)};
(56,0)**\crv{~*=<3pt>{.}(44,4)&(52,4)};
(8,3)*{f_4};
\endxy
\]
is the graph of \(\{f_4, g_4\}\) from Figure
\ref{Gt4Gt5_fig} as drawn in  Section \ref{anat:sec}
and it is also a drawing of the dynamical diagram for \(\{f_4,
g_4\}\).  In general, the dynamical diagram of a fast set of standard
functions \(S\) is a sketch of the graphs of the functions in \(S\).

Suppose that \(S \subseteq \HomeoI\) is a finite fast set of \Marked{}
functions with connected e-supports.  A bump \(b\)
of \(S\) is \emph{isolated} in \(S\) if its support contains no
transition points of \(S\).
If \(E\) is a set of isolated bumps of \(S\) and for each \(f \in S\) there is a bump
of \(f\) which is not in \(E\), then
we say that \(E\) is an \emph{extraneous set of bumps} of \(S\). 
We need a result of \cite{fast_gen} which
says that extraneous sets of bumps can be \emph{excised} without affecting
the marked isomorphism type of \(S\).
This is made precise as follows.
For \(g \in \HomeoI\) and \(E\) a set of bumps
(not necessarily bumps of \(g\)), we define \(g/E\in\HomeoI\) to be
the function which agrees with \(g\) on
\[
I\setminus\bigcup\{\supt(b) \mid b\in E\ \text{and}\ b\ \text{is a bump of }\
g\} 
\]
and is the identity elsewhere.
We extend the definition above to a set \(S \subseteq \HomeoI\) by:
\[
S/E := \{g/E \mid g \in S\}.
\]
The next theorem is a special case of Theorem \(9.1\) of \cite{fast_gen}.

\begin{thm}\label{excision} \cite{fast_gen} If \(S \subseteq \HomeoI\)
is a fast set of \Marked{} functions with connected e-support 
and \(E\) is an extraneous set of
bumps of \(S\), then the map \(g \mapsto g/E\) extends to an
isomorphism from \(\gen{S}\) to \(\gen{S/E}\).  \end{thm}

\subsection{Ordinals and their arithmetic}\label{ordinal:sec}

Recall that an \emph{ordinal} is the isomorphism type of a well-ordered set.
If \(\alpha\) and \(\beta\) are ordinals, then
\(\alpha < \beta\) is defined to mean that there is well-order of
type \(\alpha\) which is a proper initial part of a well-ordering of
type \(\beta\).  For any two ordinals \(\alpha\) and \(\beta\),
precisely one of the following is true: \(\alpha < \beta\), \(\beta
< \alpha\), or \(\alpha = \beta\).  We adopt von Neumann's convention that an
ordinal is the set of its predecessors and that \(\alpha < \beta\) means \(\alpha \in \beta\).
The least ordinal is \(0 :=\emptyset\) and the least infinite ordinal
is \(\omega\), which can be thought of as coinciding with the
natural numbers.  If \(A\) is a set of ordinals, then there is
always a least ordinal \(\sup(A):= \bigcup A\) which is an upper bound for \(A\). 
If \(\alpha\) is an ordinal, then \(\alpha+1:= \alpha \cup \{\alpha\}\) is
the least ordinal greater than \(\alpha\).  Ordinals of the form
\(\alpha + 1\) are said to be \emph{successor ordinals}; all other
nonzero ordinals are said to be \emph{limit ordinals}.

It is possible to extend the usual arithmetic operations on the finite ordinals to all ordinals as follows:
\[
\alpha + \beta :=
\begin{cases}
\alpha & \textrm{ if } \beta =0 \\
(\alpha + \gamma) + 1 & \textrm{ if } \beta = \gamma + 1 \\
\sup_{\gamma < \beta} (\alpha + \gamma) & \textrm{ if } \beta \textrm{ is limit}
\end{cases}
\]
\[
\alpha \cdot \beta :=
\begin{cases}
0 & \textrm{ if } \beta =0 \\
(\alpha \cdot \gamma) + \alpha & \textrm{ if } \beta = \gamma + 1 \\
\sup_{\gamma < \beta} (\alpha \cdot \gamma) & \textrm{ if } \beta \textrm{ is limit}
\end{cases}
\]
\[
\alpha^\beta :=
\begin{cases}
1 & \textrm{ if } \beta =0  \\ 
(\alpha^\gamma) \cdot \alpha & \textrm{ if } \beta = \gamma + 1 \\
\sup_{\gamma < \beta} (\alpha^\gamma) & \textrm{ if } \beta \textrm{ is limit}
\end{cases}
\]
The reader is cautioned that while \(+\) and \(\cdot\) are associative,
neither \(+\) nor \(\cdot\) are commutative.
For instance: 
\[
2 \cdot \omega = \sup_{n \in \omega}\ ( 2 \cdot n) = \omega < \omega \cdot 2  = \omega+\omega.
\]
Further, ordinal addition is not right cancellative, but is left
cancellative: \(\alpha+\beta = \alpha+\gamma\) implies \(\beta=\gamma\).
We adopt the standard binding conventions from ordinary
arithmetic (e.g.  \(\alpha \cdot \beta + \gamma = (\alpha \cdot
\beta) + \gamma\)) and associate exponentiation to the right (as one
does in ordinary arithmetic): \(\alpha^{\beta^\gamma} =
\alpha^{(\beta^\gamma)}\), which typically does not coincide with
\((\alpha^\beta)^\gamma = \alpha^{\beta \cdot \gamma}\).  

The ordinal \(\epsilon_0\) is the least ordinal solution to \(\omega^x = x\).
It also has the property that if \(\alpha,\beta < \epsilon_0\), then
\(\alpha+\beta\), \(\alpha\cdot \beta\), and \(\alpha^\beta\)
are all less than \(\epsilon_0\).
Further details on
ordinal arithmetic can be found in article II of \cite{MR0045635}.

\subsection{Elementary amenable groups}\label{EAClass:sec}

Consider the smallest class \(EG\) containing the finite and
abelian groups and closed under the following operations:
\begin{enumerate}

\item \label{extension}
taking an extension of one group 
by another group; 

\item \label{directed_union}
taking a directed union of a set of groups; 

\item \label{subgroup}
taking a subgroup of a group; 

\item \label{quotient}
taking a quotient of a group by a normal subgroup;

\end{enumerate}
This class of groups was first considered by Day \cite[P. 520]{MR19:1067c} under the name
\emph{elementary groups};
it is more common in the current literature to refer to them as the \emph{elementary amenable groups}.
This class was later studied by Chou \cite{Chou} who worked out much of the basic theory and showed that the operations
of extension and directed union are all that are needed to generate \(EG\).
Chou stratified \(EG\) by subclasses \(EG_\alpha\) with \(\alpha\) from the ordinals by setting:
\begin{itemize}

\item \(EG_0\) to be the class of all abelian and finite groups;

\item \(EG_{\alpha+1}\) to be those groups obtainable from
groups in \(EG_\alpha\) by a single application of operation (\ref{extension}) or
(\ref{directed_union}) above;

\item
\(EG_\alpha :=\bigcup_{\beta<\alpha}EG_\beta\) if \(\alpha\) is a limit ordinal.

\end{itemize}
It is proven in \cite{Chou} that each \(EG_\alpha\) is closed under
taking subgroups and quotients and that
every element of \(EG\) is in \(EG_\alpha\) for some ordinal \(\alpha\).

For \(G\in EG\), it has become customary to define the
\emph{elementary amenability class} \(\EA(G)\) of \(G\) as the smallest
\(\alpha\) for which \(G\in EG_\alpha\).  It follows from the
definitions that for every limit \(\alpha\) there is no \(G\in EG\)
with \(\EA(G)=\alpha\), and there is no finitely generated \(G\in EG\)
with \(\EA(G)=\alpha+1\).
From Chou's result that the \(EG_\alpha\) are closed under
taking subgroups and quotients, it follows that for \(G\) and \(H\) in
\(EG\), if \(G\) is either a subgroup or a quotient of \(H\), then
\(\EA(G)\le \EA(H)\).

\subsection{Wreath products}
\label{wreath:subsec}

Given a group of permutations \(G\) of \(X\) and a group of
permutations \(H\) of \(Y\), the 1937 article
\cite{polya:wreath} defines \(G\wr H\), the wreath product of \(G\)
and \(H\), as a group of permutations of \(X\times Y\).
Since the 1964 paper \cite{neumann:wreath}, the \emph{standard wreath product}
obtained from \(G\) and \(H\) by regarding \(H\) as
permuting itself by (e.g.) right multiplication has become standard
and ``wreath product'' often means ``standard wreath product'';
``permutation wreath product'' has been used for the older notion.
Our focus is primarily on permutation wreath products in this
article and we proceed with the definition.

Given pairs \((G,X)\) and \((H,Y)\) where \(G\) is a group acting on
\(X\) and \(H\) is a group acting on \(Y\),  we write
\(G\wr H\), the \emph{permutation wreath product}
of \(G\) and \(H\) as shorthand for the pair
\((G\wr H, X\times Y)\) where the group and the action are defined below.
We regard \(G^Y\), the set of functions from \(Y\) to \(G\), as a
group by multiplying pointwise.  With 1 the identity of \(G\) and
for \(\phi\in G^Y\), we use \(\supt(\phi)\) to denote \(\{y\in
Y\mid \phi(y)\ne 1\}\), the \emph{support} of \(\phi\).
We use \(\sum_{Y}G\) to denote the direct sum of copies of \(G\)
indexed over \(Y\) which can be viewed concretely as
the group of finitely supported elements of \(G^Y\).

The group \(H\) also acts on 
\(\sum_Y G\) on the
right by \(\phi^h(y) = \phi(yh^{-1})\).
We use this action to form the semidirect product \(\sum_Y G\rtimes H\)
on the set \((\sum_Y G)\times H\) with multiplication
\((\phi,h)(\theta,j) = (\phi \theta^{h^{-1}}, hj)\).
This semidirect product is the \emph{wreath product}
of \(G\) and \(H\) and is denoted \(G\wr H\).
The action of \(G\wr H\) on \(X\times Y\) is given by
\((x,y)(\phi,h) = (x\phi(y), yh)\).

In our setting, wreath products arise as in the next lemma.
A proof is given in the proof of \cite[Proposition 2.5]{BrinEG}.
The following definitions make the lemma easier to state.
Let \(H\) be a group acting on a set \(A\), let \(Y\) be a subset of \(A\),
and let \(\Ycal = \{Yh\mid h\in H\}\);
note that \(H\) also acts on \(\Ycal\).
We say that the action of \(H\) on \(\Ycal\) is
\emph{consistent} to mean that for all \(h\in H\) if
\(Yh\cap Y\ne\emptyset\), then \(h\) fixes \(Y\) pointwise.
We say that the action of \(H\) on \(\Ycal\) is \emph{faithful}
to mean that the only element of \(H\)
that fixes all elements of \(\Ycal\) is the identity of \(H\).
The lemma is now stated as follows (see Figure \ref{WreathDiag}):

\begin{figure}
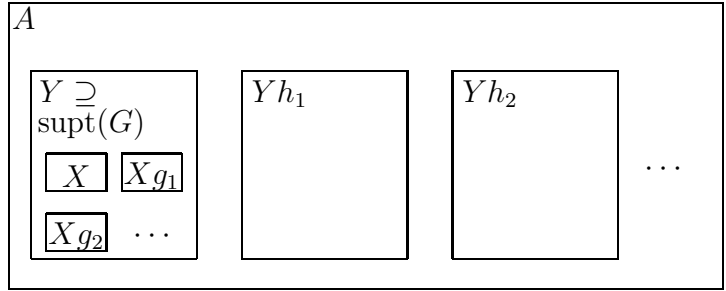

\[
\xy
(11,11); (106,11)**@{-}; (106,49)**@{-}; (11,49)**@{-};
(11,11)**@{-};
(13,47)*{A};
(14,15); (36,15)**@{-}; (36,40)**@{-}; (14,40)**@{-}; (14,15)**@{-};
(19,37)*{Y\supseteq}; (22,33)*{\supt(G)}; 
(16,24); (24,24)**@{-}; (24,29)**@{-}; (16,29)**@{-}; (16,24)**@{-};
(20,26)*{X};
(26,24); (34,24)**@{-}; (34,29)**@{-}; (26,29)**@{-}; (26,24)**@{-};
(30,26)*{Xg_1};
(16,16); (24,16)**@{-}; (24,21)**@{-}; (16,21)**@{-}; (16,16)**@{-};
(20,18)*{Xg_2};  (30,18)*{\cdots};
(42,15); (64,15)**@{-}; (64,40)**@{-}; (42,40)**@{-}; (42,15)**@{-};
(47,37)*{Yh_1}; 
(70,15); (92,15)**@{-}; (92,40)**@{-}; (70,40)**@{-}; (70,15)**@{-};
(75,37)*{Yh_2}; 
(98,27)*{\cdots};
\endxy
\]
\caption{An illustration of the sets in Lemma \ref{PreWrLem}}\label{WreathDiag}
\end{figure}

\begin{lem}\label{PreWrLem}
Suppose that \(G\) and \(H\) act on a set \(A\) on the right.
Assume there are sets \(X \subseteq \supt(G) \subseteq Y \subseteq A\)
such that the action of \(G\) on
\(\Xcal = \{Xg\mid g\in G\}\) and the action of \(H\) on
\(\Ycal = \{Yh\mid h\in H\}\) are both consistent and faithful.
Then the
action of \(W=\langle G, H\rangle\) on \(\Wcal = \{Xw\mid w\in W\}\)
is also consistent and faithful and is isomorphic to the action of
the permutation wreath product \(G\wr H\) on \(\Xcal \times \Ycal\). 
\end{lem}

We now detail how the EA-classes of groups
interact with certain permutation wreath products.
For convenience, we let \(\Sigma G\) denote the direct sum of
countably many copies of the group \(G\).
It is clear that \(\EA(\Sigma G)\ge \EA(G)\).
It is a straightforward inductive exercise to show that for all \(G\),
we have \(\EA(G + G)=\EA(G)\).
It follows that \(\EA(\Sigma G)\le \EA(G)+1\).
In the special case that \(\EA(\Sigma G)=\EA(G)\), we say that \(G\) has
property \(\Sigma\).
Notice, for instance, that every abelian group has \(\Sigma\) and that the groups with \(\Sigma\) are
closed under the elementary operations (\ref{extension})--(\ref{directed_union}).
The family of groups \(\Sfrak\) which we construct all satisfy \(\Sigma\).

The next proposition, mostly proven in \cite{BrinEG}, is very fruitful in
calculating and estimating the EA-classes of the groups we 
consider later in the article.

\begin{prop}\label{EAGWrG}
For an infinite, finitely generated group \(G\in EG\)
acting faithfully on an infinite set \(Y\),
we have
\[
\EA(G)+1 \le \EA(G\wr G) \le \EA(G)+2.
\]
Further, if \(\EA(\Sigma G)=\EA(G)\), then
\(\EA(G\wr G) = \EA(G)+1\).
\end{prop}

\section{The oscillation function}
\label{osc:sec}

In this section, we establish some further facts about the oscillation function first introduced in Section \ref{OscSig:sec}.
If \(A\) is a fast set of
\Marked{} functions, then we say that an orbital is \emph{active
with respect to \(A\)} if it contains a transition point of an
element of \(A\).  Notice that an orbital \(J\) is active with
respect to \(A\) if and only if it contains a foot of \(A\) other
than those of \(J\).

If \(f\) is a standard function with at least
one negative bump, we will call the unique common transition point
of a positive bump and a negative bump of \(f\) the \emph{expansion point} of \(f\).
The following is a sketch of a standard function
where we have highlighted the expansion point with a bullet:
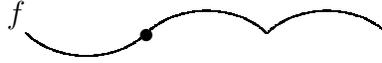
\[
\xy
(-2.3,3.5)*{f};
(15,0.8)*{\bullet};
(7,-0.85)*{\varnbump{4}{12}{16}};
(23,2.6)*{\varbump{4}{12}{16}};
(39,2.6)*{\varbump{4}{12}{16}};
\endxy
\]
The next two lemmas provide a number of useful characterizations of the oscillation function.
We omit the routine proofs.

\begin{lem}\label{Std_Base_Facts} Let \(f, g \in \HomeoI\) be standard
functions so that \(\{f,g\}\) is fast with \(f\sqsubset g\).  Then
the following hold:
\begin{enumerate}

\item The pair \((f,g)\) is standard and \(o(f,g)=1\) if and only if
the support of \(f\) is contained in the rightmost orbital of \(g\).

\item The pair \((f,g)\) is standard and \(o(f,g)=2\) if and only if
\(g\) has an expansion point, and for the rightmost orbital
\((a,b)\) of \(f\), the endpoint \(a\) is contained in the leftmost
orbital of \(g\) and the endpoint \(b\) is contained in the
rightmost orbital of \(g\).

\item If the pair \((f,g)\) is standard and \(o(f,g)>2\), then \(f\)
has an expansion point, and the leftmost orbital of \(g\) and the
rightmost orbital of \(g\) each contain exactly one transition point
of \(f\).

\end{enumerate}
\end{lem}

\begin{lem}  \label{o_lem}
If \(f,g \in \HomeoI\) are marked functions with finitely many bumps and
with disjoint sets of transition points and \(f \sqsubset g\), then
the following are true:
\begin{enumerate}

\item\label{o_sym_fact} \(o(f,g)\) is one greater than the number of
active orbitals of \(f\) with respect to \(\{f,g\}\);

\item\label{o_cut_one} if \((f,g)\) is a standard pair, then
\(o(f,g) = o(g\cut,f)+1\); 

\item \label{o_cover_char} if \(n\) is the cardinality of the
smallest cover of the feet of \(\{f,g\}\) by disjoint intervals
each of which intersects the feet of at most one of \(f\) or \(g\),
then \(o(f,g) = (n-1)/2\);

\end{enumerate}
\end{lem}

Many arguments about \(\Sgen\) are inductive and take advantage of the recursive
nature of the definition of a standard pair.
Note that if \((f,g)\) is standard and \(o(f,g) >1\), then \(f \sqsubset f^g\).
Also observe that even without an assumption that
\(f\sqsubset g\), we have \(o(f,g) = o(f^p,g^q)\) for any nonzero
\(p,q \in \Z\) and that for every homeomorphism \(h\), \(o(f^h,g^h)
= o(f,g)\).  The next lemma is useful in calculating values of
the oscillation function in subsequent sections.

\begin{lem} \label{o_conj_bound}
If \((f,h)\) and \((g,h)\) are standard pairs with \(o(g,h) \ge 1\) and \(\{f,g,h\}\) is fast (possibly \(f=g\)), then the following
are true:
\begin{enumerate}

\item 
either \(f \ll g^h\) or \(f \sqsubset g^h\);

\item \(o(f,g^h) \leq \min\big(o(f,h) , o(g,h)-1\big)\).
\end{enumerate}
\end{lem}

\begin{proof} Observe that since the feet of \(g\) and \(h\) are
disjoint, every foot of \(g^h\) is contained in a foot of \(h\).  By
(\ref{o_cover_char}) of Lemma \ref{o_lem}, \(o(f,g^h) \leq o(f,h)\).
If \(o(g,h)=1\), then Lemma \ref{Std_Base_Facts}
and the assumption that \(\{g,h\}\) is fast implies that the
support of \(g^h\) is contained in the rightmost foot of \(h\) which
is to the right of the support of \(f\).
Thus \(f \ll g^h\) and
\(o(f,g^h) = 0 \le \min\big(o(f,h) , o(g,h)-1\big)\).

Now assume that \(o(g,h)>1\).
Let \(J\) be the union of the feet of \(h\)
other than the leftmost and rightmost feet of \(h\).
Notice that since \(o(g,h) > 1\), \(J\) is nonempty.
If \(f \ll h\), then since \(g \sqsubset h\), it follows that \(f \ll g^h\) and hence
\[o(f,g^h) = 0 \leq \min\big(o(f,h) , o(g,h)-1\big).\]
If \(f \sqsubset h\), then the feet of \(f^{h^{-1}}\) are contained in \(J\)
which in turn is contained in the support of \(g\) since \(o(g,h)>1\).
Therefore \(f^{h^{-1}}\sqsubset g\), \(f \sqsubset g^h\),
and all the feet, and thus all the
transition points, of \(f^{h^{-1}}\) are contained in the orbitals
of \(g\) that are active with respect to \(\{h\cut, g\}\).
By (\ref{o_cut_one}) of Lemma \ref{o_lem}, the number of these
orbitals is \(o(h \cut,g) = o(g,h)-1\), and so
\(o(f,g^h) = o(f^{h^{-1}},g) \le o(g,h)-1\). 
\end{proof}

For the following, recall that if \(A\) is in \(\Sgen\), then
\(A\) is linearly ordered by \(<\).  If \(|A|=n\), we let
\(
a_0<a_1<\cdots <a_{n-1}
\)
be the elements of \(A\).
We also use \(a_{\max}\) to denote the
greatest element of \(A\).

\begin{defn}
If there are no nonempty \(B\) and \(C\) such that
\(A=B+C\), then we say that \(A\) is \emph{indecomposable}.
\end{defn}
Many proofs that follow argue the decomposable and indecomposable cases
separately.  The next lemma gives a useful characterization of when
an element of \(\Sgen\) is decomposable.

\begin{lem} \label{dir_sum_char}
An \(A \in \Sgen\) is decomposable
if and only if there is an \(i < |A|-1\) such that
\(o(a_i,a_{\max}) = 0\).
\end{lem}

\begin{proof}
Since \(A \in \Sgen\), we know that whenever \(b < c\) are in \(A\), \(b \ll c\) is equivalent to \(o(b,c) = 0\).
The forward implication in the lemma follows immediately from this equivalence.
To see the reverse implication, let \(j < |A|-1\) be maximal such that \(o(a_j,a_{\max}) = 0\).
Observe that every element of the support of \(a_j\) is less than every element of the support of \(a_{\max}\).
If \(i < j\), then \(\sup \supt(a_i) < \sup \supt(a_j)\) and consequently \(o(a_i,a_{\max}) = 0\).
We claim that if \(i \leq j < k < |A|-1\), then \(o(a_i,a_k) = 0\).
Since \(o(a_k,a_{\max}) > 0\), it follows that the e-support of \(a_k\) is contained in the e-support
of \(a_{\max}\).
Since \(\sup \supt(a_i) < \inf\supt(a_{\max})\), it follows that \(\sup \supt(a_i) < \inf \supt(a_k)\)
and hence that \(o(a_i,a_k) = 0\).
By the equivalence noted at the start of the proof, we therefore have that
if \(i \leq j < k <|A|\), then \(a_i \ll a_k\). 
If we take \(B = \{a_i \mid i \leq j\}\) and \(C = \{a_k \mid j < k < |A|\}\),
then we've shown that \(A = B+C\). 
\end{proof}

The following operation is useful in analyzing \(\Sgen\).
\begin{defn}
If \(A \in \Sgen\), the \emph{rotation of \(A\)} is defined by
\(A\cut := A \cup \{a_{\max} \cut\} \setminus \{a_{\max}\}\) if there is an \(i < |A|-1\) such
that \(o(a_i,a_{\max}) > 0\);
otherwise set \(A \cut := A \setminus \{a_{\max}\}\).
\end{defn}
Notice that it follows immediately from the definition of standard pair that \(A\cut\) is again in \(\Sgen\).

\begin{lem} \label{rotation_lem}
If \(A \in \Sgen\), then the following are true:
\begin{enumerate}

\item \label{sum_cut}
If \(A = B + C\) and \(C \ne \emptyset\), then \(A \cut = B + (C \cut)\).

\item \label{indecomp_cut}
If \(A\) is indecomposable,
then the least element of \(A\cut\) is \(a_{\max} \cut\).

\end{enumerate}
\end{lem}

\begin{proof}
To see (\ref{sum_cut}), observe first that \(a_{\max} = c_{\max}\).
It remains to show that for all \(i < |B|\), \(b_i \ll c_{\max} \cut\).
Because \(A = B + C\), we have that for all \(i < |B|\), \(b_i \ll c_{\max}\).
Since the support of \(c_{\max} \cut\) is contained in the support of \(c_{\max}\),
it follows that \(b_i \ll c_{\max} \cut\) for every \(i < |B|\).
To see (\ref{indecomp_cut}), observe that 
\(o(a_0,a_{\max}) > 0\) by Lemma \ref{dir_sum_char}.
Thus \((a_{\max}\cut,a_0)\) is standard and hence \(a_{\max}\cut < a_0\).
\end{proof}

The following lemma gives a useful criteria for membership in \(\Sgen\).

\begin{lem} \label{Sgen_char}
If \(A\) is a fast set of standard functions,
then \(A\) is in \(\Sgen\) provided that the following conditions are satisfied:
\begin{itemize}

\item for all \(i < |A| - 1\), either \(a_i \ll a_{\max}\) or \(a_i \sqsubset a_{\max}\);

\item \(A\cut\) is in \(\Sgen\);

\item for all \(i < |A|-1\) if \(a_i \sqsubset a_{\max}\), then \(a_{\max}\cut < a_i\).

\end{itemize}
\end{lem}

\begin{proof} Let \(A\) be a fast set of standard functions.
If \(|A| \leq 1\) there is nothing to show, so assume
that \(|A| > 1\).  Observe that \(A \setminus \{a_{\max}\} \subseteq
A\cut\) is in \(\Sgen\) by assumption and therefore in order to
verify \(A \in \Sgen\), we need only to show that \((a_i,a_{\max})\)
is standard whenever \(i < |A|-1\).  Let \(j < |A|-1\) be minimal
such that \(o(a_j,a_{\max}) > 0\); observe that this
implies \(o(a_i,a_{\max}) > 0\) for all \(j \leq i < |A|-1\).
Since the support of \(a_{\max}
\cut\) is contained in the support of \(a_{\max}\), \(a_i \ll
a_{\max}\) whenever \(i < j\).  If \(j \leq i\), then \(a_{\max}
\cut < a_i\) by hypothesis.  Furthermore, \((a_{\max} \cut,a_i)\) is standard since 
\(A\cut\) is assumed to be in \(\Sgen\). 
Since \(o(a_i, a_{\max})>0\) and thus
\(a_i \sqsubset a_{\max}\), it follows that \((a_i,a_{\max})\) is
standard as desired.  \end{proof}

\section{Signatures}

\label{signature:sec}

In this section we expand on the notion of signature as defined in
Section \ref{OscSig:sec} and prove several properties.
We show that the signature \(\As\) of an \(A \in \Sgen\) completely encodes
the marked isomorphism type of \(\gen{A}\).  Moreover, we give
a simple description of the family of all signatures.
Before
proving the main results, it will be helpful to define some
terminology and prove some lemmas.
\begin{defn}
If \(A \in \Sgen\), the \emph{complexity of \(A\)} is the
pair \[
(|A|,\sum_{i < j} o(a_i,a_j)).\]
The set of all complexities is ordered lexicographically.
\end{defn}
Observe that if \(A \in \Sgen\),
then \(A \cut\) has strictly smaller complexity than \(A\).  In what
follows, it is frequently useful to prove statements about the
elements of \(\Scal\) by induction on their complexity.
Theorem \ref{excision} allows us to remove extraneous bumps without changing the
isomorphism type of the group which is generated.
However if \(A\in \Sgen\) has extraneous bumps,
Theorem \ref{excision} does not ensure the modified generating set remains in \(\Sgen\).
This is addressed by Lemmas \ref{find_extraneous} and \ref{excise_inactive} below.

\begin{lem} \label{no_extraneous_pres}
If \(A \in \Sgen\) and \(E\) is a set of extraneous bump for \(A\cut\), then \(E\) is a set of extraneous bumps of \(A\).
In particular, if \(A \in \Sgen\) and no bump of \(A\) is extraneous in \(A\), then
no bump of \(A\cut\) is extraneous in \(A\cut\).
\end{lem}

\begin{proof}
Let \(E\) be a set of extraneous bumps for \(A\cut\).
Observe that the only situation in which a bump of \(A\cut\) is not a bump of \(A\) occurs when
\(a_{\max}\) has at most two orbitals and the bump in question is \(a_{\max}\cut\).
In this situation \(a_{\max}\cut\) has only one orbital and hence its only bump is not in \(E\).
Second, observe that since no element of \(A\cut\) is comprised only of bumps in \(E\),
no element of \(A\) is comprised only of bumps in \(E\).
Thus it suffices to show that every element of \(E\) is isolated in \(A\).
This follows from the observation that the only transition points of \(A\) which are not
transition points of \(A\cut\) are the greatest and least transition points of \(a_{\max}\),
neither of which are in the support of any bump of \(A\).
\end{proof}

\begin{lem} \label{find_extraneous}
If \(A \in \Sgen\) has an
extraneous bump, then there is a nonempty set \(E\) of extraneous
bumps of \(A\) such that \(A/E\) has the same dynamical diagram as a
member of \(\Sgen\).
\end{lem}

\begin{proof} This is proved by induction on the complexity of
\(A\).  First suppose that \(A = \{f\}\).  Observe that \(f\) must have at least 2 bumps, and has
either the same number of positive and negative bumps or one more
positive bump.  If \(f\) has more positive bumps than negative
bumps, then let \(b\) be the rightmost bump of \(f\) and observe
that \(f/\{b\}\) is still a standard function.  Similarly, if \(f\)
has the same number of positive and negative bumps and \(b\) is the
leftmost bump of \(f\), then \(f/\{b\}\) is a standard function.  In
both cases \(A/\{b\}\) is in \(\Sgen\). 
If \(A = B + C\) for \(B,C \ne \emptyset\), then either \(B\) or \(C\) has an extraneous bump
and we are finished by our induction hypothesis and the observation that
\(A/E = B/E + C/E\).

Suppose that \(|A| > 1\) is indecomposable with an extraneous
bump.  If \(o(a_i, a_{\max})=1\) for all \(i<|A|-1\), then the
support of every \(a_i<a_{\max}\) is in the rightmost bump of
\(a_{\max}\).  If the number of bumps of \(a_{\max}\) is at least 2,
we let \(E\) consist of all bumps of \(a_{\max}\) but the rightmost.
If \(a_{\max}\) has only one bump, then there is a bump extraneous
in \(A'=A\setminus \{a_{\max}\}\).
By induction there is an \(E\) consisting of extraneous bumps of \(A'\) such that \(A'/E\) has the same dynamical diagram as
an element of \(\Sgen\).
Notice that every element of \(E\) is also extraneous in \(A\) and
\(A/E=(A'/E)\cup \{a_{\max}\}\) has the same dynamical diagram as an element of
\(\Sgen\).

If there is an \(i<|A|-1\) with \(o(a_i, a_{\max})>1\), then both of
the outer orbitals of \(a_{\max}\) are active.
If \(a_{\max}\) has only three bumps and the central bump \(b\)
of \(a_{\max}\) is extraneous in \(A\), then \(a_{\max}/b\) consists of a negative bump to the left of
a positive bump with no transition points from \(A \setminus \{a_{\max}\}\) in between these bumps.
Consequently, \(A/\{b\}\) has the same dynamical diagram as an element of \(\Sgen\).

In the remaining cases, one of the following must hold:
\(a_{\max}\) has exactly two bumps,  \(a_{\max}\) has more than three bumps,
or the central bump of \(a_{\max}\) is active.
In each case, \(A \cut\) has
an extraneous bump and we can apply our induction hypothesis to find
a nonempty set of extraneous bumps \(E\) of \(A \cut\) such that \(A \cut / E\) has the same
dynamical diagram as a member of \(\Sgen\).
By Lemma \ref{no_extraneous_pres}, \(E\) is also a set of extraneous bumps of \(A\). 
Finally, it is easily checked that \(A/E\) has the same
dynamical diagram as a member of \(\Sgen\).
\end{proof}

\begin{lem} \label{excise_inactive}
If \(A\) is an element of \(\Sgen\), then there is an \(A' \in
\Sgen\) such that: 
\begin{enumerate}

\item \(\gen{A'}\) is marked isomorphic to \(\gen{A}\);

\item \(A'\) and \(A\) have the same signature;

\item \(A'\) has no extraneous bumps.

\end{enumerate}

\end{lem}

\begin{proof}
Observe that extraneous bumps are not counted by signatures.
In particular, if \(E\) is a set of extraneous bumps of \(A \in \Sgen\),
then the signature of \(A/E\) coincides with the signature of \(A\).
The proof of the lemma is now by induction on the number of extraneous
bumps of \(A\), using Lemma \ref{find_extraneous} and Theorems
\ref{comb_to_iso} and \ref{excision}.
\end{proof}

\subsection{The signature is a complete invariant for \(\Sgen\)-generated groups}

We are now ready to prove Theorem \ref{sig_thm} which asserts that 
if \(A\) and \(B\) are elements of \(\Sgen\) which have the same signature, then
the order preserving bijection between \(A\) and \(B\) extends to an isomorphism \(\gen{A} \cong \gen{B}\).

\begin{proof}[Proof of Theorem \ref{sig_thm}]
The proof is by induction on the complexity of the common signature of \(A\) and \(B\);
let \(n\) denote \(|A| = |B|\).
If \(n=1\), then \(\gen{A} \cong \Z \cong \gen{B}\).
Suppose now that \(n > 1\).
By Lemma \ref{excise_inactive}, we may assume that \(A\) and \(B\) have no extraneous bumps.
By Theorem \ref{comb_to_iso} it suffices to show that \(A\) and \(B\) have isomorphic dynamical diagrams.

If \(A = A' + A''\) for some nonempty \(A'\) and \(A''\),
then \(B = B' + B''\) for some \(B'\) and \(B''\) having the same signatures as \(A'\) and \(A''\) respectively.
Thus we can apply our induction hypothesis to conclude that
the dynamical diagrams of \(A'\) and \(B'\) are isomorphic and similarly for
\(A''\) and \(B''\).
Since the dynamical diagram of \(A\) is obtained by putting the diagram for \(A'\) to the left of the diagram for \(A''\)
--- and similarly for \(B\) --- we have that the dynamical diagrams of \(A\) and \(B\) are isomorphic.

Now suppose that neither \(A\) nor \(B\) decomposes as a sum.  By
Lemma \ref{dir_sum_char}, this means that \(o(a_i,a_{\max}) =
o(b_i,b_{\max}) > 0\) for all \(i < n-1\).  If \(a_{\max}\) has a
single orbital, then \(a_{\max}\) is a positive bump, and
\(o(b_i,b_{\max}) = o(a_i,a_{\max}) = 1\) for all \(i < n-1\).
Notice that the definition of standard pair implies that if \(i <
n-1\), then the support of \(b_i\) is contained in the rightmost
orbital of \(b_{\max}\).  Since no bump of \(b_{\max}\) is
extraneous, this must mean that \(b_{\max}\) has only one orbital
and must be a positive bump.
By our induction hypothesis, \(A \setminus \{a_{\max}\}\) and
\(B \setminus \{b_{\max}\}\) have isomorphic dynamical diagrams; let
\(D\) denote the common isomorphism type.  Notice that the dynamical
diagram for \(A\) and for \(B\) are both obtained by adding a pair of new
vertices to \(D\) --- one to the far left and one to the far right
--- as well as a right directed edge between these new vertices.
This new edge is given a label distinct from the other labels.
Hence \(A\) and \(B\) have isomorphic dynamical diagrams.

Finally, we may now assume that both \(a_{\max}\) and \(b_{\max}\) have more than one orbital.
Observe that \(A\cut\) and \(B\cut\) have the same signature and lower complexity than \(A\) and \(B\):
if \(i < n-1\), then 
\[
o(a_{\max} \cut,a_i) = o(a_i,a_{\max})-1 = o(b_i,b_{\max})-1 = o(b_{\max}\cut,b_i).
\]
By Lemma \ref{no_extraneous_pres}, \(A\cut\) and \(B\cut\) have no
extraneous bumps.  By our induction hypothesis, \(A\cut\) and
\(B\cut\) have dynamical diagrams which are isomorphic to some common
\(D\).  Notice that since every orbital of \(a_{\max}\) is active in
\(A\), \((a_{\max}) \cut\) is an isolated bump in \(A\cut\) if and
only if \(a_{\max}\) has exactly two orbitals.  Since the former
equivalent condition is a property of the dynamical diagram of
\(A\cut\), it follows that \(a_{\max}\) has exactly two orbitals if and
only if \(b_{\max}\) has exactly two orbitals.  It is now easily
checked that in both cases, \(A\) and \(B\) must have isomorphic
dynamical diagrams.  \end{proof}

\subsection{Admissible triples}

We now turn to our characterization of the set of signatures.  We start with the following proposition.
The function \(\varrho\) defined in its proof models the effects on oscillation numbers when
an indecomposable \(A \in \Sgen\) is replaced by \(A\cut\).

\begin{prop} \label{!_equiv}
The following are equivalent for an ordered triple \((p,q,r)\) of integers:
\begin{enumerate}

\item \label{rFocPoint} 
\(r \geq \min (p-1,q)\), with equality holding if \(p \ne q\);

\item \label{qFocPoint}
\(q \geq \min (p,r)\), with equality holding if \(p \ne r+1\);

\item \label{pFocPoint}
\(p \geq \min (q,r+1)\), with equality holding if \(q \ne r\);

\item \label{pqrConjunction} all of the following three inequalities
hold: 
\[
p \geq \min (q,r+1), \quad
q \geq \min (p,r), \quad
r \geq \min (p-1,q).
\]

\end{enumerate}
\end{prop}

\begin{proof}
We first show that (\ref{rFocPoint}) implies (\ref{qFocPoint})
and then argue that (\ref{qFocPoint}) and (\ref{pFocPoint})
are equivalent to (\ref{rFocPoint}) by symmetry.
Let us assume that \(p,q,r\in\Z\) with \(r \geq \min (p-1,q)\) and with equality holding if \(p \ne q\).
If \(p>q\) then we have \(r=q\) so that \(p>q=r\).
Therefore \(q=\min(p,r)\) and (\ref{qFocPoint}) holds.
If \(p<q\) then we have \(r=p-1<p<q\) so in particular \(p=r+1\) and \(q>\min(p,r)=r\) and again (\ref{qFocPoint}) holds.
Finally, if \(p=q\), then \(\min(p,r) \leq p = q\). 
To see that (\ref{qFocPoint}) holds, suppose that \(p \ne r+1\).
By our assumption \(r \geq p-1 = q-1\) and since \(r\) is an integer, \(r \geq p\).
Thus \(\min (p,r) = p = q\) as desired.

Now consider the transformation \(\varrho:\Z^3 \to \Z^3\) defined by \(\varrho(p,q,r) = (r,p-1,q-1)\).
Observe that \((p,q,r)\) satisfies (\ref{rFocPoint}) if and only if \((\bar p, \bar q, \bar r):= \varrho(p,q,r)\) satisfies
(\ref{pFocPoint}):
the assertion
``\(r \geq \min (p-1,q)\), with equality holding if \(p \ne q\)''
is the same as ``\(\bar p \geq \min (\bar q,\bar r + 1)\), with equality holding if \(\bar q \ne \bar r\).''
Similarly, \((p,q,r)\) satisfies (\ref{qFocPoint}) if and only if \(\varrho(p,q,r)\) satisfies (\ref{rFocPoint}).
Similarly, \((p,q,r)\) satisfies (\ref{pFocPoint}) if and only if \(\varrho(p,q,r)\) satisfies (\ref{qFocPoint}).
It follows that (\ref{rFocPoint}), (\ref{qFocPoint}), and (\ref{pFocPoint}) are equivalent.

Next observe that the equivalence of (\ref{rFocPoint}), (\ref{qFocPoint}), and (\ref{pFocPoint}) immediately yields that
each implies (\ref{pqrConjunction}).
Lastly, we assume (\ref{pqrConjunction}) and show that
(\ref{rFocPoint}) holds.
If \(p=q\), then (\ref{rFocPoint}) just asserts \(r \geq \min (p-1,q)\) and there is nothing to show.
If \(p < q\), then since \(p \geq \min (q,r+1)\), it must be that \(q > r+1\) and hence
\(p \geq r+1\).
Since \(r \geq \min (p-1,q) = p-1\) we have \(r = p-1 = \min(p-1,q)\).
Similarly if \(q < p\), then \(q \geq \min (p,r)\) implies that \(q \geq r\).
Taken with \(r \geq \min (p-1,q)\), this implies
\(r = q = \min(p-1,q)\).
\end{proof}

\begin{defn}
\((p,q,r) \in \Z^3\) is an \emph{admissible triple} if it satisfies any of the equivalent
relationships stated in Proposition \ref{!_equiv}.
\end{defn}

Proposition \ref{!_equiv} and the equality $\min (a-1,b-1) = \min (a,b) - 1$ immediately yield
the following corollary.

\begin{cor}\label{!_equivCor} For all \((p,q,r)\in \Z^3\) the following are equivalent:
\begin{enumerate}

\item \((p,q,r)\) is admissible.

\item \((q,r,p-1)\) is admissible.

\item \((r,p-1,q-1)\) is admissible.

\end{enumerate}
\end{cor}

\subsection{Characterizing the set of signatures}

In this section we will define a collection \(\Ssig\) and show that it coincides
with the collection of signatures of elements of \(\Sgen\).
In the process, we will introduce
certain algebraic operations on \(\Ssig\) and develop some basic facts about them.

\begin{defn}
\(\Psig\) is the collection of pairs \((\Ps,P)\)
such that \(P\) is a finite linearly ordered set and \(\Ps\) is a function from the unordered pairs of 
elements of \(P\) into the nonnegative integers.
The set \(P\) is called the \emph{base} of \(\Ps\).
If \(\{a,b\} \in P\) with \(a < b\), we will write \(\Ps(a,b)\) for \(\Ps(\{a,b\})\).
\end{defn}

The next definition and its relationship to \(\Sgen\)
will be central to much of the rest of the paper.

\begin{defn}
\(\Ssig\) is the set of all \(\As \in \Psig\) such that 
whenever \(a < b < c\) are in \(A\), the triple \((\As(b,c),\As(a,c),\As(a,b))\) is admissible.
\end{defn}

Note that vacuously \(\zero,\one \in \Ssig\).
Formally we view \(\Ssig\) as a set by
using the choice of canonical representatives from each equivalence
class noted in Section \ref{OscSig:sec}, although sometimes it will be convenient to work with different representatives.
We often write that a function on pairs is
in \(\Ssig\) when we really mean that its equivalence class is in \(\Ssig\).
Notice that just as \(\Sgen\) was defined as those finite sets of marked functions all of whose pairs
are standard, \(\Ssig\) is defined as those elements of \(\Psig\) all of whose triples are admissible.

Anticipating \(\Ssig\)'s relation to \(\Sgen\), we will often confuse the
distinction between an \(\As \in \Ssig\) and its base, writing things such as
``the cardinality of \(\As\)'' or ``the elements of \(\As\)'' when we are really referring
to the cardinality or elements of the base of \(\As\).

If \(\Bs\) and \(\Cs\) are in \(\Ssig\), define an element \(\Bs+\Cs \in \Ssig\)
with cardinality \(|\Bs| + |\Cs|\) by
\[
(\Bs + \Cs)(i,j) := 
\begin{cases}
\Bs(i,j) & \textrm{ if } i < j < |B| \\
\Cs(i-|B|,j-|B|) & \textrm{ if } |B| \leq i < j < |B| + |C| \\
0 & \textrm{ if } b \in B \textrm{ and } c \in C
\end{cases}
\]
The next lemma formalizes what
was meant in Section \ref{BasicOp:sec} when we said that \(+\) defines an operation
at the level of signatures; the proof is left to the reader.

\begin{lem} \label{sig_sum_welldef}
If \(\Bs\) and \(\Cs\) are in \(\Ssig\), then so is \(\Bs + \Cs\).
Moreover, if \(A,B,C \in \Sgen\) satisfy that \(A = B + C\), then 
the signature of \(A\) is \(\Bs + \Cs\).
\end{lem}

The following property of elements of $\Ssig$ will be used often.

\begin{lem}\label{ZeroToEnd}
Suppose \(\As\in \Ssig\).
If \(i<j<k  < |A|\) and \(\As(j,k)=0\), then
\(\As(i,k)=0\).
\end{lem}

\begin{proof}
If \(\As(i,k)\ne0\), then
\[
\As(i,j) = \min(\As(j,k)-1, \As(i,k)) = -1
\]
which is not possible.
\end{proof}

\begin{defn}
An element \(\As\) of \(\Ssig\) is \emph{indecomposable} if there do
not exist \(\Bs,\Cs \ne \zero\) such 
that \(\As = \Bs + \Cs\).
\end{defn}
We need the following analog of Lemma \ref{dir_sum_char}.

\begin{lem} \label{top_positive}
If \(\As \in \Ssig\) is indecomposable and  \(i < n = |\As|-1\),
then \(\As(i,n) > 0\). 
\end{lem}

\begin{proof}
We prove the contrapositive.
Let \(j\) be maximal such that \(\As(j,n) = 0\).  By Lemma
\ref{ZeroToEnd}, \(\As(i,n)=0\) for all \(i<j\).
If \(i<j<k\le n\), we have \(\As(i,n)=0<\As(k,n)\) so
\[
\As(i,k) = \min(\As(k,n)-1, \As(i,n)) = 0.
\]
Thus setting \(\Bs := \{0,\ldots,j\}\) 
and \(\Cs := \{j+1,\ldots,n\}\) we have \(\As=\Bs+\Cs\).
\end{proof}

We now define the analog of the rotation operation on \(\Ssig\).
Just as in the case of \(\Sgen\), we define the \emph{complexity} of an element \(\As \in \Ssig\)
to be the pair \((|\As|,\sum_{i < j} \As (i,j))\).
Set \(\zero \cut = \zero\) and \(\one \cut = \zero\).
If \(\As = \Bs + \Cs\) for \(\Bs,\Cs \ne \zero\) and \(\Cs\cut\) has been defined,
then we set \(\As \cut :=\Bs + (\Cs \cut)\).
If \(\As \in \Ssig\) is indecomposable and has
\(\{0,\ldots,n\}\) as its base for some \(n > 0\),
define \(n \cut := -1\) and \(A \cut := \{n\cut,0,\ldots,n-1\} = \{-1,0,\ldots,n-1\}\).
For \(n \cut <  i  < n\), set \(\As \cut(n\cut,i) = \As(i,n)-1\);
if \(n\cut < i < j < n\), set \(\As\cut(i,j) = \As (i,j)\).
By Lemma \ref{top_positive},
the values taken by \(\As \cut\) are nonnegative.

\begin{lem} \label{sig_rotation_lem}
The following are true:
\begin{enumerate}

\item \label{rotation_preserves_Ssig}
If \(\As\) is in \(\Ssig\), then \(\As\cut\) is in \(\Ssig\).

\item
The map \(\As \mapsto \As\cut\) is one-to-one on the indecomposable
elements of \(\Ssig\).

\item If \(\As \ne \zero\) is in \(\Ssig\), the complexity of \(\As\cut\) is strictly less than that of \(\As\).

\end{enumerate}
\end{lem}

\begin{proof}
We will only verify (\ref{rotation_preserves_Ssig}) and leave the remainder to
the intersested reader.
In order to verify that \(\As \cut\) is in \(\Ssig\), it suffices to show that
if \(n \cut < i < j < n\), then
\[
\As \cut(n\cut,i) \geq \min \big(\As \cut(i,j)-1, \As \cut(n\cut,j)\big)
\]
with equality holding if
\(\As\cut(i,j) \ne \As\cut(n\cut,j)\).
But this is equivalent to
\(\As(i,n) \geq \min \big(\As(i,j),\As(j,n)\big)\) with equality if
\(\As(j,n) \ne \As(i,j) + 1\).
Since this is true by the equivalence of (\ref{rFocPoint}) and
(\ref{qFocPoint}) in Proposition \ref{!_equiv}, we have
that \(\As\cut\) is in \(\Ssig\).
\end{proof}

\begin{thm}\label{sig_exclam}
\(\Ssig\) is the set of signatures of elements of \(\Sgen\).
Moreover,  if \(A \in \Sgen\), then
the signature of the rotation of \(A\) is the rotation of the signature of \(A\). 
\end{thm}

\begin{proof}
First we prove that every signature of an element \(A\) of \(\Sgen\) is admissible and
hence is in \(\Ssig\).
The proof is by induction on the complexity of \(A\).
If \(|A| \leq 2\), there is nothing to show, so suppose that \(|A| \geq 3\).
Also, if \(A = B+C\) for nonempty \(B,C \in \Sgen\), then \(\Bs\) and \(\Cs\) are in \(\Ssig\)
by our induction hypothesis and since \(\As = \Bs + \Cs\), \(\As \in \Ssig\) by the closure
of \(\Ssig\) under sums.
Now suppose that \(A\) is indecomposable.
Then \(A\cut\) is in \(\Sgen\) and by Lemma \ref{rotation_lem},
\(a_{\max} \cut < a_i\) for all \(i < |A|\).
Since it has lower complexity than \(A\),
\(A\cut\) is subject to the induction hypothesis and hence its signature
is in \(\Ssig\).
It suffices to show that if \(i < j< k = |A|-1\),
then \((o(a_j,a_k) , o(a_i,a_k), o(a_i,a_j))\) is admissible.
By our inductive assumption we know that
\[
 (o(a_i,a_j) , o(a_k \cut, a_j) , o(a_k \cut,a_i))
\]
is admissible
and thus
\[
(o(a_i,a_j) , o(a_j , a_k) -1, o(a_i,a_k)-1 )
\]
is admissible.
By Corollary \ref{!_equivCor}, this is equivalent to 
\[
(o(a_j,a_k) , o(a_i,a_k), o(a_i,a_j) )
\]
being admissible. 
This completes the proof that signatures of elements of \(\Sgen\) are in \(\Ssig\).

Suppose now that \(\As\) is in \(\Ssig\); we need to prove that there is an \(A \in \Sgen\) whose signature is \(\As\).
This is proved by induction on the complexity of \(\As\).
We may assume \(\As\) has base \(\{0,\ldots,n\}\).
If  \(n = 0\), there is nothing to show.
Also, if \(\As = \Bs + \Cs\), then let \(B\) and \(C\) be elements of \(\Sgen\) which have
\(\Bs\) and \(\Cs\) as their respective signatures.
By scaling and translating if necessary, we may assume that the elements of \(B\)
are supported on \((0,1/2)\) and the elements of \(C\) are supported on \((1/2,1)\).
It is now easy to check that every pair from \(A = B + C = B \cup C\) is standard
and thus \(A\) is in \(\Sgen\) and has \(\As = \Bs + \Cs\) as its signature.

Now suppose that \(n > 0\) and \(\As\) is indecomposable.
By Lemma \ref{top_positive}, \(\As(i,n) > 0\) for all \(i < n\).
Since the complexity of \(\As\cut\) is less than that of \(\As\),
our induction hypothesis implies that \(\As \cut\) is the signature of some ordered
sequence \(a_{n\cut},a_0, \ldots, a_{n-1}\) comprising an element of \(\Sgen\).
Without loss of generality, we may assume that the supports of each of these functions
is contained in \((1/3,2/3)\) and that moreover the greatest and least transition points of \(a_{n \cut}\) are
not in the closures of the feet of the \(a_i\)'s.
Let \(a_n\) be a standard function with e-support \((0,1)\)
such that \(a_n \cut = a_{n\cut}\) and whose feet are disjoint from those of \(a_i\) for all \(i < n\).
It follows that \((a_i,a_n)\) is standard for all \(i < n\).
Since the signature of \(A\) is \(\As\) by (\ref{o_cut_one}) of Lemma \ref{o_lem},
we are finished with the proof of the first conclusion of the theorem.
That the signature and rotation maps commute follows from their definitions and 
Lemmas \ref{o_lem} and \ref{rotation_lem}.
\end{proof}

\section{The inflation operation}

\label{inflation:sec}

In this section we introduce a fundamental operation on \(\Sgen\) and
establish how it influences signatures.
This operation plays a central role in subsequent sections.
Let us begin with the observation that if \(A \in \Sgen\),
then
\[
N:=\gen{ (a_i)^{a_{\max}^p} \mid i < |A|-1 \textrm{ and } p \in \Z}
\]
is a normal subgroup of \(\gen{A}\), and \(\gen{A}\) is an extension of
\(N\) by \(\Z\).
If we define, for each \(k\),
\[
A_k := \{ (a_i)^{a_{\max}^p} \mid i < |A|-1 \textrm{ and } |p| \leq k \},
\]
then \(A_k\) is fast and \(N=\bigcup \gen{A_k}\).  Each \(A_k\)  has the
same dynamical diagram as 
\[
B_k:=\{ (a_i)^{a_{\max}^p} \mid i < |A|-1 \textrm{ and } 0 \leq p
\leq 2k \}. 
\]

The need to understand the relationships between the groups
\(\gen{A}\), \(N\), and the groups \(\gen{A_k}\) motivates a family of primitive
transformations which we term \emph{inflations}.
\begin{defn}
If \(A\) is in \(\Sgen\) and \(a \in A\), then the \emph{inflation} of
\(A\) by \(a\) is the set
\[
\infl{A}{a} := A \cup \{a^2\} \cup \{b^a \mid b \in A \textrm{ and }
b < a\} \setminus \{a\}. 
\]
\end{defn}
Observe that if \(a = a_{\max}\), then by iterating this procedure we
obtain supersets of the \(B_k\).
Note that clearly \(\gen{\infl{A}{a}} \subseteq \gen{A}\) and since \(A\) is fast,
\(\gen{A}\) embeds into \(\gen{\infl{A}{a}}\);
see the end of the proof of Lemma \ref{inflate_biembedd} for details.

For \(A\), \(B\) in \(\Sgen\), we write \(A \leq B\)
if there is a sequence \((B_i \mid i \leq n)\) such that
\(B_0 = B\), \(B_n\) has the same dynamical diagram as  \(A\), and such
that if \(i < n\), then \(B_{i+1}\) is contained in an inflation of \(B_i\).
In particular for any \(a\) in \(A\), \(A \leq \infl{A}{a} \leq A\).
Allowing \(n=0\) makes \(\le\) reflexive and \(\le\) is clearly transitive. 
Define an equivalence relation \(\equiv\) on \(\Sgen\) by
\(A \equiv B\) if and only if \(A \le B \le A\).
A major aim of the rest of the paper is to show
that the relation \(\le\) coincides with
the embeddability relation on the indecomposable elements of \(\Sgen\) and
that moreover the \(\Sgen\)-generated groups are pre-well-ordered by the
embeddability relation with order type \(\epsilon_0\).

Now we assign a \Marking{} to \(\infl{A}{a}\).  The functions in
\(\infl{A}{a} \cap A = A \setminus \{a\}\) maintain their \Markings{}.
The \Markers{} of \(b^a\) are of the form \(s a\) where \(s\) is a
\Marker{} of \(b\).  Finally, if \(s\) is a \Marker{} of a positive bump
of \(a\), then \(s\) is a \Marker{} of \(a^2\); if \(s\) is a \Marker{} of
a negative bump of \(a\), then \(sa\) is a \Marker{} of \(a^2\).  This
\Marking{} has the property that if \(t\) is in the support of \(a\)
but not in one of its feet, then \(ta\) is not in a foot of \(a^2\).
In particular, \(\infl{A}{a}\) is fast.  Notice that if \(A = B+C\),
then \(\infl{A}{b} = \infl{B}{b} + C\) for all \(b \in B\), and
\(\infl{A}{c} = B+ \infl{C}{c}\) for all \(c \in C\).

\begin{lem} \label{inflate_biembedd}
If \(A\) is in \(\Sgen\) and \(a \in A\), then \(\infl{A}{a}\) is in \(\Sgen\) and  
\(\gen{\infl{A}{a}}\) is biembeddable with \(\gen{A}\).
In particular, if \(A \leq B\) are in \(\Sgen\) then \(\gen{A}\) embeds into \(\gen{B}\).
\end{lem}

\begin{proof}
The proof that \(\infl{A}{a}\) is in \(\Sgen\) is
by induction on the complexity of \(A\).  There is nothing to show
if \(|A|\le1\).  We have noted that if \(A \in \Sgen\) and \(a \in
A\), then \(\infl{A}{a}\) is fast.  Furthermore, if \(b \in A\),
then \(\{b^a,a^2\} = \{b^a,(a^2)^a\}\) is a standard pair.  Also, by
the observation made just prior to 
the statement of this lemma,
we may assume that \(A\) is indecomposable.

If \(a < a_{\max}\), then observe that \((\infl{A}{a}) \cut \subset
\infl{(A\cut)}{a}\).  By our inductive hypothesis, \(\infl{(A
\cut)}{a}\) is in \(\Sgen\) and thus \((\infl{A}{a}) \cut\) is in
\(\Sgen\).  From the definition of standard pairs and the
indecomposability of \(A\), we have that \((a_{\max}\cut, b)\) is
standard for all \(b\in A\setminus\{a_{\max}\}\).  We know \(j:=o(b,
a_{\max})>0\).  If \(j=1\), then \(\supt(b)\) is contained in the
rightmost orbital of \(a_{\max}\) which implies that \(\supt(b^a)\)
is contained in that same orbital.  Thus \(a_{\max}\cut\ll b^a\).
If \(j>1\), then the extreme orbitals of \(a_{\max}\) are active for
both \(\{a_{\max}, b\}\) and \(\{a_{\max}, b^a\}\), so
\(a_{\max}\cut \sqsubset b^a\). 
By Lemma \ref{Sgen_char}, \(\infl{A}{a}\) is in \(\Sgen\).

Now suppose that \(a = a_{\max}\).  Observe that it suffices to
assume that \(A = \{a,b,c\}\) with \(b<a\) and \(c < a\) (we allow \(b = c\)).  Since
\(A\) is indecomposable, Lemma \ref{o_conj_bound} gives that
\(c<b^a\) and \(b<c^a\).  
We start by verifying that 
\((c,b^a)\) and \((b,c^a)\) are standard pairs. 

Assume first that \(o(c,a)\ge3\) and \(o(b,a)\ge3\), and consider
pairs \((c, b^a)\) and \((b, c^a)\).
By two applications of Lemma
\ref{o_lem}, we have \(o(c\cut, a\cut)\ge1\) and \(o(b\cut,
a\cut)\ge1\).  Thus \(c\cut \sqsubset a\cut\) and \(b\cut\sqsubset
a\cut\) from which \((b\cut)^{a\cut} = (b\cut)^a = (b^a)\cut\) and
\((c\cut)^{a\cut} = (c\cut)^a = (c^a)\cut\) follows.  By our inductive
assumption applied to the pairwise standard set \(\{a \cut,b\cut,c\cut\}\),
\((c\cut, (b\cut)^{a\cut}) = (c\cut, (b^a)\cut)\) is a
standard pair.
Since \((b^a) \cut \sqsubset c\), it follows
from the definition that \((c, b^a)\) is standard.
Similarly, \((b, c^a)\) is standard as well.

Next suppose that either \(o(c,a)\le2\) or \(o(b,a)\le2\).  By the
symmetry of \(b\) and \(c\), we can assume \(o(b,a)\le o(c,a)\).
Since we assume \(A\) is indecomposable, we have \(1\le o(b,a)\le
2\).

We first assume \(o(b,a)=2\).
From Lemma \ref{Std_Base_Facts}, we
have that the extreme orbitals of \(a\) are active in both
\(\{a,c\}\) and \(\{a,b\}\).  Further the extreme orbitals of \(a\)
contain all the transition points of \(b\) with only one transition
point of \(b\) in the rightmost orbital of \(a\).  This puts the
support of \(c\) in the rightmost orbital of \(b^a\).  From Lemma
\ref{Std_Base_Facts}, this makes \((c, b^a)\) standard.  If
\(o(c,a)=2\), a similar argument makes \((b,c^a)\) standard.  If
\(o(c,a)\ge3\), then \(o(a\cut,c)\ge 2\) and
it follows from Lemma \ref{Std_Base_Facts} that
the extreme orbitals of \(c\) each contain at least one transition
point of \(a\).  This puts all of the transition points of
\(b^{a^{-1}}\) into the extreme orbitals of \(c\) with only one
transition point of \(b^{a^{-1}}\) in the rightmost orbital of
\(c\).  From Lemma \ref{Std_Base_Facts}, this implies \((b^{a^{-1}},
c)\) standard, and thus \((b,c^a)\) is also standard.

Now assume \(o(b,a)=1\).  From Lemma \ref{Std_Base_Facts}, this puts
the support of \(b\) in the rightmost orbital of \(a\).  If
\(o(c,a)=1\), then \(c\ll b^a\) and \(b\ll c^a\) making both
\((c,b^a)\) and \((b,c^a)\) standard.  If \(o(c,a)\ge2\), then from
Lemma \ref{Std_Base_Facts}, the only transition point of \(c\) in the
rightmost orbital of \(a\) is the rightmost transition point of
\(c\).  We still have \(c\ll b^a\) making \((c, b^a)\) standard.
But now the support of \(b^{a^{-1}}\) is in the rightmost orbital of
\(c\) which makes \((b^{a^{-1}}, c)\) standard by Lemma
\ref{Std_Base_Facts}.

Observe that the other pairs such as \((b^a,a)\) and \((c^a,b^a)\) are conjugates
of standard pairs and therefore are standard as well.
This completes the proof that \(\Sgen\) is closed under inflation.

To see that \(\gen{\infl{A}{a}}\) is biembeddable in \(\gen{A}\), first note that
\(\gen{\infl{A}{a}} \subseteq \gen{A}\).
Also, since \(A':=A\setminus \{a\} \cup \{a^2\}\) has the same dynamical diagram as
\(A\), \(\gen{A'} \subseteq \gen{\infl{A}{a}}\) is isomorphic to \(\gen{A}\) by Theorem \ref{comb_to_iso}.
\end{proof}

We can also define an operation of inflation and a relation \(\le\)
on \(\Ssig\) that corresponds to inflation and \(\le\) on \(\Sgen\).
\begin{defn}
If \(\As\) is in \(\Ssig\) and \(m \in \As\), then the
\emph{inflation} of \(\As\) by \(m\), denoted \(\infl{\As}{m}\)
has as its base (where \(i^m\) is a formal symbol)
\begin{equation}\label{InflUnderlying}
A \cup \{i^m \mid (i \in A) \mand ( i < m) \mand (\As(i,m) > 0)\}
\end{equation}
to which the linear order on \(A\) is extended by declaring
\(j < i^m < j^m  < m \leq k\) whenever \(i < j < m \leq k < |A|\) and \(\As(i,m) > 0\) (in which case \(\As(j,m) > 0\) as well).
The function \(\infl{\As}{m}\) extends that of \(\As\) by defining
for each \(i,j < m \leq k < |\As|\):
\begin{align*}
\infl{\As}{m}(i^m,j^m) & := \As(i,j) \\
\infl{\As}{m}(i,j^m) & := \min\big(\As(j,m)-1,\As(i,m)\big) \\
\infl{\As}{m}(i^m,k) & := \min\big(\As(i,m),\As(m,k)\big).
\end{align*}
Here we adopt the notational convention that \(\As(m,m) = \infty\).
\end{defn}

From the provision \(\As(i,m)>0\) in (\ref{InflUnderlying}), we get
the following fact that parallels the behavior of inflations in
\(\Sgen\): if \(\As=\Bs+\Cs\), then \(\infl{\As}{b} = \infl{\Bs}{b}
+ \Cs\) for all \(b\in \Bs\), and \(\infl{\As}{c} =
\Bs+\infl{\Cs}{c}\) for all \(c\in \Cs\).

\begin{lem} \label{infl_sig_commute}
If \(A\) is in \(\Sgen\) and \(m < |A|\), then the signature of
the inflation of \(A\) by \(a_m\) is the inflation of \(\As\) by \(m\).
\end{lem}

\begin{proof}
The proof is by induction on the complexity of \(A\).
Since the lemma is vacuous if \(|A| \leq 1\) and
by the remarks made just before Lemma \ref{inflate_biembedd},
we may assume that \(A\) is indecomposable.

First suppose that \(m = |A|-1\).
Notice that the cardinality of the inflation of the signature and the signature of the inflations are equal and
moreover that the order preserving bijection pairs each conjugate with its symbolic conjugate in the inflated
signature.
Since \(\infl{A}{{a_m}}\) is in \(\Sgen\), Theorem \ref{sig_exclam} implies that if \(i,j < m\) then
\begin{align*}
o(a_i,a_j^{a_m}) & \geq \min\big( o(a_j^{a_m},a_m)-1 , o(a_i,a_m) \big) \\
 &  = \min\big( o(a_j,a_m)-1 , o(a_i,a_m) \big). 
\end{align*}
On the other hand, Lemma \ref{o_conj_bound} implies
\[
o(a_i,a_j^{a_m}) \leq \min\big( o(a_j,a_m)-1 , o(a_i,a_m) \big).
\]
Thus \(o(a_i,a_j^{a_m}) = \min\big( o(a_j,a_m)-1 , o(a_i,a_m) \big) = \infl{\As}{m}(i,j^m)\),
as desired.

If \(m < |A|-1\), then by our induction hypothesis the signature of the inflation \(\infl{(A\cut)}{m}\) is 
the corresponding inflation \(\infl{(\As\cut)}{m}\) of the signature \(\As \cut\).
Since the signature and rotation maps commute, it therefore suffices to verify that whenever \(i < m\)
\[
o(a_i^m, a_{\max}) = \min \big( o(a_i,a_m),o(a_m,a_{\max}) \big).
\]
By our induction hypothesis
\begin{align*}
o(a_i^m,a_{\max}) & = o(a_{\max} \cut , a_i^m) + 1 \\
& = 
\min
\big( 
o(a_{\max} \cut ,a_m),
o(a_i, a_m) -1
\big) + 1 \\
&
= \min \big( 
o(a_m, a_{\max}) -1,
o(a_i, a_m) -1
\big) + 1  \\
&
= \min \big( o(a_i,a_m),o(a_m,a_{\max}) \big).
\end{align*}
\end{proof}

\begin{remark}
Note the complementary role which the inequalities Proposition \ref{!_equiv} and those in
Lemma \ref{o_conj_bound} play in this proof.
\end{remark}

We now define the relation \(\le\) on \(\Ssig\) just as with
\(\Sgen\): \(\As \leq \Bs\) if and only if there is a sequence
\(\Bs_0 = \Bs\), \(\Bs_1\), ..., \(\Bs_n =\As\) such that if \(i <
n\) then \(\Bs_{i+1}\) is contained in (i.e., is a restriction of)
an inflation of \(\Bs_i\).

\begin{prop} \label{reduction_to_embedding}
If \(A, B \in \Sgen\) then
the following are true:
\begin{enumerate}

\item if \(A \leq B\), then \(\As \leq \Bs\);

\item if \(\As \leq \Bs\), then there exist \(A' \in \Sgen\) with
signature \(\As\) such that \(A' \leq B\).

\end{enumerate}
In particular, if \(\As \leq \Bs\) are in \(\Ssig\), then \(\gen{A}\) embeds into \(\gen{B}\).
\end{prop}

\begin{proof}
The first assertion is an immediate consequence of Lemma
\ref{infl_sig_commute}.  The second assertion in the case \(\As =
\infl{\Bs}{m}\) for some \(m\) is also an immediate consequence of
Lemma \ref{infl_sig_commute}; the general case follows by induction.
The final assertion follows from the second assertion 
and Lemma \ref{inflate_biembedd}.  \end{proof}

\section{Wreath products of \(\Sgen\)-generated groups}

\label{wreath:sec}

In this section we show that the partial binary operation \(*\) on \(\Sgen\) introduced in Section \ref{BasicOp:sec}
corresponds to forming a permutation wreath product of the associated groups.
We also explain why the operation is only partially defined.
Recall that if \(A,B,C \in \Sgen\) then \(A = B * C\) asserts that \(A = B
\cup C\) and both \(b < c\) and \(o(b,c) =1\) hold whenever \(b \in
B\) and \(c \in C\).
Observe that if \(A \in \Sgen\) and \(A = B*C\) for
nonempty \(B,C\), then there is an open interval \(J \subseteq I\)
such that: \begin{itemize}

\item
\(J\) contains the supports of all elements of \(B\);

\item
\(J\) is contained in the rightmost orbital of each element of \(C\)
and is disjoint from the feet of \(C\).

\end{itemize}
Notice that this implies in particular that \(\Cs(i,j) > 0\) for all \(i < j < |C|\) and hence that
\(\Cs = \exp(\Xs)\) for some \(\Xs\).
By our characterization of \(\Ssig\), \(\Xs\) is in \(\Ssig\).
If moreover \(\Bs(i,j) > 0\) for all \(i < j < |B|\),
then there is a \(t_0\) which is in the rightmost orbital of each \(b_i\).
Fix such \(J\) and \(t_0\) and define 
\[
\Xcal := \{t_0 f \mid f \in \gen{B}\}, \qquad
\qquad
\Ycal := \{ J f \mid f \in \gen{C}\}.
\]
The goal of this section is to prove the following proposition.

\begin{prop} \label{wreath_prop}
Suppose that \(A \in \Sgen\) and \(\As(i,j) > 0\) for all \(i < j <
|A|\).  If \(A = B * C\) for nonempty \(B,C \subseteq A\), then the
actions of \(\gen{B}\) on \(\Xcal\) and \(\gen{C}\) on \(\Ycal\) are
faithful and consistent.  In particular the action of \(\gen{A}\) on
\(\Xcal \times \Ycal\) is the permutation wreath product of these
actions: \(\gen{B * C} \cong \gen{B} \wr \gen{C}\).
\end{prop}

\begin{proof}
Observe that by Lemma \ref{excise_inactive},
we may assume that \(A\) has no extraneous
bumps.  The proposition is an immediate consequence of
Lemma \ref{PreWrLem} and 
the next two lemmas
using the assignments \(K = \{t_0\}\) with \(S=B\) for Lemma \ref{ComplIntLem} and
\(K=J\) with \(S=C\) for Lemma \ref{FaithfulLem}.
\end{proof}

\begin{lem}\label{ComplIntLem}
Let \(S \in \Sgen\) and \(K\) be a singleton or an open interval
disjoint from the feet of \(S\).  If \(g \in \gen{S}\) and \(K g\cap
K\ne \emptyset\), then \(g |_K\) is the identity.
\end{lem}

\begin{proof}
In the language of \cite{fast_gen}, every point in \(K\) has trivial
history.  Lemma 5.6 of \cite{fast_gen} implies that every orbit of
\(\gen{S}\) intersects \(K\) in at most one point.
This gives the conclusion if \(K\) is a singleton.  If \(K\) is open, observe that
if \(K g\cap K\ne\emptyset\), then some \(t\in K\) has \(tg\in K\)
which implies \(tg=t\).  Let \(x\) be any fixed point of \(g\) in
\(K\).  By continuity, there is an open subset \(U\) of \(K\) about
\(g\) with \(Ug\subseteq K\) implying that \(g\) is the identity on
\(U\).  Thus the fixed set of \(g\) in \(K\) is open in \(K\).
Again by continuity, the fixed set of \(g\) in \(K\) is closed in
\(K\) and must equal \(K\).
\end{proof}

\begin{lem}\label{FaithfulLem}
Let \(S \in \Sgen\) have no extraneous bumps and satisfy that
\(\mathsf{S}(i,j) > 0\) for all \(i < j < |S|\).  If \(J\) is a
singleton or an open interval contained in the rightmost orbital of
every \(f \in S\) and disjoint from the feet of \(S\), then the
action of \(\gen{S}\) on \(\{J g \mid g \in \gen{S} \}\) is
faithful.
\end{lem}

\begin{proof} 
By Proposition \ref{faithful_prop} and Lemma \ref{ComplIntLem},
it suffices to show that \(\bigcup \{J g \mid g \in \gen{S}\}\)
intersects each orbital of \(S\).
We prove this by induction on the complexity of \(S\).
Let \(f\) be the maximum element of \(S\).
If \(|S| \leq 1\), there is nothing to show.
By our inductive assumption,
\[
X:=\bigcup \{J g \mid g \in \gen{S \setminus \{f\}}\}
\]
intersects every orbital of \(S \setminus \{f\}\).
Moreover the closure of \(X\) contains the set of all transition points of \(S \setminus \{f\}\)
and hence intersects every active orbital of \(f\).
Since every orbital of \(f\) is active and open, \(X\) intersects every orbital of \(f\).
\end{proof}

\section{The reduction relation on signatures}

\label{arithmetic:sec}

Our next task is to analyze the structure of the transitive,
reflexive relation \((\Ssig,\leq)\) and show that it closely resembles a well ordering;
this will be made precise in Section \ref{S_WF:subsec} below. 
We introduce a subcollection \(\Bsig \subseteq \Ssig\) of signatures in \emph{block form}.
We show that elements of \(\Ssig\)
are equivalent to elements of \(\Bsig\) and that, modulo permuting summands, elements of \(\Bsig\)
are equivalent to unique elements of \(\Rsig\).
Our goal is to analyze the structure of \((\Rsig,\leq)\) and prove
Theorem \ref{R_thm} from the introduction.

\subsection{Block form and reduced block form signatures}

\label{Bcal_Rcal:subsec}
We use the notation
\[
\sum_{i < k+1} \As_i := \big(\sum_{i < k} \As_i
\big) + \As_k,
\quad
\text{and}
\quad
\As\cdot m := \sum_{i<m}\As.
\]
Recall that if \(\As\) is in \(\Ssig\), then \(\exp(\As)\)
has the same base as \(\As\) and
\(\exp(\As)(i,j) = \As(i,j)+1\) for all \(i < j\) in \(A\).
Note that \(\exp(\one) = \one\).
Observe that \(\exp\) maps \(\Ssig\) injectively into \(\Ssig\) and
that the range of \(\exp\) is exactly those elements of \(\Ssig\)
which take only positive values.
As noted in Section \ref{wreath:sec}, if \(A = B * C\) for \(A,B,C \in \Sgen\), then \(\Cs = \exp(\Xs)\)
for some \(X \in \Sgen\).
This partial operation on \(\Sgen\) induces a partial operation on \(\Ssig\), which can be described as follows:
\[
\Bs * \Cs (i,j) :=
\begin{cases}
\Bs(i,j) & \textrm{ if } i < j < |B| \\
\Cs(i-|B|,j-|B|) & \textrm{ if } |B| \leq i < j < |C| \\
1 & \textrm{ if } i < |B| \leq j < |C| 
\end{cases}
\]
with the operation defined precisely when 
\(\Cs = \exp(\Xs)\) for some \(\Xs \in \Ssig\) (if this condition is not met, then the result will not be in \(\Ssig\)).
Observe that \(+\) and \(*\) are associative
operations on \(\Ssig\) and that \(\exp(\As+\Bs) =
\exp(\As)*\exp(\Bs)\).

We now give a more detailed description of the family
\(\Rsig\) from Section \ref{BasicOp:sec}.
It is easiest to define \(\Rsig\) if we first define a class \(\Bsig\)
(signatures in block form), and an ordinal function \(\rho\) on \(\Bsig\).
\begin{defn}
The class \(\Bsig\) is the smallest class containing \(\{\zero,
\one\}\) so that whenever \((\Xs_i \mid i < n)\) is a sequence of
elements of \(\Bsig\), then \(\sum_{i<n} \exp(\Xs_i)\) is in
\(\Bsig\).
\end{defn}

It should be noted that, unlike \(\Ssig\), \(\Bsig\) is not hereditary with respect to \(\leq\).
Even if we start with the signature \(\As\) of a pair with oscillation 3, it is easily checked that there are
\(\As' \leq \As\) which cannot be generated from \(\one\) using \(+\), \(*\), and \(\exp\).
This can be achieved by iterating the procedure of inflating by the maximum element twice and then removing
the maximum element.

Observe that if \(\Bs\ne \zero\) is in \(\Bsig\), then
there exists a unique sequence \((\As_i \mid i < n)\) of nonzero
elements of \(\Bsig\) such that \(\Bs = \sum_{i<n} \exp(\As_i)\).
Define \(\rho\) recursively on \(\Bsig\) by \(\rho(\zero) = 0\),
\(\rho(\one) = 1\), \(\rho\big(\exp(\As)\big) =
\omega^{-1+\rho(\As)}\), and \(\rho(\sum_{i<n} \exp\big(\As_i)\big)
= \sum_{i<n} \rho\big(\exp(\As_i)\big)\).  Here \(-1 + \alpha =
\beta\) if \(\alpha = 1 + \beta\) (so \(-1 + \alpha = \alpha\) if
\(\alpha\) is infinite).  This technicality exists because \(\one\)
is a fixed point of \(\exp\), but the ordinal 1 is not a fixed point
of \(\alpha\mapsto \omega^\alpha\).
Note that \(\rho\) is not one-to-one: for instance
\[
\rho(\Zs + \exp(\Zs + \Zs)) = 1 + \omega = \omega = \rho(\exp(\Zs + \Zs)).
\]
\begin{defn}
\(\Rsig\) (signatures in reduced block form) is the smallest subfamily of
\(\Bsig\) which contains \(\zero\) and \(\one\) so that
if \((\Xs_i \mid i < n)\) is a sequence of nonzero elements of \(\Rsig\) and
\(\rho(\Xs_{i+1}) \leq \rho(\Xs_i)\) for all \(i < n-1\), then
\(\sum_{i<n} \exp(\Xs_i)\) is in \(\Rsig\).
\end{defn}

The reader can verify that if \(\Xs\) is the signature in Figure \ref{osc_fig}, then
\(\Xs\in \Rsig\) and
\[
\rho(\Xs) = \omega^{\omega^{(\omega+2)}}.
\]

Observe that if \(\Bs \in \Bsig\) and \(\As \in \Ssig\)
is contained in \(\Bs\), then \(\As \in \Bsig\) and moreover
\(\rho(\As) \leq \rho(\Bs)\) (both facts are established by
induction on the complexity of \(\Bs\)).  The next lemma has a
straightforward inductive proof and is left to the reader
(recall that we formally define \(\omega^{-1+0} = 0\)).

\begin{lem}\label{RhoIsIso} The restriction of \(\rho\) to \(\Rsig\)
is a bijection between \(\Rsig\) and the ordinals below \(\epsilon_0\).
Moreover, if \(\As,\Bs \in \Rsig\) with \(\As + \Bs \in \Rsig\), then
\[
\rho(\As + \Bs) = \rho(\As) + \rho(\Bs) \qquad \qquad \rho(\exp(\As)) = \omega^{-1 + \rho(\As)}.
\]
\end{lem}

Our first task in proving Theorem \ref{R_thm}
is to show \(\rho\) preserves the order which \(\Rsig\) inherits from \(\Ssig\) --- 
together with Lemma \ref{RhoIsIso}, this is what is needed
to establish the second assertion in Theorem \ref{R_thm}.
This will be completed by the end of Section \ref{Rcal:subsec}.
The remainder of the proof of Theorem \ref{R_thm} will be given in
Sections \ref{block:subsec}--\ref{Scal_reduction:subsec}.

The next lemma gives basic facts about the algebraic
operations and 
their interaction with the relation \(\le\) on \(\Ssig\);
the proof is straightforward and left to the reader.

\begin{lem} \label{monotone}
Suppose that \(\As \leq \As'\) and \(\Bs \leq \Bs'\) are in \(\Ssig\).
The following are true:
\begin{enumerate}

\item \(\exp(\infl{\As}{m}) \leq \infl{(\exp(\As))}{m}\) for \(m < |A|\).

\item \(\exp (\As) \leq \exp (\As')\).

\item \label{sums:monotone} \(\As+\Bs \leq \As' + \Bs'\).

\item \(\As * \exp(\Bs) \leq \As' * \exp(\Bs')\).

\end{enumerate}
\end{lem}

Recall that \(\As\equiv\Bs\) means \(\As\le \Bs\le \As\).

\begin{lem} \label{duplication}
The following are true:
\begin{enumerate}

\item \label{gen_sum}
If \((\As_i \mid i < n)\) and \(\Bs\neq\zero\) are in \(\Ssig\) and
\(j < n\) is such that \(\As_i \leq \As_j\) for all \(i < n\), then
\((\sum_{i<n} \As_i) * \exp (\Bs) \equiv \As_j * \exp (\Bs)\).

\item \label{prod_case}
If \(\As \in \Ssig\) then for all \(m\), \(\exp(\As)\cdot m \le \exp(\As+\one)\).

\end{enumerate}
\end{lem}

\begin{proof} For (\ref{gen_sum}), first observe that either \(j > 0\) and
\(\As_0 \leq \sum_{k=1}^{n-1} \As_k\) or else \(\As_{n-1} \leq \As_0
\leq \sum_{k=0}^{n-2} \As_k\).  Thus, by induction, it is sufficient
to prove the lemma when \(n=2\).  Furthermore, by Lemma
\ref{monotone}, it suffices to prove that \((\As+\As) * \exp(\Bs)
\leq \As * \exp(\Bs)\).
In fact we show that \((\As+\As)*
\exp(\Bs) = \infl{\big(\As* \exp(\Bs)\big)}{m}\) where \(m\) is the
minimum element of \(\exp(\Bs)\) regarded as a suborder of \(\As*
\exp(\Bs)\).  To see this, let \(i,j\) be elements of \(\As\) and
\(k\) be an element of \(\exp(\Bs)\) above \(m\).  We have that
\begin{align*}
\infl{(\As* \exp(\Bs))}{m}(i,j^m) & = \min \big(\As*
\exp(\Bs)(i,m),\As* \exp(\Bs)(j,m)-1 \big)
\\
& =\min (1,0) = 0 \\
\infl{\big(\As* \exp(\Bs)\big)}{m}(j^m,k) & = \min\big(\As*
\exp(\Bs)(j,m),\As* \exp(\Bs)(m,k) \big)
\\
& =
\min\big(1,\exp(\Bs)(m,k)\big) = 1
\end{align*} which coincides
with the definition of \((\As+\As)*\exp(\Bs)\).

To see that (\ref{prod_case}) is true, let \(\As\) and \(m\) be given.
By (\ref{gen_sum}) we have \((\exp(\As)\cdot m)*\one \equiv \exp(\As)*\one\).
The conclusion follows by observing
\[
(\exp(\As)\cdot m)\le (\exp(\As)\cdot m)*\one \equiv
\]
\[
\exp(\As)*\one = \exp(\As)*\exp(\one) = \exp(\As+\one).
\]
\end{proof}

Lemma \ref{duplication} is an early hint that the
arithmetic on \(\Ssig\) imitates the behavior of arithmetic on the ordinals.
For instance, in analogy to (\ref{gen_sum}), using \(1 + \omega = \omega\) we note
\[
(\omega+1 ) \cdot \omega\ =\ \sup_{n \in \omega}\ (\omega+1) \cdot n\ =\ \sup_{n \in \omega}\ \omega \cdot n + 1\ =\ \omega^2.
\]
This property of the arithmetic on \(\Ssig\) will be a constant theme in the rest of
this section and will be exploited in many of the proofs.

\subsection{Properties of reduced block form signatures}
\label{Rcal:subsec}
Next we begin our analysis of \(\Rsig\).

\begin{lem} \label{R_ordering}
If \(\Bs\) is in \(\Rsig\) and \(\alpha < \rho(\Bs)\), then there is an
\(\As\) in 
\(\Rsig\) such that
\(\As \leq \Bs\) and \(\rho(\As) = \alpha\).
\end{lem}

\begin{proof}
The proof  is by induction on \(\rho(\Bs)\).
If \(\rho(\Bs) = 0\), then the lemma is vacuously true.
Now suppose \(\rho(\Bs) \ge 1\) and let \(\Bs = \sum_{i<n} \exp(\Xs_i)\)
where \(\rho(\Xs_{i+1}) \leq \rho(\Xs_i)\) for each \(i < n-1\)
(note that 
possibly \(n=1\) in which case
\(\Bs = \exp(\Xs_0)\)).
Let \(k < n\) be minimal such that \(\alpha < \rho\big(\sum_{i \leq k}
\exp(\Xs_i)\big)\).
If \(k=0\) and \(\Xs_0 = \one\), then \(\alpha < \exp(\Xs_0) = 1\)
and thus 
\(\alpha = 0\).
In this case we take \(\As = \zero\).

If \(k=0\) and \(\Xs_0 = \Ys+\one\) for some \(\Ys \ne \zero\) in \(\Rsig\), then
\[
\alpha < \rho\big(\exp(\Xs_0)\big) = \rho\big(\exp(\Ys + \one)\big) =
\omega^{-1 + \rho(\Ys) + 1} = \rho\big(\exp(\Ys)\big) \cdot \omega
\]
and there is an \(m\) such that \(\alpha < \rho\big(\exp(\Ys)\big)
\cdot m\). 
From (\ref{prod_case}) of Lemma \ref{duplication},
\[
\Bs' := \exp(\Ys) \cdot m \leq \exp(\Ys+\one) = \exp(\Xs_0) \leq  \Bs.
\]
Since \(\Ys\) is in \(\Rsig\), \(\Bs'\) is in \(\Rsig\) and by our induction hypothesis
there is an \(\As\) in \(\Rsig\) such that \(\As \leq \exp(\Ys)
\cdot m = \Bs' 
\leq \Bs\)
and such that \(\rho(\As) = \alpha\).

If \(k=0\) but we are not in the previous cases, then \(\rho(\Xs_0)
= \delta\) 
is a limit ordinal.
Let \(0 < \gamma < \delta\) be such that \(\alpha < \omega^{-1+
\gamma}\). 
By our induction hypothesis, there is a \(\Cs \leq \Xs_0\) in
\(\Rsig\) such 
that
\(\rho(\Cs) = \gamma\) and hence \[
\rho\big(\exp(\Cs)\big) = \omega^{-1 + \gamma} < \omega^{-1 +
\delta} = \rho(\Bs).
\]
Applying our induction hypothesis again, there is an \(\As \leq
\exp(\Cs)\le \exp(\Xs_0)=B\) such that
\(\As\) is in \(\Rsig\) and \(\rho(\As) = \alpha\).

Finally, if \(0 < k < n\), then let \(0 < \alpha' <
\rho\big(\exp(\Xs_k)\big)\) be such  that
\[
\alpha = \rho(\sum_{j < k} \exp\big(\Xs_j)\big) + \alpha'.
\]
Now
\[
\alpha' < \rho\big(\exp(\Xs_k)\big) < 
\rho\big( \exp(\Xs_0) + \exp(\Xs_k)\big) \leq \rho(\Bs).
\]
By our induction hypothesis there is
an \(\As' \leq \exp(\Xs_k)\) such that \(\As' \in \Rsig\) and
\(\rho(\As') = \alpha'\).
Observe that \(\As' = \sum_{i < m} \exp(\Ys_i)\) for some \(m \geq 1\)
where the \(\Ys_i\)'s come from \(\Rsig\) and
\(\rho\big(\exp(\Ys_{i+1})\big) \leq 
\rho\big(\exp(\Ys_i)\big)\) for all
\(i < m\).
In particular, \(\rho\big(\exp(\Ys_0)\big) \leq \alpha' <
\rho\big(\exp(\Xs_k)\big)\) and therefore
\[
\As = \sum_{i < k} \exp(\Xs_i) + \sum_{j < m} \exp(\Ys_j)
\]
is in \(\Rsig\) and satisfies that \(\rho(\As) = \alpha\).
\end{proof}

While the notation hints that the converse of the next
lemma should be true, its proof is subtle and will be shown in Lemma \ref{R_order_char}
below.

\begin{lem} \label{rank_to_reduction}
If \(\As\) and \(\Bs\) are in \(\Rsig\), then \(\rho(\As) \leq
\rho(\Bs)\) implies \(\As \leq \Bs\).
\end{lem}

\begin{proof}
If \(\rho(\As) < \rho(\Bs)\), then by Lemma \ref{R_ordering} there is an
\(\As' \leq \Bs\) in \(\Rsig\) such that
\(\rho(\As') = \rho(\As)\).
Since \(\rho\) is one-to-one on \(\Rsig\), we have that \(\As'=\As\).
Similarly, if \(\rho(\As) = \rho(\Bs)\), then \(\As = \Bs\).
\end{proof}

For \(\Xs \in \Ssig\), recall that 
\(E(\Xs) = \exp\big(\exp(\Xs)\big)\). 

\begin{lem} \label{proper_successor}
If \(\As,\Bs \in \Rsig\) and \(\As+ \one \leq \Bs\), then
\(\gen{E(\Bs)}\) does 
not embed into \(\gen{E(\As)}\).
\end{lem}

\begin{proof}
Let \(n\) be the maximum element of \(E(\As+\one)\) and note
that \(E(\As) * E(\As)\) is obtained from 
\(\infl{E(\As + \one)}{n}\) by removing its maximum element.
Consequently 
\[
E(\As) * E (\As)  \leq E(\As+\one) \leq E(\Bs).
\]
By Propositions \ref{reduction_to_embedding} and \ref{wreath_prop}
\[
\gen{E(\As) * E(\As)} \cong \gen{E(\As)} \wr \gen{E(\As)}
\]
embeds into \(\gen{E(\Bs)}\).
By Proposition \ref{EAGWrG}, 
the EA-class of \(\gen{E(\As)}\) is less than that of
\(\gen{E(\As)} \wr \gen{E(\As)}\) which is at most
that of \(\gen{E (\Bs)}\).
Consequently \(\gen{E (\Bs)}\) is not embeddable into
\(\gen{E(\As)}\). 
\end{proof}

\begin{lem} \label{embedding_to_rank}
If \(\As\) and \(\Bs\) are in \(\Rsig\) and \(\rho(\As) < \rho(\Bs)\),
then \(\gen{E(\Bs)}\) does not embed into \(\gen{E(\As)}\).
\end{lem}

\begin{proof}
If \(\rho(\As) < \rho(\Bs)\), then \(\rho(\As+\one) \leq \rho(\Bs)\).
By Lemma \ref{rank_to_reduction}, \(\As + \one \leq \Bs\) and
therefore we have the desired conclusion 
by Lemma \ref{proper_successor}.
\end{proof}

This now leads to the following characterization of the restriction
of \(\leq\) to \(\Rsig\).
Taken together with Lemma \ref{RhoIsIso}, this completes the proof of the second half
of Theorem \ref{R_thm}. 

\begin{lem} \label{R_order_char}
If \(\As\) and \(\Bs\) are elements of \(\Rsig\), then the following
are equivalent: 
\begin{enumerate}

\item \label{R_reduction}
\(\As \leq \Bs\).

\item \label{R_rank}
\(\rho(\As) \leq \rho(\Bs)\).

\end{enumerate}
\end{lem}

\begin{proof}
Lemma \ref{rank_to_reduction} establishes that (\ref{R_rank})
implies (\ref{R_reduction}). 
If \(\As\le \Bs\), Lemma \ref{monotone} gives
\(E(\As)\le E(\Bs)\) which implies that
\(\gen{E(\As)}\) embeds in \(\gen{E(\Bs)}\) by Proposition
\ref{reduction_to_embedding}.  The contrapositive of Lemma
\ref{embedding_to_rank} with the roles of \(\As\) and \(\Bs\)
reversed gives \(\rho(\As)\le \rho(\Bs)\).
\end{proof}

\begin{remark}
While Lemma \ref{R_order_char} does not mention groups or embeddings
between them, it is convenient to use the group theory concept of
the EA-class to establish the implication. 
It is not obvious how to provide a
proof of Lemma \ref{R_order_char} which avoids group theory.
\end{remark}

\subsection{Properties of block form signatures} \label{block:subsec}

We now broaden our analysis to \(\Bsig\).

\begin{lem} \label{regroup}
If \(\As\), \(\Bs\), and \(\Cs\) are in \(\Ssig\) with
\(\Bs,\Cs \ne\zero\), then 
\[
\exp\big(\As + \Bs * \exp(\Cs)\big) \leq \exp \big((\As+\Bs) *
\exp(\Cs)\big). 
\]
\end{lem}

\begin{proof}
Let \(m\) be the minimal element of \(\Bs\) in \(\exp\big((\As+\Bs)
* \exp(\Cs)\big)\). 
It suffices to check that the map
\[
\pi: \exp\big(\As + \Bs * \exp(\Cs)\big) \to
\infl{\exp\big((\As+\Bs) * \exp(\Cs)\big)}{m} 
\]
defined by
\[
\pi (i) =
\begin{cases}
i^m & \textrm{ if } i \in \As \\
i & \textrm{ if } i \in \Bs* \exp(\Cs)
\end{cases}
\] 
is an embedding.

First observe that the restrictions of \(\pi\) to
\(\As\) and \(\Bs * \exp(\Cs)\) are embeddings.
Furthermore, if \(i \in \As\) and
\(j \in \Bs * \exp(\Cs)\), then \(i^m < m \leq j\) and hence \(\pi\)
is order preserving. 
Finally suppose that \(i \in \As\) and \(j \in \Bs * \exp(\Cs)\).
By definition, 
\[
\infl{\exp\big((\As+\Bs) * \exp(\Cs)\big)}{m} (\pi(i),\pi(j))  =
\]
\[
\min\big( \exp\big((\As+\Bs) * \exp(\Cs)\big)(i,m), 
\exp\big((\As+\Bs) * \exp(\Cs)\big)(m,j) \big)
\]
\[
=\min \big(1,\exp\big((\As+\Bs) * \exp(\Cs)\big)(m,j) \big)
\]
\[
 = 1 =  \exp\big(\As+\Bs*\exp(\Cs)\big)(i,j).
\]
\end{proof}

\begin{lem} \label{weak_init_seg}
If \(\As\) is in \(\Bsig\) and \(\beta < \rho(\As)\),
then there exist \(\Bs,\Cs \in \Bsig\) such that \(\exp(\As)\) is
equivalent to \(\exp(\Bs+\Cs)\) and 
\(\beta \leq \rho(\Bs) < \rho(\As)\).
Moreover, if \(\As\) is indecomposable, then we can take \(\Cs = \As\).
\end{lem}

\begin{proof}
The proof is by induction on \(\rho(\As)\) and then on the
cardinality of \(\As\). 
Toward this end, let \(\As\) be given;
there are now several cases to consider.
If \(\As = \zero\), then the lemma is vacuously true and if
\(\As = \one\) then \(\beta = 0\) and we can take \(\Bs =\zero\) and \(\Cs = \As\).
 
Next suppose that \(\As = \Bs_0 + \Cs_0\) with both \(\Bs_0\) and
\(\Cs_0\) not \(\zero\).
If \(\beta \leq \rho(\Bs_0)\), then we are done.
Otherwise, let \(\gamma\) be such that
\(\beta = \rho(\Bs_0) + \gamma\), noting that \(\gamma < \rho(\Cs_0)\).
Since the cardinality of \(\Cs_0\) is smaller than that of \(\As\),
we can apply our induction hypothesis to find 
\(\Bs_1\) and \(\Cs\) such that \(\exp(\Bs_1+\Cs)\) is equivalent to
\(\exp(\Cs_0)\) and 
\(\gamma \leq \rho(\Bs_1) < \rho(\Cs_0)\).
Observe that \(\Bs = \Bs_0 + \Bs_1\) and \(\Cs\) satisfy the
conclusion of the lemma: 
\[
\beta = \rho( \Bs_0 ) + \gamma \le \rho( \Bs_0 ) + \rho( \Bs_1 )
= \rho( \Bs_0 + \Bs_1 ) = \rho(\Bs)
\]
while basic manipulations with the arithmetic in \(\Bsig\) gives
\[
\exp(\Bs + \Cs) =  \exp( \Bs_0 ) * \exp( \Bs_1 + \Cs) \equiv
\exp(\Bs_0) * \exp(\Cs_0) = \exp(\As). 
\]

Next suppose that \(\As = \exp(\Ds + \one)\) for \(\Ds \ne \zero\).
Let \(n\) be such that \(\beta \leq \omega^{-1 + \rho(\Ds)} \cdot n\).
Set \(\Bs = \exp(\Ds) \cdot n\).
By our choice of \(n\) we have that \(\beta \leq \rho(\Bs) < \rho(\As)\).
Also, it is clear that \(\exp(\As) \leq \exp(\Bs+\As)\) and hence it
is sufficient in this case to show 
that \(\exp(\Bs+\As) \leq \exp(\As)\).
This follows from Lemmas \ref{regroup} and \ref{duplication}:
\[
\begin{split}
\exp\big(\exp(\Ds) \cdot n + &\exp(\Ds+\one)\big) =
\exp\big(\exp(\Ds) \cdot n + \exp(\Ds)* \exp(\one)\big) \\
\leq &\exp\big(\big(\exp(\Ds) \cdot (n+1)\big) * \exp(\one)\big) \\
\leq &\exp\big(\exp(\Ds) * \exp(\one)\big) =
\exp\big(\exp(\Ds+\one)\big).
\end{split}
\]

Finally, suppose that \(\As = \exp(\Ds)\) and \(\rho(\Ds)\) is a
limit ordinal \(\delta\). 
Let \(\gamma < \delta\) be such that \(\beta \leq \omega^{-1 + \gamma}\).
By our induction hypothesis, there exist \(\Es,\Fs \in \Bsig\) such that
\(\exp(\Ds) \equiv \exp(\Es + \Fs)\) and \(\gamma \leq \rho(\Es) <
\delta\).
We set \(\Bs = \exp(\Es)\).
As in the previous case, it suffices to show that
\(\exp(\Bs+\As) \leq \exp(\As)\). 
This again follows  from Lemmas \ref{regroup} and \ref{duplication}: 
\[
\begin{split}
\exp\big(\Bs + \As\big) \equiv &\exp\big(\exp(\Es) +
\exp(\Es+\Fs)\big) \\
 = &\exp \big(\exp(\Es) + \exp(\Es)*\exp(\Fs)\big) \\
\leq &\exp\big(\big(\exp(\Es) + \exp(\Es)\big) * \exp(\Fs)\big) \\
 \leq &\exp\big(\exp(\Es)* \exp(\Fs)\big) \equiv
 \exp\big(\exp(\Ds)\big) = \exp(\As).
\end{split}
\]
\end{proof}

\begin{lem} \label{exact_init_seg}
If \(\As\) is in \(\Bsig\) and \(\beta < \rho(\As)\),
then there exist \(\Bs,\Cs \in \Bsig\) such that \(\exp(\As)\) is
equivalent to \(\exp(\Bs+\Cs)\) and 
\(\beta = \rho(\Bs)\).
Moreover, if \(\As\) is indecomposable, then we can take \(\Cs = \As\).
\end{lem}

\begin{proof}
The proof is by induction on \(\rho(\As)\).
As before, the case \(\rho(\As) \leq 1\) is trivial.
If \(\beta < \rho(\As)\), then by Lemma \ref{weak_init_seg},
there exist \(\Xs,\Ys \in \Bsig\) such that 
\(\exp(\Xs+\Ys) \equiv \exp (\As)\), \(\beta \leq \rho(\Xs) <
\rho(\As)\), and \(\Ys = \As\) if 
\(\As\) is indecomposable.
If \(\rho(\Xs) = \beta\), then we are done.
Otherwise,
by our inductive assumption, there exist \(\Bs,\Rs \in \Bsig\) such
that 
\(\rho(\Bs) = \beta\) and \(\exp (\Bs+\Rs) \equiv \exp(\Xs)\).
If we set \(\Cs := \Rs + \Ys\), we have that
\[
\begin{split}
\exp(\Bs + \Cs) = &\exp\big(\Bs + (\Rs+\Ys)\big) = \exp(\Bs+\Rs)
* \exp(\Ys) \\
\equiv  &\exp(\Xs) * \exp(\Ys) = \exp(\Xs+\Ys)  \equiv \exp(\As).
\end{split}
\]
If \(\As\) is indecomposable, then we get \(\exp(\Bs+\As)\equiv
\exp(\As)\) from
\[
\begin{split}
\exp(\As) \le &\exp(\Bs+\As) \le \exp(\Bs + \Rs + \As)
=  \exp( \Bs + \Rs) * \exp(\As) \\
\equiv & \exp(\Xs)*\exp(\As) = \exp(\Xs+\As) \equiv \exp(\As).
\end{split}
\]
So \(\Cs=\As\) fits the conclusion of the lemma.
\end{proof}

\begin{lem} \label{block_to_R}
If \(\As\) is an indecomposable element of \(\Bsig\), then there is
a unique indecomposable \(\Bs \in \Rsig\) such that 
\(\As \equiv \Bs\).
\end{lem}

\begin{proof}
First observe that if \(\As \in \Bsig\) is indecomposable, then
\(\As = \exp(\Xs)\) for some \(\Xs\). 
Uniqueness of \(\Bs\) follows immediately from Lemmas \ref{RhoIsIso} and
\ref{R_order_char}. 

The proof of existence of \(\Bs\) is by induction on \(\rho(\As)\).
If \(\rho(\As) = 1\), then \(\As = \one\) and \(\As \in \Rsig\).
If \(\As = \exp(\Xs)\) and \(\Xs\) is indecomposable, then by our
induction hypothesis, 
there is a \(\Ys \in \Rsig\) such that \(\Ys \equiv \Xs\).
Since \(\Bs = \exp(\Ys)\) is also in \(\Rsig\),
it follows from Lemma \ref{monotone} that \(\As = \exp(\Xs) \equiv
\exp(\Ys) = \Bs\). 

If \(\As = \exp\big( \sum_{i<n} \exp(\Xs_i)\big)\), then by our
induction hypothesis there are \((\Ys_i \mid i < n)\) in \(\Rsig\)
such that 
\(\exp(\Ys_i) \equiv \exp(\Xs_i)\) for each \(i < n\).
If there is no \(k < n-1\) such that \(\rho(\Ys_k) <
\rho(\Ys_{k+1})\), then 
\(\exp\big( \sum_{i<n} \exp(\Ys_i)\big)\) is in \(\Rsig\), is indecomposable, and is
equivalent to \(\As\). 

Suppose now that there is a \(k < n-1\) such that \(\rho(\Ys_k) <
\rho(\Ys_{k+1})\) and note that in this case \(n > 1\). 
We first claim that \(\exp\big(\exp(\Ys_k) + \exp(\Ys_{k+1})\big)
\equiv \exp\big(\exp(\Ys_{k+1})\big)\). 
To see this, observe that by Lemma \ref{exact_init_seg} with
\(\As=\exp(\Ys_{k+1})\) and \(\beta=\rho(\exp(\Ys_k))\)
there exists
a \(\Rs \in \Bsig\) such that 
\(\rho(\Rs) = \rho\big(\exp(\Ys_k)\big)\) and \(\exp\big(\Rs +
\exp(\Ys_{k+1})\big) \equiv \exp\big(\exp(\Ys_{k+1})\big)\). 
Since
\[
\rho\big(\exp(\Ys_k)\big) < \rho\big(\exp(\Ys_{k+1})\big) \leq \rho(\As),
\] we can apply
our induction hypothesis to conclude that \(\Rs \equiv \exp(\Ys_k)\)
and thus 
that \(\exp\big( \exp(\Ys_k) + \exp(\Ys_{k+1})\big) \equiv
\exp\big(\exp(\Ys_{k+1})\big)\). 
Removing \(\Ys_k\) from the sum, reindexing the remaining summands, and repeating the process if
necessary, we eventually 
arrive at an indecomposable element of \(\Rsig\) which is equivalent to \(\As\).
\end{proof}

\subsection{Representing elements of \(\Ssig\) in \(\Rsig\)}

\label{Scal_reduction:subsec}

Next we turn to the general analysis of elements of \(\Ssig\) in order to complete the proof of Theorem \ref{R_thm}.
We need the following characterization of nontrivial products,
analogous to Lemma \ref{dir_sum_char}.

\begin{lem} \label{factor_lem}
Suppose \(\As\) is in \(\Ssig\) and has cardinality at least \(2\) and
maximum element \(n\).
If \(\As\) satisfies the following conditions: 
\begin{itemize}

\item for all \(i < n\), \(\As(i,n) \geq 1\) and

\item
there exist \(m < n\) such that \(\As(i,n) = 1\) if and only if \(i \leq m\),
\end{itemize} 
then \(\As = \Bs*\exp(\Cs)\) for some nonzero
\(\Bs,\Cs \in \Ssig\).
\end{lem}

\begin{proof}
If \(i \leq m < j < n\), then \(\As(i,n) = 1 < \As(j,n)\) and hence
\(\As(i,j) = \min \big(\As(j,n)-1,1\big) = 1\).
Furthermore, if \(m < i < j < n\) then \(\As(i,j) \geq  \min (\As(j,n)-1, \As(i,n)) > 0\).
Thus \(\As = \Bs * \exp(\Cs)\), where \(\exp(\Cs)\) consists of the
elements of \(\As\) above 
\(m\) and \(\Bs\) consists of the remaining elements.  This holds
even if \(m=n-1\).  In this case \(\As = \Bs*\one = \Bs *
\exp(\one)\). 
\end{proof}

The next lemma immediately yields Proposition \ref{+*_form}.

\begin{lem} \label{decomp_lem}
If \(\As \ne \zero,\one\) is in \(\Ssig\),
then one of the following is true for some nonzero \(\Bs,\Cs \in \Ssig\):
\begin{enumerate}

\item \(\As = \exp(\Bs)\),

\item \(\As = \Bs + \Cs\), or

\item \(\As\) is equivalent to \(\Bs * \exp(\Cs)\) and \(\Bs * \exp(\Cs)\) has the same 
cardinality as \(\As\).

\end{enumerate}
\end{lem}

\begin{proof}
Let \(\As\) be given with maximum element \(n\). 
We are done if there are no \(i < j \leq n\) such
that \(\As(i,j) = 0\) since then \(\As = \exp(\Bs)\) for some \(\Bs\).
If \(\As(i,n) = 0\) for some \(i < n\), then \(\As\) is decomposable
by Lemma \ref{top_positive} 
and we are finished.

Now suppose the first two conclusions of the lemma do not occur.
In this case, there is an \(i < j < n\) such that \(\As(i,j) = 0\)
but there is no \(i' < n\) such 
that \(\As(i',n) = 0\).
Since \(\As\) is in \(\Ssig\), it must be that \(\As(j,n) = 1\) and
in particular there is a 
\(j < n\) such that \(\As(j,n) = 1\).
Define \(\Ds\) to be the restriction of \(\infl{\As}{n}\) such that
\[
\Ds = \{i \mid \As(i,n) =1\}  \cup \{i^n \mid i<n \textrm{ and }\As(i,n) > 1\} \cup \{n\}.
\]
Clearly \(\Ds \leq \As\).
Observe that \(\Ds\) satisfies the hypothesis of Lemma
\ref{factor_lem} and thus has the desired form 
\(\Bs * \exp(\Cs)\) for some nontrivial \(\Bs\) and \(\Cs\).

Thus it is sufficient to show that \(\As \leq \Ds\).  For each
\(i<n\) with \(\As(i,n)=1\), let \(u_i\) be the unique least element
of \(\As\) with \(i<u_i\le n\) and \(\As(u_i,n)>1\).  This always
exists since \(\As(n,n)=\infty\).
We now perform a sequence of inflations of \(D\) by the members of
\[
J := \{u_i^n \mid i<n \textrm{ and } \As(i,n)=1\}
\]
in increasing order and then remove some of the resulting conjugates
so that the  
base of the resulting \(\Xs \in \Ssig\) consists of:
\begin{itemize}

\item all \(i^n\) such that \(i \le n\) and \(\As(i,n) > 1\) (where
\(n^n:=n\));

\item all \(i^{(u_i^n)}\) such that \(\As(i,n)=1\).

\end{itemize}
We claim that \(\As\) is equivalent to \(\Xs\).
Define \(\pi: A \to X\) by 
\[
\pi(i) =
\begin{cases}
i^n, &i\le n,\ \As(i,n) > 1, \\
i^{(u_i^n)},
&i<n,\ \As(i,n) = 1.
\end{cases}
\]
The proof is completed by Claims \ref{PiOrder} and \ref{PiIso} below.
\end{proof}

\begin{claim} \label{PiOrder}
\(\pi\) is order preserving.
\end{claim}

\begin{proof}
Let \(i < j < n\).
If \(\As(i,n)\) and \(\As(j,n)\) are both greater than \(1\),
then \(\pi(i) = i^n < j^n = \pi(j)\).
If \(\As(i,n) = 1\) and \(\As(j,n) > 1\), then
\(i<u_i\le j\) and \(\pi(i) = i^{(u_i^n)} < u_i^n\le j^n = \pi(j)\).
If \(\As(i,n) > 1\) and \(\As(j,n) = 1\), then \(i<j<u_j\),
\(i^n<u_j^n\), and in the inflation by \(u_j^n\), we have
\(\pi (i) = i^n<j^{(u_j^n)} = \pi(j)\). 

Finally if \(\As(i,n) = \As(j,n) = 1\), then either \(u_i=u_j\) or
\(u_i<u_j\).  In the first case, \(\pi(i) =
i^{(u_j^n)}<j^{(u_j^n)} = \pi(j)\).  In the second case,
\(i^{(u_i^n)} < u_i^n < u_j^n\), so \(\pi(i) = i^{(u_i^n)} < u_i^n <
j^{(u_j^n)} = \pi(j)\).
\end{proof}

\begin{claim} \label{LowOscSpefics}
If \(i<u_i\) and \(\As(i,n) = 1\), then
\(\Ds(i,u_i^n) = 1\).
If additionally we have \(u_i<j<u_j\) and \(\As(j,n) = 1\), then \(\As(i,j) = 0\).
\end{claim}

\begin{proof}
The definition of inflation gives
\[
\Ds(i,u_i^n) = \min\big(\As(u_i,n)-1, \As(i,n)\big) = 1.
\]
For the second conclusion, we have \(\As(j,n) = 1< \As(u_i,n)\), so 
\[
\As(u_i,j) = \min\big(\As(j,n)-1, \As(u_i,n)\big) = 0.
\]
Now \(i<u_i\) and Lemma \ref{ZeroToEnd} gives \(\As(i,j)=0\).
\end{proof}

\begin{claim}\label{PiIso}
If \(i < j \le n\), then \(\Xs\big(\pi(i),\pi(j)\big) = \As(i,j)\).
\end{claim}

\begin{proof}
We first consider \(j=n\).  
If either \(\As(i,n)>1\) or both \(\As(i,n)=1\) and
\(u_i=n\),
then \(\Xs(\pi(i), \pi(n)) = \Ds(i^n,n)
= \As(i,n)\).
If \(\As(i,n)=1\) and \(u_i<n\), then by Claim
\ref{LowOscSpefics} we have
\[
\Xs(i^{(u_i^n)}, n) = \min(\Ds(i,u_i^n), \Ds(u_i^n,n)) =
\min(1, \Ds(u_i^n,n)) =
1.\]

For \(i<j<n\), we first assume \(\As(i,n)\ne \As(j,n)\).
Since \(\Xs\) is in \(\Ssig\), the previous case implies
\begin{align*}
\Xs(\pi(i),\pi(j)) & = \min\big( \Xs(\pi(j),n)-1 ,
\Xs(\pi(i),n) \big) \\ 
& = \min\big( \As(j,n)-1 , \As(i,n) \big) \\
& = \As(i,j).
\end{align*}
If \(\As(i,n) = \As(j,n) > 1\),
then \(\pi(i) = i^n\) and \(\pi(j) = j^n\).
In this case we have that
\(\Xs(i^n,j^n) = \Ds(i^n,j^n) = \As(i,j)\).

Finally we have the case \(\As(i,n) = 1 = \As(j,n)\).  If
\(u_i = u_j\), then
\(\Xs(i^{(u_j^n)}, j^{(u_j^n)}) = \Ds(i,j) = \As(i,j)\).  So we
assume 
\(i<u_i<j<u_j\).  From the beginning of the proof \(\Xs(j^{(u_j^n)},
n) = \As(j,n)=1\).  Since \(\Xs(u_i^n,n) = \Ds(u_i^n,n) =
\As(u_i,n)>1\), we have 
\[
\Xs(u_i^n, j^{(u_j^n)}) = \min(\Xs(j^{(u_j^n)}, n)-1, \Xs(u_i^n, n))
=0.
\]
Now \(i^{(u_i^n)}< u_i^n\), so by Lemma \ref{ZeroToEnd}, we have
\(\Xs(i^{(u_i^n)}, j^{(u_j^n)})=0\).  By Claim \ref{LowOscSpefics},
\(\As(i,j)=0 = \Xs (\pi(i),\pi(j))\).  
\end{proof}

\begin{lem}  \label{standard_to_block}
If \(\As\) is in \(\Ssig\), then there is a \(\Bs \in \Bsig\) such
that \(\As \equiv \Bs\). 
Moreover, if \(\As\) is indecomposable, then \(\Bs\) can be taken to
be indecomposable. 
\end{lem}

\begin{proof}
The proof is by induction on the complexity of \(\As\).
If \(\As = \one\) or \(\As=\zero\), then \(\As\) is already in \(\Bsig\).
If \(\As = \exp(\Bs)\) for some \(\Bs\), then
by our induction hypothesis, there is a \(\Bs' \in \Bsig\) such that
\(\Bs' \equiv \Bs\).
Since \(\exp(\Bs') \in \Bsig\) and by Lemma \ref{monotone} \(\As
\equiv \exp(\Bs')\), we are done. 
If \(\As = \Bs+\Cs\), then by our induction hypothesis, there are
\(\Bs',\Cs' \in \Bsig\) such 
that \(\Bs' \equiv \Bs\) and \(\Cs' \equiv \Cs\).
Again, \(\Bs'+\Cs' \in \Bsig\) and \(\As = \Bs+\Cs \equiv \Bs' + \Cs'\).

If \(\As\) is not of these three forms, then Lemma \ref{decomp_lem}
implies 
that there exist \(\Bs,\Cs \in \Ssig\) such that
\(\As \equiv \Bs*\exp(\Cs)\) and the cardinality of
\(\Bs*\exp(\Cs)\) is the same as that of \(\As\). 
As above, by induction hypothesis, we may assume that \(\Bs\) and
\(\Cs\) are both in \(\Bsig\). 
If \(\Bs = \exp(\Xs)\) for some \(\Xs\) (which would necessarily be
in \(\Bsig\)), 
then \(\Bs*\exp(\Cs) = \exp(\Xs) * \exp(\Cs) = \exp(\Xs+\Cs)\),
which is in \(\Bsig\). 
Thus we may assume that \(\Bs = \sum_{i=1}^n \exp(\Xs_i)\) where each
\(\Xs_i\) is in \(\Bsig\).
Let \(i \leq n\) be such that \(\rho(\exp(\Xs_i))\) is maximized.
By Lemmas \ref{R_order_char} and \ref{block_to_R},
we have that \(\exp(\Xs_j) \leq \exp(\Xs_i)\) for all \(j \neq n\).
By Lemma \ref{duplication},
\[
\big(\sum_{j=1}^n \exp(\Xs_j)\big) * \exp(\Cs) \equiv \exp(\Xs_i) *
\exp(\Cs) = \exp(\Xs_i + \Cs)\in \Bsig.
\]
\end{proof}

We are now in a position to complete the proof of Theorem \ref{R_thm}.
Let \(\Ssigp \subseteq \Ssig\) consist of all \(\As \in \Ssig\) for which there is a 
\(\Bs \in \Rsig\) such that \(\As \equiv \Bs\) and define \(\rho(\As):=\rho(\Bs)\).
Since the restriction of \(\equiv\) to \(\Rsig\) is just the equality relation by Lemmas \ref{RhoIsIso} and
\ref{R_order_char}, this is well defined.
Clearly \(\Rsig \subseteq \Ssigp\).
By Lemmas \ref{block_to_R} and \ref{standard_to_block},
\(\Ssigp\) includes all of the indecomposable elements of \(\Ssig\).
Furthermore, if \(( \As_i \mid i \leq n )\) is a sequence of indecomposable elements of \(\Ssig\) and
\(\rho(\As_i) \geq \rho(\As_{i+1})\) for all \(i < n\), then 
\(\sum_{i \leq n} \As_i\) is in \(\Ssigp\).
In particular, 
Lemma \ref{R_order_char} immediately extends to all elements of \(\Ssigp\).

Now observe that if \(\As \in \Ssig\), then \(\As = \sum_{i < n} \As_i\) for some sequence of indecomposable \(\As_i\)'s.
Hence there is a permutation \(\sigma\) of \(\{0,\ldots,n-1\}\) such that
\(\sum_{i < n} \As_{\sigma(i)}\) is in \(\Ssigp\) --- i.e. some \emph{reordering} of \(\As\) is in \(\Ssigp\).
In particular for every \(\As \in \Ssig\) there is an \(\As' \in \Ssigp\) such that
\(\gen{\As} \cong \gen{\As'}\).
This completes the proof of Theorem \ref{R_thm}.

\subsection{\((\Ssig,\leq)\) is well-founded} \label{S_WF:subsec}
We now extend the definition of \(\rho\) to \(\Ssig\) and prove that \((\Ssig,\leq)\) is well-founded.
Observe that if \(\As'\) is a reordering of \(\As\) and both are in \(\Ssigp\), then \(\rho(\As) = \rho(\As')\).
If \(\As \in \Ssig\) define
\(\rho(\As) := \rho(\As')\) where \(\As'\) is in \(\Ssigp\) and \(\As'\) is a reordering of \(\As\);
by our observation this does not depend on our choice of \(\As'\).
Notice also that if \(\As \in \Ssig\), then \(\rho(\As+\one) = \rho(\As) + 1\).
Unlike in the case of \(\Ssigp\), however, it is not immediately clear that \(\As \equiv \Bs\) implies
\(\rho(\As) = \rho(\Bs)\).
This will be a consequence of Proposition \ref{S_order_char} below.

In what follows, we will write the \emph{sequence of ranks of \(\As\)} to refer to the sequence of \(\rho\)-values of
the indecomposable summands of \(\As\).
The \emph{multiset of ranks of \(\As\)} is the set with repetitions consisting of the range of the sequence of
ranks of \(\As\).
Thus \(\rho(\As)\) is the sum of the multiset of ranks of \(\As\) ordered largest to smallest.
Notice that for all \(\As, \Bs \in \Ssig\):
\begin{itemize}

\item if \(\As\) and \(\Bs\) have the same multiset of ranks, then \(\rho(\As) = \rho(\Bs)\);

\item if the multiset of ranks of \(\As\) is obtained from the multiset of ranks of \(\Bs\) by replacing one of the ordinals by one or more smaller ordinals, then \(\rho(\As) < \rho(\Bs)\).

\end{itemize}
This is a consequence of Theorem \ref{CNF} and the fact that any multiset of ranks consists of indecomposable ordinals
--- those of the form
\(\omega^\xi\).

If \(\As \ne \zero\) is in \(\Ssig\), let \(\As^-\) be the result of removing the maximum element of \(\As\).

\begin{lem} \label{A-_succ}
For any nonzero \(\As \in \Ssig\), \(\As^- + \one \leq \As\).
\end{lem}

\begin{proof}
This is proved by induction on the complexity of \(\As\).
If \(\As = \As^- + \one\), there is nothing to show (this includes the case when \(\As = \one\)).
Let \(n = |\As| -1 > 0\) and fix an \(i < n\) such that \(\As(i,n) > 0\).
Define \(\Bs\) to be the restriction of \(\infl{\As}{n}\) to 
\(\{0,\ldots,n-1,i^n\}\), noting that \(\Bs^- = \As^-\).
We now have that \(\Bs(i,i^n) = \As(i,n)-1\) and for
all \(j < n\), 
\[
\Bs(j,i^n) = \min \big(\As(i,n)-1,\As(j,n)\big) \leq \As(j,n).
\]
Thus \(\Bs \leq \As\) and \(\Bs\) has smaller complexity than \(\As\).
By our induction hypothesis \(
\As^- + \one =
\Bs^- + \one \leq \Bs \leq \As\). 
\end{proof}

\begin{prop} \label{S_order_char}
If \(\As\) and \(\Bs\) are elements of \(\Ssig\), then \(\As \leq \Bs\) implies \(\rho(\As) \leq \rho(\Bs)\).
If moreover \(\As\) and \(\Bs\) are in \(\Ssigp\), then \(\rho(\As) \leq \rho(\Bs)\) implies \(\As \leq \Bs\).
\end{prop}

\begin{proof}
As noted in the lead-up to Lemma \ref{A-_succ}, when \(\As,\Bs \in \Ssigp\) both implications in the lemma follow from Lemma \ref{R_order_char} and the transitivity of \(\leq\).
Suppose now that \(\As \leq \Bs\) is in \(\Ssig\).
Fix indecomposable \(\Bs_0,\ldots,\Bs_{n-1}\) in \(\Ssig\) such that \(\Bs = \sum_{i < n} \Bs_i\).
It suffices to show that \(\rho(\As) \leq \rho(\Bs)\) if either \(\As\) is an inflation of \(\Bs\) or if \(\As\) is obtained from
\(\Bs\) by deleting an element.

First suppose that \(\As = \infl{\Bs}{m}\) for some \(m < |\Bs|\).
Let \(j < n\) be such that \(m\) comes from the summand \(\Bs_j\) and observe that 
\[
\As = \sum_{i < j} \Bs_i + \infl{\Bs_j}{m} + \sum_{j < i < n} \Bs_i.
\]
Since \(\Bs_j \equiv \infl{\Bs_j}{m}\) are indecomposable, \(\rho(\Bs_j) = \rho(\infl{\Bs_j}{m})\).
It follows that \(\As\) and \(\Bs\) have the same multiset of ranks and hence that \(\rho(\As) = \rho(\Bs)\). 

Now suppose that \(m < |\Bs|\) and \(\As\) is obtained from \(\Bs\) by removing \(m\).
Let \(j < n\) be such that \(m\) comes from the summand \(\Bs_j\) and let \(\Bs_j'\) denote
the result of removing the corresponding element.
If \(\Bs_j'\) is indecomposable (including the possibility \(\Bs_j' = \zero\)),
then \(\rho(\Bs_j') \leq \rho(\Bs_j)\).
Since the multiset of ranks of \(\As\) is either the same as that of \(\Bs\) or the result of decreasing one of its entries,
\(\rho(\As) \leq \rho(\Bs)\).

Suppose now that \(\Bs_j' = \sum_{k \leq l} \Cs_k\) for some indecomposable \(\Cs_0,\ldots,\Cs_l\)
in \(\Ssig\) with \(l \geq 1\).
First note that if the maximum elements of \(\Bs_j\) and \(\Bs_j'\) coincide --- i.e. if \(m\) does does
not correspond to the maximum element of \(\Bs_j\) --- then Lemma \ref{dir_sum_char} implies
that \(\Bs_j\) is decomposable, contrary to our assumption.
Since the multiset of ranks of \(\As\) is obtained by replacing \(\rho(\Bs_j)\) with
\(\rho(\Cs_0), \ldots,\rho(\Cs_l)\), it suffices to show that \(\rho(\Cs_k) < \rho(\Bs_j)\) for all \(k \leq l\).
By applying Lemma \ref{A-_succ} to the subset of \(\Bs_j\) consisting of the elements of \(\Cs_k\) together
with \(m\), we obtain that \(\Cs_k + \one \leq \Bs_j\).
Since \(\Cs_k+\one\) and \(\Bs_j\) are all in \(\Ssigp\), it follows that
\(\rho(\Cs_k) < \rho(\Cs_k) + 1 \leq \rho(\Bs_j)\) as desired.
\end{proof}

\begin{prop} \label{decrease_rank}
If \(\As \ne \zero\) is in \(\Ssig\),
then \(\rho(\As^-) < \rho(\As)\) and \(\As^- < \As\). 
\end{prop}

\begin{proof}
By Lemma \ref{A-_succ}, \(\As^- + \one \leq \As\).
By Proposition \ref{S_order_char}, \(\rho(\As^- + \one) \leq \rho(\As)\).
Hence \[
\rho(\As^-) < \rho(\As^{-}) + 1 = \rho(\As^{-} + \one) \leq \rho(\As).
\]
Combining this with \(\As^- \leq \As\) and Proposition \ref{S_order_char}, we obtain \(\As^{-} < \As\).
\end{proof}

\begin{prop} \label{S_WF}
If \(\As,\Bs \in \Ssig\) and \(\As < \Bs\), then \(\rho(\As) < \rho(\Bs)\).
In particular every nonempty subset of \(\Ssig\) has a \(\leq\)-minimal element.
\end{prop}

\begin{proof}
By Proposition \ref{S_order_char}, \(\As  < \Bs\) implies \(\rho(\As) \leq \rho(\Bs)\).
Suppose for contradiction that there are \(\As < \Bs\) in \(\Ssig\) such that \(\rho(\As) = \rho(\Bs)\).
By replacing \(\As\) and \(\Bs\) if necessary, we may assume that \(\As\) is obtained by removing
a member of \(\Bs\) or by inflating \(\Bs\).
Moreover, since \(\As < \Bs\), \(\As\) is not an inflation of \(\Bs\).
Observe that the sets of ranks of \(\As\) and \(\Bs\) must be the same since \(\rho(\As) = \rho(\Bs)\).
On the other hand the sequence of ranks must be different since otherwise the corresponding summands
would be equivalent by Lemma \ref{R_order_char}, which would imply
\(\As \equiv \Bs\) by Lemma \ref{monotone}. 
Note however, that the manipulation which produced \(\As\) from \(\Bs\)
can change at most one entry in the sequence
of ranks and in particular cannot properly reorder them.
\end{proof}

\section{The embeddability relation on \(\Sgen\)-generated groups}
\label{main:sec}

In this section we give a proof of Theorem \ref{main_thm}.
If \(\xi < \epsilon_0\), then \(\Rs_\xi \in \Rsig\) is the
unique element with \(\rho(\Rs_\xi) = \xi\) and 
\(G_\xi :=\gen{\Rs_\xi}\);
these are just restatements of the definitions of \(\Rs_\xi\) and \(G_\xi\) given
in Section \ref{BasicOp:sec} which made only implicit reference to \(\rho\).
With these definitions, it is immediate that \(G_{\xi+1} \cong G_\xi + \Z\) and by
Proposition \ref{reduction_to_embedding} and Lemma \ref{rank_to_reduction}, \(G_\xi\) embeds into \(G_\eta\)
if \(\xi \leq \eta\).
We will see in Section \ref{R_order_char:subsec} that \(G_\xi\) embeds into \(G_\eta\)
only if \(\xi \leq \eta\). 

\subsection{EA-class calculations}
We now verify the EA-class calculations asserted in Theorem \ref{main_thm}. 
Observe that if \(\xi < \eta < \epsilon_0\), then \(\EA(G_\xi) \leq
\EA(G_\eta)\).
We first note the following lemma.
 
\begin{lem} \label{EA_calc_lem}
If \(A \in \Sgen\) and \(|A| > 1\), then
\[
\sup_{\Bs < \As} \EA(\gen{B}) \leq \EA(\gen{A}) \leq \big( \sup_{\Bs
< \As} \EA(\gen{B}) \big) + 2
\]
and \(\gen{A}\) has property \(\Sigma\).
In particular, if \(\sup_{\Bs < \As} \EA(\gen{B})\) is a limit ordinal, 
then \(\EA(\gen{A}) = \big( \sup_{\Bs < \As} \EA(\gen{A}) \big) +2\).
\end{lem}

\begin{proof}
Since \(\Bs < \As\) implies \(\gen{B}\) embeds into \(\gen{A}\),
the first inequality follows from the monotonicity of EA-class.
To see the second inequality, define for each
\(k \in \omega\) 
\[
B_k := \{a_i^{{a_{\max}}^p} \mid i < |A|-1 \textrm{ and } 0 \leq p
\leq 2k\}. 
\]
As noted in the beginning of Section \ref{inflation:sec}, if \(A_k\) is obtained from
\(A\) by iteratively inflating by \(a_{\max}\), then \(B_k \leq A_k^-\).  
In particular by Proposition \ref{decrease_rank},
\(\rho(\Bs_k) < \rho(\As)\).
By Proposition \ref{S_order_char}, \(\Bs_k < \As\).
Setting
\[
N:= \bigcup_{k=0}^\infty \gen{a_i^{a_{\max}^p} \mid i < |A|-1
\textrm{ and } -k \leq p \leq k } 
\]
we have that \(\gen{A}\) is an extension of \(N\) by \(\Z\) and
\(N\) is an increasing union of groups isomorphic to ones of the
form \(\gen{B_k}\). 
Thus 
\[
\EA(\gen{A}) \leq \EA(N) +1 \leq \big( \sup_k \EA(\gen{B_k}) \big) +
2 \leq \big( \sup_{\Bs < \As} \EA(\gen{B}) \big ) + 2 
\]
as desired.

The above argument in particular shows that every
\(\Sgen\)-generated group is in the smallest class 
that contains the abelian groups and is closed under the elementary
operations of 
\emph{extensions} and \emph{directed unions}. 
Since the class of groups which has property \(\Sigma\) includes
this class, it follows that every \(\Sgen\)-generated 
group has \(\Sigma\).
\end{proof}

\begin{lem}
If \(\xi = \omega^{(\omega^\alpha) \cdot ({2^n})}\) for
\(0 \le \alpha < \epsilon_0\) and \(n < \omega\),
then \(\EA(G_\xi) = \omega \cdot \alpha + n + 2\) if \(\alpha>0\)
and \(\EA(G_\xi) = n+1\) if \(\alpha=0\).
\end{lem}

\begin{proof}
Define 
\[
\Xi:= \{ \omega^{(\omega^\alpha) \cdot ({2^n})} \mid 0 < \alpha <
\epsilon_0 \textrm{ and }  n < \omega\} 
\]
and set \(\theta (\omega^{(\omega^\alpha) \cdot (2^n)} ) = \omega
\cdot \alpha + n + 2\). 
We verify by induction on \(\xi\)
that \(\EA(G_\xi) = \theta(\xi)\), which is what is asserted
in the first conclusion. 

The least element of \(\Xi\) is \(\omega^\omega\).
In this case, \(\Rs_{\omega^\omega} = E (\one + \one)\)
is the signature of a standard pair with oscillation 2 and
\(G_{\omega^\omega}\) is the Brin-Navas group, which has EA-class
\(\omega+2 = \theta(\omega^\omega)\). 
Next observe that if \(\omega^{(\omega^\alpha) \cdot (2^n)}\) is in
\(\Xi\), then the next element of \(\Xi\) 
is \(\omega^{(\omega^\alpha) \cdot (2^{n+1})}\).
In particular, \(\xi \in \Xi\) is a limit point of
\(\Xi\) precisely when \(\xi = \omega^{\omega^\alpha}\) 
for some \(\alpha > 1\).
If \(\xi \in \Xi\) is of the form \(\omega^{(\omega^\alpha)
\cdot (2^{n+1})}\), 
then setting \(\xi' = \omega^{(\omega^\alpha) \cdot (2^n)}\) and \(\beta = \omega^\alpha \cdot 2^n\)
we have
\[
\Rs_\xi  = \exp(\Rs_\beta + \Rs_\beta) = \exp(\Rs_\beta) * \exp(\Rs_\beta) =  \Rs_{\xi'} * \Rs_{\xi'}.
\]
By Proposition \ref{wreath_prop}, \(G_\xi = G_{\xi'}
\ \wr\  G_{\xi'}\). 
By Proposition \ref{EAGWrG} and our induction hypothesis, 
\[
\EA(G_\xi) = \EA(G_{\xi'}) + 1 = \theta(\xi') +
1 = \theta (\xi). 
\]

Now assume that \(\xi \in \Xi\) is a limit point
of \(\Xi\). 
In what follows \(\xi'\) always represents an element of
\(\Xi\). 
We first claim that
\begin{equation}\label{ThetaContinuous}
\theta (\xi) = \big (\sup_{\xi' < \xi} \theta (\xi') \big )+ 2. 
\end{equation}
If \(\xi = \omega^{\omega^{\alpha + 1}}\), then
(\ref{ThetaContinuous}) follows from 
the fact that 
\(\xi = \sup_n \omega^{(\omega^\alpha) \cdot (2^n)}\) and
consequently that
\[
\begin{split}
\theta(\xi) = &\omega \cdot (\alpha + 1) + 2 = \omega \cdot \alpha +
\omega + 2 \\
= &\big ( \sup_n \omega \cdot \alpha + n \big) + 2 =  
 \big ( \sup_n \theta (\omega^{(\omega^\alpha) \cdot (2^n)}) \big ) + 2.
\end{split}
\]
If \(\xi = \omega^{\omega^\alpha}\) for a limit ordinal \(\alpha\),
then (\ref{ThetaContinuous}) follows 
from the continuity of the maps \(\alpha \mapsto \omega^{\omega^\alpha}\)
and \(\alpha \mapsto \omega \cdot \alpha\). 

Observe that in both cases \(\sup_{\xi' < \xi}
\theta(\xi') = \omega \cdot \alpha + \omega\) is a limit ordinal. 
Now observe that by Lemma \ref{EA_calc_lem} and our induction hypothesis
\[
\EA(G_\xi) = \big(\sup_{\xi' < \xi} \EA(G_{\xi'}) \big) + 2 =
\big (\sup_{\xi' < \xi} \theta 
(\xi') \big )+ 2 
= \theta (\xi).
\]
The first equality holds since \(\sup_{\xi' < \xi}
\EA(G_{\xi'}) = \sup_{\xi' < \xi} \theta (\xi')\) 
is a limit ordinal.

If \(\alpha=n=0\), then \(\xi=\omega\), \(\Rs_\omega=\exp(\one+\one)
= \one * \one\), \(G_\omega = \Z\wr\Z\), and \(EA(G_\omega)=1\).
The last conclusion follows from Lemma \ref{EA_calc_lem} and
arguments similar to those above.
\end{proof}

We have established the following proposition.

\begin{prop} \label{F-less}
For each \(\xi < \epsilon_0\),
\(G_\xi\) is elementary amenable and hence does not
contain a copy of Thompson's group \(F\).
\end{prop}

\begin{remark}
This shows in particular that the Brin-Sapir Conjecture holds within the class \(\Sfrak\).
Define a binary relation \(\strip\) on \(\Sgen\) by \(A \strip B\) if 
\(A\) is obtained from \(B\) by inflating it an arbitrary number of times with \(b_{\max}\)
and then removing \(b_{\max}\).
By Propositions \ref{S_order_char} and \ref{decrease_rank}, \(A \strip B\) implies \(\rho(A) < \rho(B)\).
Hence \((\Sgen,\strip)\) is well-founded.
Just this fact readily implies that the Brin-Sapir Conjecture holds for the class of
\(\Sgen\)-generated groups: by the arguments in the proof of Lemma \ref{EA_calc_lem}, there cannot be a
\(\strip\)-minimal \(A \in \Sgen\) which generates a counterexample to the Brin-Sapir Conjecture.
In fact this inductive approach to proving this conjecture is what led to the discovery of the
results in the present paper.
We believe that the correct approach to proving the Brin-Sapir Conjecture
is to prove a well-foundedness result for a suitable extension of this relation to a
broader class of generating sets. 
\end{remark}

\subsection{Reduction and embeddablity agree on \(\Rsig\)}\label{R_order_char:subsec}
We will now show that if \(\As\) and \(\Bs\) are in \(\Rsig\), then
\(\As \leq \Bs\) is equivalent to \(\gen{A} \emb \gen{B}\).
The forward implication has already been established in Proposition \ref{reduction_to_embedding}.
By Lemma \ref{R_order_char},
it suffices to show that if \(\As\) and \(\Bs\) are in \(\Rsig\)
and \(\rho(\Bs) < \rho(\As)\), then \(\gen{\As}\) does not embed into \(\gen{\Bs}\).
In order to do this, we will extend our analysis to
groups generated by decomposable elements of  \(\Rsig\) and \(\Ssig\).
The following lemma is used in this section and the next.

\begin{lem}\label{FullWidth}
Let \(H\) be generated by a finite geometrically fast 
system \(S\) of functions  with \(|S|>1\) and a \(\sqsubset\)-maximum element \(h\) such that
\(g \sqsubset h\) for every \(g \in S \setminus \{h\}\).
Let \(N{}\) be the normal closure of \(S\setminus\{h\}\) taken in \(H\).
Then the following hold:
\begin{enumerate}

\item \label{H_emb}
There is an embedding of \(H\) into \(\PLoI\) so that
the image of \(S\) has a \(\sqsubset\)-maximum element and so that every
element of \(H\setminus N{}\) has e-support identical to that of the image of \(h\).

\item \label{dont_commute}
For each \(f\in N{}\setminus\{1\}\) and each \(g\notin N{}\), we
have \([f,g]\ne1\).
Thus the center of \(H\) is trivial.

\item \label{central_cyclic}
The centralizer of every element of \(H\setminus N{}\) is
cyclic.
\end{enumerate}
\end{lem}

\begin{proof} 
In this proof, \(\gen2\) is the multiplicative
subgroup of \(\R\) generated by \(2\), and \(3\gen2\) is the coset
containing \(3\).  
Observe that by replacing \(S\) by a generating set with the same dynamical diagram,
we may assume that the support of \(h\) is dense.
Next we adjust the elements of \(S\)
slightly so that after the adjustment it still satisfies the hypotheses,
the isomorphism class of \(\gen{S}\) has not changed,
and each element of \(S\) is piecewise linear
with all slopes used by \(h\) coming from \(3\gen 2\) and all
slopes used by elements of \(S\setminus\{h\}\)
coming from \(\gen 2\).
We can do this by keeping all transition points the same
and changing each bump so that its graph is two affine pieces.  If
the pieces have slopes sufficiently close to 0 or \(+\infty\), as
appropriate, then the feet of each new bump are small enough to
be contained in the feet of the bump it replaces.   The
dynamical diagram is unchanged and the group generated
isomorphic to the original.

All elements of \(H\) are of the form \(f=uh^i\) with \(u\in N{}\) a
product of elements of \(C=\{g^{h^j} \mid g\in S\setminus\{h\}, j\in
\Z\}\).  It follows from the chain rule and the fact that \(h\) is
the identity on no open interval that every element of \(C\) has
slopes restricted to \(\langle 2\rangle\).  The element \(f=uh^i\)
is outside \(N{}\) if and only if \(i\ne0\).  Thus every element
\(f\) of \(H\setminus N{}\) has slope in \(3^i\langle 2\rangle\)
everywhere the slope of \(f\) is defined.  
Hence every \(f\in H\setminus N{}\) has e-support equal to
that of \(h\), proving (\ref{H_emb}).

With \(f\in N{}\setminus \{1\}\) and \(g\in H\setminus N{}\), the
only way to have \([f,g]=1\) is for the orbitals of \(f\) to be
among the orbitals of \(g\).  It then follows from Theorem 4.18 of
\cite{picric} that on each orbital \(J\) of \(g\), there would be a
piecewise linear bump \(b\) with support \(J\) with \(f|_J\) and \(g|_J\) each a
power of \(b|_J\).  This is not possible because 3 to a nonzero,
rational power is never equal to such a power of 2, proving the
first part of (\ref{dont_commute}).  The center of \(H\) is trivial because there is
no room for a nontrivial central element.

To prove (\ref{central_cyclic}), we add to the information in the
paragraph above.  For \(g\in H\setminus N{}\) and \([f,g]=1\), we
now know \(f\notin N{}\), and the e-supports of \(f\) and
\(g\) equal that of \(h\).  Thus the orbitals of \(f\) and \(g\) are
identical.  Again, Theorem 4.18 of \cite{picric} makes each of
\(f\) and \(g\) powers of a common root on each orbital.  By Theorem
4.15 of \cite{picric}, for each orbital \(J\) of \(g\) there is
a unique minimum root \(r_J\) so that all roots of \(f\) and \(g\)
on \(J\) are integral powers of \(r_J\).  For each orbital \(J\) of
\(g\), let \(m_J\) and \(n_J\) be the integers so that
\(f|_J=r_J^{m_J}\) and \(g|_J=r_J^{n_J}\).

Assume by way of contradiction that there are orbitals \(J\ne K\) of
\(f\) for which \(m_J/n_J \ne m_K/n_K\).  Now
\[
\begin{split}
\left(f|_J\right)^{n_J} &= r_J^{(m_J n_J)} = \left(g|_J\right)^{m_J},
\quad\mathrm{while}  \\
\left(f|_K\right)^{n_J} &= r_K^{(m_K n_J)} \ne r_K^{(m_J n_K)} =
\left(g|_K\right)^{m_J}.
\end{split}
\]

Thus \(a=f^{n_J}g^{-m_J}\) is the identity on \(J\) and not on
\(K\).  From (\ref{H_emb}) we have \(a\in N{}\).  But this
contradicts the fact that \(a\) commutes with \(g\notin N{}\) which
by (\ref{dont_commute}) means that \(a\notin N{}\).  Thus \(m_J/n_J
= m_K/n_K\) for all orbitals \(J\) and \(K\) of \(g\).  Thus \(f\)
is determined by \(g\) and by the restriction of \(f\) to one
particular orbital of \(g\).  By Theorem 5.5 of \cite{picric}, if
\(J\) is an orbital of \(g\), then the restriction of the
centralizer of \(g\) to \(J\) is cyclic.  Thus the centralizer of
\(g\) in \(H\) is cyclic.  \end{proof}

We've already noted that Proposition \ref{reduction_to_embedding} and Lemma \ref{rank_to_reduction} imply
that if \(\As \leq \Bs\) are in \(\Rsig\), then \(\gen{\As}\) embeds into 
\(\gen{\Bs}\).
The next lemma provides the converse alluded to earlier.

\begin{lem}\label{RsigEmbedOrder}
If \(\As,\Bs \in \Rsig\) and \(\rho(\As) < \rho(\Bs)\), then
\(\gen{\Bs}\) does not embed into \(\gen{\As}\). 
\end{lem}

\begin{proof} Assume the lemma is false, and 
let \((\beta, \alpha)\) be the lexicographically minimal pair
for which there are \(\As\) and \(\Bs\) in \(\Rsig\) so that
\(\rho(\As)=\alpha<\beta=\rho(\Bs)\) and \(\gen\Bs\) embeds in
\(\gen\As\).  Since we know that \(\gen\Cs\) embeds in \(\gen\Bs\)
whenever \(\Cs\in \Rsig\) has \(\rho(\Cs)<\rho(\Bs)\), our choice of
\(\alpha\) and \(\beta\) must have \(\beta=\alpha+1\) and
\(\gen\Bs\) embeds in no \(\gen\Ds\) with \(\Ds\in \Rsig\) and
\(\rho(\Ds)<\rho(\As)\).

The result holds if \(\alpha\) is finite since then
\(\gen\As=\Z^\alpha\) and \(\gen\Bs=\Z^{\alpha+1}\) and an embedding
of \(\gen\Bs\) in \(\gen\As\) is not possible.  We thus assume
that \(\alpha\) is infinite with Cantor normal form 
\begin{equation}\label{AlphaForm}
\alpha=\omega^{\alpha_0}+\omega^{\alpha_1}+\cdots +
\omega^{\alpha_k} + n
\end{equation}
with \(\alpha_0\ge \alpha_1 \ge \cdots \ge \alpha_k>0\) and
\(n\ge0\).

We thus have
\[
\begin{split}
\As &= \As_0+\As_1+\cdots +\As_k+\one\cdot n,
\qquad\mathrm{and} \\
\Bs &= \As_0+\As_1+\cdots +\As_k+\one\cdot(n+1),
\end{split}
\]
where, for \(0\le i\le k\), each \(\As_i\) is in \(\Rsig\) and
indecomposable with \(|\As_i|>1\) and \(\rho(\As_i) =
\omega^{\alpha_i}\).

With \(G_i=\gen{\As_i}\), we have \(\gen\As\) representable as \(G_0
+ G_1 + \cdots + G_k + \Z^n\) and \(\gen\Bs\) as \(G_0+ G_1 + \cdots
+ G_k + \Z^{n+1}\).  We assume a homomorphic embedding
\(\phi:\gen\Bs \rightarrow \gen\As\).  For \(0\le i\le k\), we let
\(\pi_i:\gen\As \rightarrow G_i\) be the projection homomorphism,
and let \(\phi_i=\pi_i\circ \phi:\gen\Bs\rightarrow G_i\).

If for all \(i\) with \(0\le i\le k\), we have \(\phi_i(\Z^{n+1})\)
trivial, then \(\phi\) embeds \(\Z^{n+1}\) into \(\Z^n\) which is
not possible.  Thus for some \(i\) with \(0\le i\le k\) and some
element \(x\in \Z^{n+1}\), we have that \(y=\phi_i(x)\)
is a nonidentity element of \(G_i\).  Let \(A_i\in \Sgen\) have
signature \(\As_i\), let  \(h_i\) be the maximum
element of \(A_i\), and let \(N_i\) be the normal closure in
\(G_i=\gen{A_i}\) of \(A_i\setminus\{h_i\}\).

We note that for all \(g\in \Bs\), we have \([g,x]=1\).  Thus
\([\phi_i(g), y]=1\) in \(G_i\).
If \(y\notin N_i\), then by Lemma
\ref{FullWidth}(\ref{central_cyclic}) the centralizer \(C_i\) of
\(y\) in \(G_i\) is cyclic.  From this, we have \(\phi_i(g)\in
C_i\).  This puts the image of \(\phi\) in
\[
\gen{\As_0} + \cdots \gen{\As_{i-1}} + C_i + \gen{\As_{i+1}} +
\cdots \gen{\As_k} + \Z^n
\]
which is isomorphic to the group generated by
\[
\Es = \As_0 + \cdots + \As_{i-1} + \As_{i+1} + \cdots + \As_k +
\one\cdot(n+1)
\]
which is in \(\Rsig\).  Since (\ref{AlphaForm}) is in normal form
and since \(\rho(\As_i)\geq \omega\), we have \(\rho(\Es)<\rho(\As) =
\alpha\).  This contradicts our initial choice of \(\alpha\).

If \(y\in N_i\), then by Lemma \ref{FullWidth}(\ref{dont_commute})
the centralizer \(C_i\) of \(y\) in \(G_i\) is contained in
\(N_i\), and \(\phi_i(g)\) is in \(N_i\).  This puts the image of
\(\phi\) in 
\[
\gen{\As_0} + \cdots \gen{\As_{i-1}} + N_i + \gen{\As_{i+1}} +
\cdots \gen{\As_k} + \Z^n.
\]

Since \(\gen\Bs\) is finitely generated, we can replace \(N_i\) in
the above by a suitable finitely generated subgroup.  There is an
\(l\) so that \(\phi_i(\gen\Bs)\) is contained in \(\gen{K_l}\)
where
\[
K_l := \{h^{{h_i}^j} \mid h\in \As_i,\, h<h_i,\, \mathrm{and}\,
-l\le j\le l\}.
\]
If \(K'_l\) is similarly defined with \(-l\le j\le l\) replaced
by \(0\le j\le 2l\), then \(\gen{K_l}\) is isomorphic to
\(\gen{K'_l}\).
But \(K'_l\le K''_l\) where \(K''_l\) is an
appropriately chosen iterated inflation of \(A_i\).
By Lemma
\ref{inflate_biembedd}, \(K''_l\) and thus also \(K'_l\) are in
\(\Sgen\).       
Since \(\As_i\) and \(\Ks_i''\) are both indecomposable and \(\As_i \equiv \Ks_i''\),
by Propositions \ref{S_order_char} and 
\ref{decrease_rank} we have that \(\rho(\Ks'_l)<\rho(\Ks''_l) = \rho(\As_i)\).

We now have that \(\gen{\Bs}\) embeds in 
\[
\gen{\As_0} + \cdots \gen{\As_{i-1}} + \gen{K'_l} + \gen{\As_{i+1}} +
\cdots \gen{\As_k} + \Z^n.
\]
Once again, since (\ref{AlphaForm}) is in normal form, we get that
\[
\rho(\As_0 + \cdots \As_{i-1} + \Ks'_l + \As_{i+1} +
\cdots \As_k + \one \cdot n)< \rho(\As)=\alpha
.
\]  Again, this contradicts our choice
of \(\alpha\).  This completes the proof.
\end{proof}

We are finally in a position to observe that the proof of Theorem
\ref{main_thm} has been completed. 
Recall that \(\Ssigp\) consists of all \(\As\) in \(\Ssig\) such that for some \(\Bs \in \Rsig\),
\(\As \equiv \Bs\).

\begin{prop}  \label{order_char_prop}
If \(\As\) and \(\Bs\) are in \(\Ssigp\), then the following are
equivalent: 
\begin{enumerate}

\item \label{reduction} \(\As \leq \Bs\)

\item \label{embedding} \(\gen{\As}\) embeds into \(\gen{\Bs}\)

\item \label{lower_rank} \(\rho(\As) \leq \rho(\Bs)\).

\end{enumerate}
Moreover, for each \(\As\) in \(\Ssig\), there is a unique element
\(\Bs\) of \(\Rsig\) such that 
\(\gen{\As}\) is biembeddable with \(\gen{\Bs}\).
\end{prop}

\begin{proof}
Let \(\As\) and \(\Bs\) be given elements of \(\Ssigp\) and set \(\alpha:= \rho(\As)\) and \(\beta: = \rho(\Bs)\).
By the definition of \(\rho\) on \(\Ssigp\) and the choice of the signatures \(\Rs_\xi\) \((\xi < \epsilon_0)\), we
have that \(\As \equiv \Rs_{\alpha}\) and \(\Bs \equiv \Rs_{\beta}\).
By Proposition \ref{reduction_to_embedding}, (\ref{reduction}) implies (\ref{embedding});
by the contrapositive of Proposition \ref{RsigEmbedOrder}, (\ref{embedding}) implies (\ref{lower_rank});
by Proposition \ref{R_order_char}, (\ref{lower_rank}) implies (\ref{reduction}).

Finally, if \(\As\) is any element of \(\Ssig\), then there is a reordering \(\widetilde{\As}\) of \(\As\) which is
in \(\Ssigp\).
We then have \(\gen{\As} \cong \gen{\widetilde{\As}}\) is biembeddable with
\(\Rs_{\alpha}\) where \(\alpha:= \rho(\widetilde{\As}) = \rho(\As)\).
Uniquess follows from Lemmas \ref{RhoIsIso} and \ref{RsigEmbedOrder}.
\end{proof}

\section{\((\Ffrak,\emb)\) is not linear}

\label{nonlinear:sec}

We conclude this paper by showing that the class of those subgroups of \(F\) 
admitting a finite geometrically fast generating set is not linearly
ordered by the embeddability relation.
Consider the groups \(B + \Z\) and \(G\) with
geometrically fast generating sets specified by
the dynamical diagrams in Figures \ref{BZ_fig} and \ref{G_fig} below.
We apply Lemma \ref{FullWidth} 
embedding \(B+\Z\) and \(G\) in \(\PLoI\)
so that \(\{a,b\}\) and \(\{f,g,h\}\)
satisfy (\ref{H_emb}) of the Lemma.
Notice that \(\gen{a,b}\) is 
the Brin-Navas group \(B\); 
\(c\) generates a copy of \(\Z\) which commutes with the elements of
\(B\). 
\begin{figure}[b]
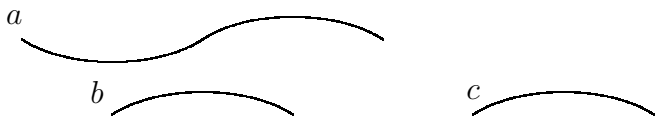

\[
\xy
(24,10); (48,10)**\crv{(30,6)&(42,6)}; (23,13)*{a};
(48,10); (72,10)**\crv{(54,14)&(66,14)};
(36,0); (60,0)**\crv{(42,4)&(54,4)}; (34,3)*{b};
(84,0); (108,0)**\crv{(90,4)&(102,4)}; (84,3)*{c};
\endxy
\]
\caption{A diagram for \(B + \Z\)} \label{BZ_fig}
\end{figure}
\begin{figure}[b] 
\[
\xy
(0,15); (24,15)**\crv{(6,11)&(18,11)}; (-1,18)*{f};
(24,15); (48,15)**\crv{(30,11)&(42,11)};
(48,15); (72,15)**\crv{(54,19)&(66,19)};
(72,15); (96,15)**\crv{(78,19)&(90,19)};
(36,0); (60,0)**\crv{(42,4)&(54,4)}; (34,2)*{h};
(12,0); (84,0)**\crv{(30,10)&(66,10)}; (10,2)*{g};
\endxy
\]
\caption{The group \(G = \gen{f,g,h}\)}\label{G_fig}
\end{figure}

\begin{thm}\label{NonEmbed} With \(G\) and \(B + \Z\) as above,
the following are true:
\begin{enumerate}

\item \label{EA-calc}
\(\EA(G)=\EA(B + \Z) = \omega+2\).

\item \label{BZ_nonemb}
There is no embedding of \(B + \Z\) into \(G\).

\item \label{nonemb_in_BZ}
There is no embedding of \(G\) into \(B + \Z\).

\end{enumerate}
\end{thm}

\begin{proof}
We begin by verifying (\ref{EA-calc}).
Define \(N_G\) to be the normal closure of \(\{g,h\}\) in \(G\) and
\(N_B\) to be the normal closure of \(b\) in \(B\).
The group \(N_G\) is generated by \(C=\{h^{f^i}, g^{f^i}\mid i\in \Z\}\).
If \(C_n = \{g^{f^i}, h^{f^i} \mid -n \leq i \leq n\}\),
then Proposition \ref{EAGWrG} implies \(\sup_n \EA(\gen{C_n}) = \omega\) and 
hence that \(\EA (N_G) = \omega+1\).
Since \(G\) is finitely generated and \(G/N_G\cong \Z\), we have
\(\EA(G)=\omega+2\).
Similarly, the group \(N_B + \Z\) is generated by \(D=\{b^{a^i}\mid
i\in \Z\} \cup \{c\}\). 
An analogous computation shows that \(\EA(N_B + \Z) = \omega+1\) and that
\(\EA(B + \Z) = \omega+2\).

Next we turn to (\ref{BZ_nonemb}).
Assume there is an embedding of \(B + \Z\) into \(G\).
The centralizer of \(c\) in \(B + \Z\) includes \(B\) and is not abelian. 
By (\ref{central_cyclic}) of Lemma \ref{FullWidth} the image of
\(c\) must be in \(N_G\). 
By (\ref{dont_commute}) of Lemma \ref{FullWidth} the image of every
element of \(B\) must be in 
\(N_G\) since otherwise that image would not commute with the image
of \(c\). 
Thus an embedding of \(B + \Z\) into \(G\) has its image in \(N_G\). 
But \(\EA(N_G)=\omega+1\) and \(\EA(B + \Z)=\omega+2\) so this is
not possible.

We will verify (\ref{nonemb_in_BZ}) after a series of claims.

\begin{claim}\label{GBZClaim}
If \(G\) embeds in \(B + \Z\), then
\(G\) embeds in \(B\).
\end{claim}

\begin{proof} If an embedding exists it can be composed with the
projection to \(B\).  
The kernel of this projection consists of all \((1,c^k)\) in \(B + \Z\).
These are all in the center of
\(B + \Z\) and if the image of the embedding intersected this kernel,
then the image would have nontrivial center. 
But by (\ref{dont_commute}) of Lemma \ref{FullWidth},
\(G\) has trivial center.
Thus the composition of the embedding with the projection is one-to-one.
\end{proof}

Next we introduce some tools used in \cite{ATaylor}.
Define the following predicates where the variables are intended to range over
elements of \(B + \Z\) and \(G\):
\[
\begin{split}
C(x,y) &:=\formula{[x,y]=1}, \\
D(x,y) &:= \formula{(\neg C(x,y)) \land C(x,x^y)}, \\
T(x,y,z) &:= \formula{D(x,y)\land D(x,z) \land D(y,z) \land C(x,y^z)}.
\end{split}
\]
We think of \(D(x,y)\) as saying that \(y\) ``dominates'' \(x\)
in that a typical pair that satisfies this is a fast pair of nested
one bump functions with the orbital of \(y\) as the larger of the two.
We think of \(T(x,y,z)\) as saying that \((x,y,z)\) forms
a ``tower'' in that a typical triple that satisfies \(T\) is a fast
triple of nested one bump functions with \(z\) the largest and \(x\)
the smallest.
While these are typical, they are not the only examples and we need
to know a little more about the functions that satisfy these
predicates.

\begin{claim}\label{DomClaim} If \(x,y \in F\) and 
\(D(x,y)\) holds, then:
\begin{enumerate}

\item \label{J_conclusion}
For each orbital \(J\) of \(x\) one of the following holds:
\begin{enumerate}

\item \label{J_disjoint}
\(J\) is disjoint from
all orbitals of \(y\),

\item \label{J_coincide}
\(J\) equals an orbital of \(y\) with
\([x |_J, y |_J]=1\),  or

\item \label{J_inside}
\(J y\) is disjoint from \(J\). 

\end{enumerate}

\item \label{orbital_conclusion}
There is an orbital \(J\) of \(x\) such that \(Jy\) is disjoint from \(J\).

\end{enumerate}

\end{claim}

\begin{proof}
Toward proving (\ref{J_conclusion}), fix an orbital \(J\) of \(x\).
Observe that \(Jy\) is an orbital of \(x^y\).
If \(Jy\) intersects \(J\) but differs from \(J\),
then \(x\) and \(x^y\) can't commute.
If \(Jy\) is disjoint from \(J\), then (\ref{J_inside}) holds.

Now suppose \(Jy = J\).
The chain rule implies that \(x\) and \(x^y\) agree at the endpoints of \(J\).
If \(x |_J \ne x^y |_J\), then Theorem 4.18 of \cite{picric} implies \(x^y |_J\) cannot commute with
\(x |_J\), contradicting \(D(x,y)\).
It follows that \(x |_J\) commutes with \(y |_J\).
Moreover either \(y |_J\) is the identity
or else \(J\) is an orbital of \(y\).
Thus if \(Jy = J\) then either (\ref{J_disjoint}) or (\ref{J_coincide})
hold.

Finally observe that if Conclusion (\ref{orbital_conclusion}) fails, then \([x,y]=1\) by
Conclusion (\ref{J_conclusion}).
This would contradict \(D(x,y)\).
\end{proof}

\begin{claim}\label{TowerClaim}
If \(x, y, z \in F\)
and \(T(x,y,z)\) holds, then if \(J\) is an orbital of \(x\) properly
contained in an orbital \(K\) of \(y\), then there is an orbital
\(L\) of \(z\) that properly contains \(K\). 
\end{claim}

\begin{proof} 
By Claim \ref{DomClaim} applied to \(x\) and \(y\),
we have that \(Jy\) is disjoint from \(J\).
In particular, \(x |_K\) does not commute with \(y |_K\).
By Claim \ref{DomClaim} applied to \(y\), \(z\), and the orbital \(K\), we have
that either \(K\) is an orbital of \(z\) and \(y^z |_K = y |_K\) or else
\(Kz\) is disjoint from \(K\).
The former is impossible since it implies \(\neg C(x |_K , y^z |_K)\) which is contrary to \(C(x,y^z)\).
It follows that \(Kz\) is disjoint from \(K\).
Thus any orbital \(L\) of \(z\) which intersects \(K\) contains all of \(K \cup Kz\)
and hence properly contains \(K\).
\end{proof}

\begin{claim}\label{NoDomAtTop}
Let \(x\) and \(y\) be in \(B \setminus N_B\).
Then \(D(x,y)\) is false. 
\end{claim}

\begin{proof}
Since \(\{a,b\}\) were chosen using (\ref{H_emb}) of Lemma \ref{FullWidth}, 
we know that the e-supports of \(x\) and \(y\) are connected
and identical.
Also note that \(x^y\) has connected e-support equal to that of \(x\).

Suppose for contradiction that \(D(x,y)\) holds:
\(x^y\ne x\) and \([x^y,x]=1\).
Since \(x\) and \(x^y\) commute and have the same connected e-support,
they have identical orbitals.
Thus \(y\) fixes all the transition points of \(x\).
This implies that the derivatives of \(x\) and \(x^y\) agree near the
ends of each orbital of \(x\). 
But commuting bumps on the same
orbital on which the slopes agree near the ends of the orbital must
be identical bumps by Theorem 4.18 of \cite{picric}.  Thus
\(x^y=x\) and \(D(x,y)\) cannot hold. 
\end{proof}

We return to the task of proving (\ref{nonemb_in_BZ}) of Theorem \ref{NonEmbed}.
By Claim \ref{GBZClaim} it suffices to show that \(G\) does not embed in \(B\). 
Suppose for contradiction that there is an embedding \(\phi : G \to B\).
Recall that \(G\) is generated by \(f < g < h\) as illustrated in Figure \ref{G_fig}.

\begin{claim}\label{GInBZLoc} 
The \(\phi\)-image of \(N_G\) is contained in \(N_B\).
\end{claim}

\begin{proof} 
Observe that if \(x\) is in \(B \setminus N_B\), then since \(N_B\) is normal we have
that \(x^z \not \in N_B\).
Thus for all \(z\in B\), \(D(x,x^z)\) is false by Claim \ref{NoDomAtTop}.
Since both \(D(h,h^f)\) and \(D(g,g^f)\) are true,
we have \(\phi(g) \in N_B\) and  \(\phi(h) \in N_B\).
The  conclusion of the claim now follows from the fact
that \(N_G\) is the normal closure of \(\{f,g\}\) in \(G\) and
\(N_B\) is normal in \(B\). 
\end{proof}

For each \(i\), define \(h_i := h^{f^{i}}\).
Observe that 
\[
\supt (h_0) \subsetneq \supt (h_1) \subsetneq \supt (h_2) \subsetneq \cdots \subsetneq \supt (g).
\]
Define \(A:=\{h_i\mid i\ge0\}\cup\{g\}\) and note that all elements
of \(A\) are in \(N_G\) and must have \(\phi\)-images in \(N_B\).

For any triple \(x < y < z\) from \(A\), we have \(T(x,y,z)\) and hence \(T(\phi(x),\phi(y),\phi(z))\).
It follows from (\ref{orbital_conclusion}) of Claim \ref{DomClaim}
and Claim \ref{TowerClaim} that there are intervals
\(
I_0\subsetneq I_1\subsetneq I_2 \subsetneq \cdots \subsetneq J
\)
where each \(I_i\) is an orbital of \(\phi(h_i)\) and \(J\) is an
orbital of \(\phi(g)\).

However the orbitals used in \(N_B\) come in families \(\Jscr_n\)
indexed over \(\Z\) with each orbital in family \(\Jscr_n\) contained in
an orbital in family \(\Jscr_{n+1}\).
Since orbital \(I_0\) of \(\phi(h_0)\)
must come from some \(\Jscr_m\) and orbital \(J\) of \(\phi(g)\) must come
from some \(\Jscr_n\) with \(m<n\), there are only finitely many
different orbitals available in \(N_B\) between \(I_0\) and \(J\).
This contradicts the assumption that there is an embedding of \(N_G\) into \(N_B\).
Since this was shown to follow from the existence of an embedding
of \(G\) into \(B+\Z\), we have completed our proof of (\ref{nonemb_in_BZ}).
\end{proof}

\begin{remark}
Theorem \ref{NonEmbed} does not provide a counterexample to Conjecture \ref{univ_conj}. 
If we let \(a=g_4\) and \(b=f_4\) be the generators of \(G_{\tau_4}\) as shown in
Figure \ref{Gt4Gt5_fig}, then the reader can check that \(A=\{a^2, b^{a^{-1}b^{-1}},
b^{ab}\}\) is fast and that after extraneous bumps are excised
from \(A\), the dynamical diagram for \(A\) is identical to that of
\(\{f,g, h\}\), the generating set for \(G\).
Thus \(G\) embeds in \(G_{\tau_4}\).
\end{remark}

\def\cprime{$'$} \def\cprime{$'$} \def\cprime{$'$} \def\cprime{$'$}

\end{document}